\RequirePackage{fix-cm}
\documentclass[smallextended]{svjour3}       

\smartqed  
\usepackage{graphicx}
\smartqed  
\usepackage{tikz,pgfplots,placeins,subfigure}
\usepackage{amsfonts,todonotes,paralist}
\usepackage{bm}
\usepackage{wrapfig}
\usepackage{mathtools}
\usepackage{xcolor}
\usepackage{txfonts}

\newcommand{\R}{\mathbb{R}}
\newcommand{\dd}{\textup{d}}

\newcommand{\dbl}{\left\llbracket}
\newcommand{\dbr}{\right\rrbracket}

\def\charf {\mbox{{\text 1}\kern-.30em {\text l}}}

\begin{document}
	\title{Self-Similar Radially Symmetric Solutions of the Relativistic Euler Equations  with Synge Energy}

	
	
	\titlerunning{Self-Similar Radially Symmetric Solutions of the Relativistic Euler Equations}        
	
	\author{Tommaso Ruggeri,  Ferdinand Thein, and Qinghua Xiao}

	\authorrunning{T. Ruggeri,  F. Thein, Q. Xiao} 

\institute{T. Ruggeri \at
              Department of Mathematics 	University of Bologna and Accademia Nazionale dei Lincei, Italy,\\
              \email{tommaso.ruggeri@unibo.it}       
           \and
           F. Thein \at
             Johannes Gutenberg - Universit\"at Mainz, Institute for Mathematics, Staudingerweg 9, D-55128 Mainz, Germany, \\
              \email{fthein@uni-mainz.de}  
              \and
              Q. Xiao \at
               Innovation Academy for Precision Measurement Science and Technology, Chinese Academy
	of Sciences, Wuhan 430071, China, \\
                \email{xiaoqh@apm.ac.cn}  
            }	
	\date{\today}

	\maketitle
  \begin{abstract}
	We consider self-similar, radially symmetric solutions  of the relativistic Euler equations with constitutive relations from relativistic kinetic theory, based on Synge energies for monatomic and its extension to diatomic gases. For the corresponding initial--boundary value problem, including the spherical piston problem, we prove existence and uniqueness of solutions valid for all values of the relativistic parameter 
				\(\gamma = mc^{2}/(k_{B}T)\), 	thus covering both the classical limit (\(\gamma \to \infty\)) and the ultra-relativistic regime (\(\gamma \to 0\)). 
				We further establish key structural properties of Synge energies, showing the strict negativity of the second derivative with respect to pressure at constant entropy and the monotone dependence of the characteristic velocity on~\(\gamma\). These results extend the classical theory of  self-similar flows to the relativistic framework with kinetic-theory-based constitutive equations. 
\end{abstract}

\section{Introduction}
The relativistic Euler system governs the motion of a non-dissipative fluid and arises from the conservation of particle number and energy–momentum in a relativistic framework.
It plays a central role in physics: in astrophysics and cosmology it models relativistic jets, supernova explosions, neutron star mergers, and the early universe; in high-energy physics it provides the basic continuum description of the quark–gluon plasma in relativistic heavy-ion collisions.
From the mathematical viewpoint, it represents a fundamental example of nonlinear hyperbolic PDEs, yet the theory is still far less developed than for the classical Euler equations.

Since the pioneering work of Taub~\cite{Taub}, a substantial literature has developed on the qualitative theory and numerical analysis of relativistic Euler equations, including shock waves, self-similar solutions, and the Riemann problem. In the general-relativistic setting, the coupling with Einstein’s field equations introduces additional challenges that, at present, can only be fully addressed numerically (see, e.g.,~\cite{Dumbser}). Numerical studies in the ultra-relativistic regime have recently been presented in~\cite{KLW,Kunik2024,RT2025}.

A crucial aspect of relativistic fluid dynamics is the closure relation, namely the constitutive equation linking pressure, energy density, and particle number density.
In contrast to classical fluid mechanics—where the ideal gas law provides a well-established model—formulating an accurate and general equation of state in the relativistic context is a highly nontrivial task. Most commonly used relations are valid only in asymptotic regimes: either the classical (non-relativistic) limit or the ultra-relativistic regime, where the rest mass of the particles becomes negligible or the temperature becomes extremely high.
As highlighted by Ruggeri, Xiao, and Zhao~\cite{Ruggeri-Xiao-Zhao-ARMA-2021}, this limitation constitutes one of the weakest points of the current theory of relativistic fluids.
To overcome it, insights from the mesoscopic scale—particularly relativistic kinetic theory—can be employed.

For rarefied gases, Synge~\cite{Synge} derived a constitutive relation based on the equilibrium Jüttner distribution, leading to a physically consistent expression involving ratios of modified Bessel functions of the second kind.
The Synge equation has the advantage of reducing to the classical ideal gas law in the low-temperature limit while approaching simplified ultra-relativistic laws as the temperature grows.
Pennisi and Ruggeri~\cite{Pennisi_Ruggeri} later extended this framework to polyatomic gases, producing Synge-type energies with explicit formulas in the diatomic case. Although physically consistent, these models are mathematically intricate, and only recently has the Riemann problem with Synge’s law been rigorously solved~\cite{Ruggeri-Xiao-Zhao-ARMA-2021}.

There is also a long tradition of studying self-similar solutions of the radial Euler equations, including the important \emph{spherical piston problem}, which describes the motion of a spherical boundary $r=R(t)$ driving a compressible fluid through a prescribed radial velocity condition.
Classical contributions are due to Guderley~\cite{Guderley-1942}, Taylor~\cite{Taylor-PRSA-1950}, and Neumann~\cite{Neumann-Collected works-1963}; see also Courant–Friedrichs~\cite{Courant-Friderichs} and the survey~\cite{Jessen-2011}.
The relativistic extension has been extensively investigated in connection with astrophysical phenomena such as supernova explosions and gamma-ray bursts.
Taub~\cite{Taub} derived the relativistic Rankine–Hugoniot conditions, while the Blandford–McKee solution~\cite{Blandford_McKee} describes ultra-relativistic spherical blast waves.
General treatments of relativistic piston-driven flows are found in Anile~\cite{anile} and in subsequent work on imploding relativistic shocks.
More recently, Lai~\cite{Lai2019} investigated self-similar solutions of the radial relativistic Euler equations with classical constitutive laws, while Peng–Lien~\cite{Peng-Lien-2012}, Ding–Li~\cite{Ding-Li-ZAMP-2013}, and others~\cite{Ha-QAM-2014,Lai2023,Lu-QAM-2024} studied the piston problem in various regimes.

Regarding the classical regimes, several further remarks are in order. In~\cite{Zhang-ARMA-1998}, self-similar solutions to the isentropic Euler system with axisymmetry were analyzed.
Thanks to the particular structure of the Euler equations, the self-similar variables $(u,v)$ satisfy an autonomous system after a suitable transformation, enabling a precise structural analysis.
The resulting self-similar solutions are continuous for general initial data with $u_0 \ge 0$, and of class $C^1$ in the special case $v_0=0$.
The method of~\cite{Zhang-ARMA-1998} was extended in~\cite{HL-BLMSS-2014} to the full Euler system with radial symmetry.
For the corresponding initial–boundary value problem, the authors constructed a global bounded continuous solution for $u_0 \ge 0$ and a discontinuous weak solution with a shock wave for $u_0<0$, providing a detailed analysis of their structure.
Further contributions on solutions with symmetry or similarity can be found in~\cite{Jenssen2018,Jenssen2023}.
Smooth self-similar solutions to the Euler equations have also been investigated in the preprint~\cite{Merle2019}.

The purpose of this work is to extend these studies to the relativistic Euler equations with Synge energies for monoatomic and diatomic gases under radial symmetry.
We establish existence and uniqueness theorems for self-similar solutions and prove structural properties of Synge’s energy—most notably, the strict negativity of the second derivative with respect to pressure at constant entropy, and the monotone dependence of the characteristic velocity on the relativistic temperature parameter.

The paper is organized as follows.
Section~\ref{sec:Euler_eqnes} introduces the relativistic Euler equations.
Section~\ref{sec:synge_eos} reviews the constitutive relations from kinetic theory, with emphasis on Synge energies.
Section~\ref{sec:Radiale} derives the differential system under radial symmetry and formulates the ODEs governing self-similar solutions.
The main results and their proofs in the general radially symmetric case are given in Section~\ref{sec:main_res_prf}.
The theorem and its proof for the spherical piston problem are presented in Section~\ref{sec:piston_res_prf}.

The Appendix collects technical estimates on Synge energies and Bessel
functions, including the proof that the second derivative of the Synge
energy with respect to pressure at constant entropy satisfies
$e_{pp}<0$ for all $\gamma>0$.

\section{Governing Equations}\label{sec:Euler_eqnes}

Relativistic flows are governed by the conservation of particle number and energy--momentum \cite{anile,Cercignani-Kremer,Groot-Leeuwen-Weert-1980,Taub,Synge}:
\begin{align}\label{Euler0}
\begin{split}
   &\partial_\alpha V^\alpha = 0 
   \quad \Longleftrightarrow \quad 
   \frac{1}{c}\partial_t V^0 + \partial_i V^i = 0, \\
   &\partial_\alpha T^{\alpha\beta} = 0 
   \quad \Longleftrightarrow \quad 
   \begin{cases}
      \dfrac{1}{c}\partial_t T^{0j} + \partial_i T^{ij} = 0, \\[1ex]
      \dfrac{1}{c}\partial_t T^{00} + \partial_i T^{0i} = 0.
   \end{cases}
\end{split}
\end{align}
Here $\alpha,\beta=0,\dots,d$ and $i,j=1,\dots,d$, where $d\in\{1,2,3\}$ is the space dimension. We write $x^\alpha \equiv (ct,x^j)$ with
\[
\partial_\alpha =\frac{\partial}{\partial x^\alpha} \equiv  \left(\frac{1}{c}\partial_t,\, \partial_{x^i}\right),
\]
$c$ denotes the light velocity while $(t,x^i)$ are time and space coordinates.
The particle flux and the energy--momentum tensor are
\begin{align}\label{def:energy_mom_tensor}
V^\alpha = \rho U^\alpha ,
\qquad
T^{\alpha\beta} = \frac{e+p}{c^2} U^\alpha  U^\beta - p g^{\alpha\beta},
\end{align}
where $\rho = nm$ is the  density, $m$ the rest mass and $n$ the particle number, $U^\alpha \equiv\Gamma(c,v^j)$ the four-velocity with Lorentz factor
\[
\Gamma = \frac{1}{\sqrt{1-\frac{v^2}{c^2}}},
\]
$p$ is the pressure, $g^{\alpha\beta}$ the Minkowski metric, i. e diagonal with $g^{00}=+1$ and the other terms in the diagonal equal to $-1$. The energy  reads
\[
e = \rho c^2 + \rho \varepsilon,
\]
where $\varepsilon$ is the specific internal energy. In space--time form, \eqref{Euler0} becomes (cf.~\cite{anile,Ruggeri2021b})
\begin{align}\label{Euleroplane}
\begin{split}
 & \partial_t\!\left(\frac{c\rho}{\sqrt{c^2-v^2}}\right) + \partial_i \left(\frac{c\rho  {v^i}}{\sqrt{c^2-v^2}}\right)=0, \\
  &  \partial_t\!\left(\frac{e+p}{c^2-v^2}v^j\right) + \partial_i\!\left(\frac{e+p}{c^2-v^2}v^i v^j + p\delta^{ij}\right) =0, \\
  & \partial_t\!\left(e+\frac{v^2}{c^2-v^2}(e+p)\right) + \partial_i \left(\frac{c^2(e+p)}{c^2-v^2} {v^i}\right)=0.
\end{split}
\end{align}

As a consequence of \eqref{Euleroplane}, smooth solutions also satisfy the entropy balance
\begin{align}\label{relEuler_sys:entropy}
\partial_\alpha (\rho \eta U^\alpha) =0, \quad \Leftrightarrow \quad 
 \partial_t\!\left(\frac{c\rho \eta}{\sqrt{c^2-v^2}}\right) + \partial_i\left(\frac{c\rho \eta {v^i}}{\sqrt{c^2-v^2}}\right)=0 ,
\end{align}
provided the Gibbs relation
\begin{align}\label{Gibbs}
T\,d\eta = d\varepsilon - \frac{p}{\rho^2}\,d\rho
\end{align}
holds. Thus, entropy $\eta$ is constant along particle paths:
\begin{align*}
  \partial_t \eta + v_j \partial^j \eta =0.
\end{align*}

   \section{Constitutive equations and Synge Energy for monatomic and diatomic gas} \label{sec:synge_eos}
The system \eqref{Euleroplane} consists of $d+2$ scalar differential equations for $d+3$ fields $(\rho,v^j,p,e)$.
Therefore, a constitutive equation is required in order to have as many equations as unknowns.
For example, one may prescribe $e \equiv e(\rho,p)$. 
However, it is usually more convenient to introduce the temperature $T$ as a parameter and to specify the so-called thermal and caloric equations of state:
\begin{equation*}
  e \equiv e(\rho,T), \qquad p \equiv p(\rho,T).
\end{equation*}
Starting with Taub in 1949 \cite{Taub}, many authors have associated the relativistic Euler equations with classical thermal and caloric equations of state, or with simplified constitutive relations in the ultra-relativistic regime.
As observed in \cite{Ruggeri-Xiao-Zhao-ARMA-2021}, to obtain more realistic equations in the relativistic regime, particularly for rarefied gases, it is necessary to justify this association at the mesoscopic scale using kinetic theory.

Within the relativistic kinetic framework, we have the Boltzmann--Chernikov equation \cite{Cercignani-Kremer}:
\begin{equation}\label{B-C}
  p^\alpha \partial_\alpha f = Q,
\end{equation}
where $f \equiv f(x^\alpha,p^\beta)$ is the distribution function, $p^\alpha$ is the four-momentum with the property $p^\alpha p_\alpha = m^2 c^2$, and $Q$ is the collision term. 
In this case, $V^\alpha$ and $T^{\alpha\beta}$ are given as moments of the distribution function:
\begin{align}\label{moments3}
  \begin{split}
    V^\alpha &= m c \int_{\R^3} f \, p^\alpha \, d \vec{P}, \qquad
    T^{\alpha \beta} = c \int_{\R^3} f \, p^\alpha p^\beta \, d \vec{P},
  \end{split}
\end{align}
where
\begin{equation*}
  d \vec{P} = \frac{dp^1 \, dp^2 \, dp^3}{p^0}.
\end{equation*}
It is well known that the equilibrium distribution is the so-called J\"uttner distribution:
\begin{equation}\label{fJ}
  f_J = \frac{1}{4 \pi m^3 c^3} \, \frac{n \gamma}{K_2(\gamma)} \,
  \exp\!\left(- \frac{\gamma}{m c^2} U_\beta p^\beta\right),
\end{equation}
where $\gamma$ is the dimensionless parameter, some time called \emph{relativistic temperature parameter} or \emph{coldness parameter} (see, e.g. \cite{Cercignani-Kremer})
\begin{equation}\label{coldness}
  \gamma = \frac{m c^2}{k_B T},
\end{equation}
with $k_B$ being the Boltzmann constant. 
This parameter measures the relativistic character of the fluid: the fluid is relativistic if $\gamma \ll 1$; the classical limit corresponds to $\gamma \to \infty$; and the ultra-relativistic limit corresponds to $\gamma \to 0$.

	Substituting \eqref{fJ} into \eqref{moments3} and considering \eqref{def:energy_mom_tensor}, we obtain the following expressions for pressure and energy:
\begin{align}\label{caloric}
  \begin{split}
    p &= \frac{\rho c^2}{\gamma} = \frac{k_B}{m} \rho T, \quad
    e = \frac{\rho  c^2}{K_2(\gamma)} \left(K_3(\gamma) - \frac{1}{\gamma} K_2(\gamma)\right),
  \end{split}
\end{align}
where $K_i(\gamma)$ with $i \geq 0$ denotes the modified Bessel function of the second kind.  
The expression for the energy in \eqref{caloric} is known as the Synge energy \cite{Synge,anile,Rezzolla}.  

Pennisi and Ruggeri \cite{Pennisi_Ruggeri}, drawing an analogy with the classical case first introduced by C. Borgnakke and P. S. Larsen \cite{Borgnakke1975}, proposed a generalized Boltzmann--Chernikov equation. This retains the form of \eqref{B-C} but extends the distribution function $f \equiv f(x^\alpha,p^\beta,\mathcal{I})$ to depend on an additional variable $\mathcal{I}$, which accounts for the energy due to a molecule's internal degrees of freedom. Instead of \eqref{moments3}, they considered the following moments:
\begin{align*}
  \begin{split}
    V^\alpha &= m c \int_{\R^3} \int_0^{+\infty} f \, p^\alpha \, \phi(\mathcal{I}) \, d \vec{P} \, d\mathcal{I}, \\
    T^{\alpha \beta} &= \frac{1}{mc} \int_{\R^3} \int_0^{+\infty} f \, ( mc^2 + \mathcal{I} ) \, p^\alpha p^\beta \, \phi(\mathcal{I}) \, d \vec{P} \, d\mathcal{I},
  \end{split}
\end{align*}
where the energy--momentum tensor is influenced by the internal modes.  
Here, $\phi(\mathcal{I})$ is the state density of the internal modes, meaning that $\phi(\mathcal{I}) \, d\mathcal{I}$ represents the number of internal states of a molecule with microscopic internal energy between $\mathcal{I}$ and $\mathcal{I}+d\mathcal{I}$.

Using the \emph{Maximum Entropy Principle} (MEP), the generalized J\"uttner equilibrium distribution for polyatomic gases was derived in \cite{Pennisi_Ruggeri} (for more details on the MEP, see \cite{Ruggeri2021b}, the historical survey \cite{Dreyer}, and the references therein):
\begin{equation}\label{5.2n}
  f_E = \frac{n \gamma}{A(\gamma) K_2(\gamma)} \frac{1}{4 \pi m^3 c^3}
  \exp\!\left(- \frac{\gamma}{mc^2} \left( 1 + \frac{\mathcal{I}}{m c^2} \right) U_\beta p^\beta \right),
\end{equation}
with $A(\gamma)$ given by
\begin{equation*}
  A(\gamma) = \frac{\gamma}{K_2(\gamma)} \int_0^{+\infty} \frac{K_2(\gamma^*)}{\gamma^*} \, \phi(\mathcal{I}) \, d\mathcal{I}, 
  \qquad \gamma^* = \gamma + \frac{\mathcal{I}}{k_B T}.
\end{equation*}
The pressure and energy for polyatomic gases, consistent with the distribution function \eqref{5.2n}, are given by \cite{Pennisi_Ruggeri}:
\begin{align}\label{3n}
  \begin{split}
    p &= \frac{\rho c^2}{\gamma} = \frac{k_B}{m} \rho T, \\
    e &= \frac{n m c^2}{A(\gamma) K_2(\gamma)} 
      \int_0^{+\infty} \left( K_3(\gamma^*) - \frac{1}{\gamma^*} K_2(\gamma^*) \right) \phi(\mathcal{I}) \, d\mathcal{I}.
  \end{split}
\end{align}
We note that the expression for pressure is the same for both monatomic and polyatomic gases, while \eqref{3n}$_2$ generalizes the Synge energy to polyatomic gases.

Formally, if we take the measure $\phi(\mathcal{I})$ to be the Dirac delta function, $\phi(\mathcal{I}) = \delta(\mathcal{I})$, all integrals reduce to those of the monatomic Synge theory. In particular, the caloric and thermal equations of state in \eqref{3n} reduce to those of a monatomic gas \eqref{caloric}.  
On the other hand, in the classical limit $\gamma \to \infty$, the macroscopic internal energy
\begin{equation*}
  \varepsilon = \frac{e}{\rho} - c^2
\end{equation*}
converges to that of a classical polyatomic gas:
\begin{equation*}
  \lim_{\gamma \to \infty} \varepsilon = \frac{D}{2}\frac{k_B}{m} T,
\end{equation*}
provided we choose the measure $\phi(\mathcal{I}) = \mathcal{I}^a$, with exponent $a=(D-5)/2$, where 
$D=3+f^i$ is the total number of molecular degrees of freedom, given by the three translational dimensions plus the internal degrees of freedom $f^i \geq 0$ associated with rotation and vibration.  
For monatomic gases, $D=3$ and $a=-1$.  
In \cite{PRc}, it was proven that equation \eqref{3n} converges to the Synge equation \eqref{caloric} as $a \to -1$.

The evaluation of the integral in \eqref{3n} is required
to determine the polyatomic energy. However, in \cite{Ruggeri-Xiao-Zhao-ARMA-2021} it was shown that, for diatomic gases ($D=5$, i.e., $a=0$), the integral admits an explicit expression, reducing \eqref{3n} to
\begin{align}
  \begin{split}
    p &= \frac{\rho c^2}{\gamma} = \frac{k_B}{m} \rho T, \qquad
    e_\text{diat} = \rho c^2 \left(\frac{3}{\gamma} + \frac{K_0(\gamma)}{K_1(\gamma)}\right). \label{caloricdia-1}
  \end{split}
\end{align}
As observed in \cite{Ruggeri-Xiao-Zhao-ARMA-2021}, the Synge energy for monatomic gases in \eqref{caloric} can be written in a form analogous to that for diatomic gases in \eqref{caloricdia-1}, namely
\begin{align}
  e_\text{mono} = \rho c^2 \left(\frac{3}{\gamma} + \frac{K_1(\gamma)}{K_2(\gamma)}\right). \label{caloricdia-2}
\end{align}
This allows monatomic and diatomic gases to be treated in a unified way, using the following constitutive equations:
\begin{equation}\label{syngeg}
  p=\frac{\rho c^2}{\gamma}, \quad
  e=p \, \gamma \, \Phi_i(\gamma), \quad
  \Phi_i(\gamma) = \frac{3}{\gamma}+ h_i(\gamma), \quad
  h_i(\gamma) = \frac{K_i(\gamma)}{K_{i+1}(\gamma)}, \, (i=0,1).
\end{equation}
For $i=1$, the gas is monatomic \eqref{caloricdia-2}, while for $i=0$ it is diatomic \eqref{caloricdia-1}.

	\subsection{Change of variables from $(\rho,T)$ to $(\eta,\gamma)$}
	Now it is convenient to take  $(p,\eta)$ as independent variables for constitutive equations. 
Let us now calculate
\begin{equation*}
	e_p = \left( \frac{\partial e}{\partial p}\right)_\eta, \qquad  
	e_{pp} = \left( \frac{\partial^2 e}{\partial p^2}\right)_\eta.
\end{equation*}
To this end, the Gibbs equation \eqref{Gibbs} can be rewritten as
\begin{equation}\label{GibbsR}
	T\, d\eta = \frac{1}{\rho}\, de - \frac{e+p}{\rho^2}\, d\rho.
\end{equation}
Therefore, for constant $\eta$, we have
\begin{equation}\label{de1}
	de = \frac{e+p}{\rho}\, d\rho.
\end{equation}
From \eqref{syngeg}$_1$, it follows that
\begin{equation}\label{dp1}
	dp = p\left(\frac{d\rho}{\rho}-\frac{d\gamma}{\gamma}\right).
\end{equation}
Differentiating $e$ from \eqref{syngeg}$_2$ and comparing it with \eqref{de1} for $\eta=\text{const.}$, we obtain
\begin{equation*}
	d\rho = \rho \, \gamma \Phi'_i \, d\gamma.
\end{equation*}
Inserting this identity into \eqref{dp1}, we find
\begin{equation}\label{dgamma1}
	p\,\frac{d\gamma}{\gamma} = \frac{1}{\gamma^2 \Phi'_i -1}\, dp,
	\qquad \Rightarrow \qquad 
	\frac{p}{\gamma}\,\gamma_p = \frac{1}{\gamma^2 \Phi'_i -1}.
\end{equation}
Then, from \eqref{syngeg}$_2$ and \eqref{dgamma1}, we obtain
\begin{equation}\label{ep_gamma_phi_rel}
	e_p = \frac{\gamma^2 \Phi'_i\big(\gamma\Phi_i+1\big)}{\gamma^2\Phi'_i -1}.
\end{equation}
Taking into account \eqref{syngeg}$_3$, we have from \cite{Ruggeri-Xiao-Zhao-ARMA-2021} that
\begin{equation}\label{ep2}
	e_p = 3+ \sigma_i(\gamma),
\end{equation}
with
\begin{equation}\label{sigmai2}
	\sigma_i(\gamma) = \frac{\gamma h_i+4}{g_i} + \gamma h_i + 1,
\end{equation}
and
\begin{equation}\label{gi}
	g_i = \gamma^2 h'_i - 4.
\end{equation}
Since $e_p$ is only a function of $\gamma$, using \eqref{ep2} and \eqref{dgamma1}, we deduce
\begin{equation}\label{pepp2}
	p\, e_{pp} = \frac{\gamma \,\sigma'_i(\gamma)}{g_i}.
\end{equation}
Similarly, from \eqref{syngeg}$_{1,2}$ we can calculate $d\rho$ and $dp$. Then, inserting into the Gibbs equation \eqref{GibbsR} and taking $p$ constant, we deduce
\[
\gamma_\eta = \left(\frac{\partial \gamma}{\partial \eta}\right)_p = \frac{m}{k_B}\frac{\gamma}{g_i},
\]
and from \eqref{ep2} we obtain
\begin{equation}\label{epeta}
	e_{p\eta} = \frac{m}{k_B}\frac{\gamma \sigma'_i}{g_i} 
	= \frac{m}{k_B} \, p\, e_{pp}.
\end{equation}
In particular, from \eqref{transform} we can express all $K_j(\gamma)$ as a combination of only two Bessel functions.  
Therefore, using \eqref{syngeg}$_4$, \eqref{deriva}, and \eqref{transform}, we obtain
\begin{equation}\label{17}
	h'_i(\gamma) = h_i^2 + \frac{(2i+1)h_i}{\gamma} - 1,
\end{equation}
and then from \eqref{gi}
\begin{equation}\label{gi4}
	g_i = \gamma^2 (h_i^2-1) + (2i+1)\gamma h_i - 4.
\end{equation}
After straightforward but lengthy algebraic manipulations, we deduce from \eqref{sigmai2} and \eqref{gi4} that
\begin{equation}\label{sigmap}
	\sigma'_i = q_i \left(1-\frac{1}{s_i}\right)
	- \frac{(\gamma h_i+4)\,\big(q_i(2\gamma h_i+2i+1)-2\gamma\big)}{s_i^2},
\end{equation}
with
\begin{equation*}
	q_i = h_i(\gamma h_i+2(i+1))-\gamma, 
	\qquad s_i = \gamma^2(1-h_i^2) - \gamma h_i(2i+1)+4.
\end{equation*}
Therefore, using \eqref{ep2}, \eqref{sigmai2}, \eqref{pepp2}, \eqref{17}, and \eqref{sigmap}, we can obtain explicit expressions for $e_p$ and $p e_{pp}$, which depend only on $\gamma$ and are non-dimensional quantities.  
In particular, for $i=1$ and $i=0$, we obtain
\begin{align*}
	e_{p|\text{mono}} &= 3 + \frac{-\gamma^2 + \gamma^3h_1^3 - \gamma^3h_1 + 4\gamma^2h_1^2}{-\gamma^2 + \gamma^2h_1^2 + 3\gamma h_1 - 4},\\
	e_{p|\text{diat}} &= 3 + \frac{-\gamma^2 + \gamma^3h_0^3 - \gamma^3h_0 + 2\gamma^2h_0^2 - 2\gamma h_0}{-\gamma^2 + \gamma^2 h_0^2 + \gamma h_0 - 4},
\end{align*}
and
\begin{align}
	p\, e_{pp|\text{mono}} &= \frac{\gamma^2 I_1(\gamma) I_2(\gamma) I_3(\gamma)}{\big[-\gamma^2 + \gamma h_1 (\gamma h_1 + 3) - 4\big]^3}, \label{epp-mon} \\
	p\, e_{pp|\text{dia}} &= \frac{\gamma f_{\text{dia}}}{\big[\gamma(\gamma(h_0^2 - 1) + h_0) - 4\big]^3}. \label{epp-dia}
\end{align}
Here,
\begin{align*}
	I_1(\gamma) &= -\gamma^2 + \gamma h_1 (\gamma h_1 + 2) - 8, \quad
	I_2(\gamma) = -\gamma + h_1 (\gamma h_1 + 4) - 1,\\
	I_3(\gamma) &= -\gamma + h_1 (\gamma h_1 + 4) + 1,
\end{align*}
and
\begin{align*}
	f_{\text{dia}} &= \gamma^5 \big(h_0^2 - 1\big)^3 + 4\gamma^4 h_0 \big(h_0^2 - 1\big)^2 
	- 32\gamma^2 h_0 \big(h_0^2 - 1\big) - 24\gamma h_0^2 \\
	&\quad + \gamma^3\big(-4 h_0^4 + 11h_0^2 - 7\big) + 16 h_0.
\end{align*}
These expressions are not new, as they were previously derived separately for monatomic and diatomic gases in \cite{Ruggeri-Xiao-Zhao-ARMA-2021}.  
However, this unified treatment now makes it possible to establish an important new inequality valid for both types of gases.
\subsection{Proof of $e_{pp}<0$}
In \cite{Ruggeri-Xiao-Zhao-ARMA-2021} it was proved that $e_p>3$ for any $\gamma>0$.  
Therefore, from \eqref{ep2} we deduce that $\sigma_i>0$.  
As a consequence of \eqref{pg} and \eqref{eppd-den}, we obtain for any $\gamma \geq 0$:
\begin{equation*}
    g_i < 0.
\end{equation*}
%

Moreover, \cite{Ruggeri-Xiao-Zhao-ARMA-2021} verified that the maximal
characteristic velocity of the Euler system in the rest frame, denoted by $\hat{\lambda}$, 
when expressed in units of the speed of light, is given by
\begin{equation}\label{LLLL}
    \tilde{\lambda} = \frac{\hat{\lambda}}{c} = \frac{1}{\sqrt{e_p}}
    = \frac{1}{\sqrt{3+\sigma_i}}.
\end{equation}
This quantity depends only on $\gamma$ and attains its maximum value in the ultra-relativistic limit, 
where $\tilde{\lambda}=1/\sqrt{3}$.  

To prove rigorously that $\tilde{\lambda}$ decreases monotonically with $\gamma$, 
it suffices to show from \eqref{LLLL} that $\sigma_i' > 0$.  
From \eqref{pepp2}, this is equivalent to proving that 
$p \, e_{pp}$ is negative.  
This condition, however, appears not to have been established in the literature for the Synge energy, 
and is in fact stronger than the inequality proved in \cite{Ruggeri-Xiao-Zhao-ARMA-2021}, namely
\begin{equation*}
    (e+p)\, e_{pp} - 2 e_p (e_p - 1) < 0.
\end{equation*}
The inequality $e_{pp}<0$ plays a crucial role in our two main theorems. 
Its rigorous proof is rather technical and will be presented in 
Appendix~\ref{Appendice1}.

	%
      \section{Radial symmetry and self-similarity.} \label{sec:Radiale}
For radially symmetric flows with velocity
\begin{equation*}
\mathbf{v}(t,\mathbf{x}) = u(t,r) \, \mathbf{n},
\quad r=|\mathbf{x}|, \quad \mathbf{n}=\frac{\mathbf{x}}{r}, \quad |\mathbf{n}| =1,    
\end{equation*}
system \eqref{Euleroplane} reduces to
\begin{align}\label{Euler31}
\begin{split}
\partial_t\!\left(\frac{c\rho}{\sqrt{c^2-u^2}}\right) + \partial_r\!\left(\frac{c\rho u}{\sqrt{c^2-u^2}}\right) &= -\frac{d-1}{r}\frac{c\rho u}{\sqrt{c^2-u^2}}, \\
\partial_t\!\left(\frac{e+p}{c^2-u^2}u\right) + \partial_r\!\left(\frac{e+p}{c^2-u^2}u^2 + p\right) &= -\frac{d-1}{r}\frac{e+p}{c^2-u^2}u^2, \\
\partial_t\!\left(e+\frac{u^2}{c^2-u^2}(e+p)\right) + \partial_r\!\left(\frac{c^2(e+p)u}{c^2-u^2}\right) &= -\frac{d-1}{r}\frac{c^2(e+p)u}{c^2-u^2}.
\end{split}
\end{align}
The corresponding entropy law \eqref{relEuler_sys:entropy} becomes:
\begin{align*}
\partial_t\!\left(\frac{c\rho \eta}{\sqrt{c^2-u^2}}\right) + \partial_r\!\left(\frac{c\rho \eta u}{\sqrt{c^2-u^2}}\right) = -\frac{d-1}{r}\frac{c\rho \eta u}{\sqrt{c^2-u^2}},
\end{align*}
and hence
\begin{align*}
0 = \partial_t \eta + u\partial_r \eta.
\end{align*}
If in addition the flow is self-similar, i.e. variables depend on $(t,r)$ only through $s=t/r$, then \eqref{Euler31} reduces to the ODE system
\begin{align}\label{systemODE}
\begin{split}
\frac{du}{ds} &= \frac{(d-1)u(c^2-u^2)(u-sc^2)}{g}, \\
\frac{dp}{ds} &= \frac{(d-1)uc^2(us-1)(e+p)}{g}, \\
-s(us-1)\frac{d\eta}{ds} &= 0, \\
g &= c^2 e_p (us-1)^2 - (u-sc^2)^2.
\end{split}
\end{align}
 Note that for relativistic flows $0\leq \xi \leq c$, where $\xi=1/s$. The derivation parallels \cite{Lai2019}, with adapted notation.

 \section{Main result and proof for the general radially symmetric case}\label{sec:main_res_prf}
\begin{theorem}\label{thm:main}
Let $(u_0(r),p_0(r),\eta_0(r))\in (-c,c)\times\R_{>0}\times\R_{>0}$ be an initial datum such that
\[
(u,p,\eta)(0,r) = (u_0,p_0,\eta_0),
\]
and considering the following boundary condition:
\[
 (e(p,\eta)\,u)(t,0)=0.
\]
Then the system \eqref{Euler31} admits a unique self-similar entropy solution.  
\begin{itemize}
    \item If $u_0 \in [0,c)$, the solution is continuous and free of shocks.  
    \item If $u_0 \in (-c,0)$, the solution contains a single shock followed by a constant state.  
\end{itemize}
\end{theorem}
%

The proof of Theorem \ref{thm:main} is presented in the sequel through the proofs of several lemmas.  
Subsection \ref{denom_prf} begins with an analysis of the denominator of \eqref{systemODE}.  
The argument is then divided into two cases: $s \in (0,1/c]$, treated in Subsection \ref{small_s_prf}, and $s > 1/c$.  
For $s > 1/c$, the structure of the solution depends crucially on the initial velocity $u_0$, and three possible cases for $u_0 \in (0,c)$ are distinguished in Subsection \ref{init_pos_u_prf}.  
Finally, for $u_0 \in (-c,0)$ shocks appear, and the proof is completed in Subsection \ref{init_neg_u_prf}.

   \subsection{Investigation of the denominator $g$}\label{denom_prf}
   For the proof of Thm. \ref{thm:main}, it is crucial to discuss the denominator $g$ given by $\eqref{systemODE}_4$ due to its occurrence in the ODEs of velocity and pressure in system \eqref{systemODE}.
   We first study the denominator $\eqref{systemODE}_4$ for $|u| < c$ and $s \in (0,1/c]$. We obtain
   \begin{align}
   	g &= c^2e_p(us - 1)^2 - (u - sc^2)^2 =  \frac{2}{3}c^2e_p(us - 1)^2 + \frac{1}{3}c^2e_p(us - 1)^2 - (u - sc^2)^2\notag\\
   	&\stackrel{\eqref{ep2}}{>} 
   	\frac{2}{3}c^2e_p(us - 1)^2 + c^2(us - 1)^2 - (u - sc^2)^2 = \frac{2}{3}c^2e_p(us - 1)^2 + \underbrace{(u^2 - c^2)}_{\stackrel{|u|< c}{<} 0}\underbrace{(c^2s^2 - 1)}_{\stackrel{cs \leq 1}{\leq} 0}\notag\\
   	&\geq \frac{2}{3}c^2e_p(us - 1)^2 \stackrel{\eqref{ep2}}{>} 2c^2(us - 1)^2 > 0.\label{ineq:g_pos_v1}
   \end{align}
   For the case $s > 1/c$, we rewrite $g$ as follows:
   \begin{align}\label{def:g_A}
   	\begin{split}
   	    g &= c^2e_p(us - 1)^2 - (u - sc^2)^2 = A\cdot B,\\
   	A &:= c\sqrt{e_p}(us - 1) - (u - sc^2) = (c\sqrt{e_p}s - 1)u - (c\sqrt{e_p} - sc^2),\\
   	B &:= c\sqrt{e_p}(us - 1) + (u - sc^2) = (c\sqrt{e_p}s + 1)u - (c\sqrt{e_p} + sc^2).
   	\end{split}
   \end{align}
   Further we have the relations
   \begin{align}\label{cond_g_1}
  \begin{split}
       	&A \leq 0\quad\Rightarrow\quad B < 0 \quad\Rightarrow\quad g \geq 0,\\
   	&B \geq 0\quad\Rightarrow\quad A > 0 \quad\Rightarrow\quad g \geq 0. 
  \end{split}
   \end{align}
   Thus, we can conclude that
   \begin{align*}
   	g < 0\quad\Leftrightarrow\quad A > 0\;\wedge\;B<0. 
   \end{align*}
   Additionally, we define functions
   \begin{align}\label{def:phi_A}
           	\varphi_A(s) = \frac{c\sqrt{e_p} - sc^2}{c\sqrt{e_p}s - 1}, \qquad 
   	\varphi_B(s) = \frac{c\sqrt{e_p} + sc^2}{c\sqrt{e_p}s + 1}. 
   \end{align}
   For $s > 1/c$, we have for both denominators
   \begin{align*}
   	c\sqrt{e_p}s + 1 > c\sqrt{e_p}s - 1 > c\sqrt{3}\frac{1}{c} - 1 > 0, 
   \end{align*}
   and therefore we can directly conclude
   \begin{align}
   	A \begin{cases} &< 0,\\ &= 0,\\ &> 0\end{cases}\quad\Leftrightarrow\quad u\begin{cases} &< \varphi_A,\\ &= \varphi_A,\\ &> \varphi_A\end{cases}
   	\quad\text{and}\quad
   	B \begin{cases} &< 0,\\ &= 0,\\ &> 0\end{cases}\quad\Leftrightarrow\quad u\begin{cases} &< \varphi_B,\\ &= \varphi_B,\\ &> \varphi_B\end{cases}.\label{AB_u_phi_relation}
   \end{align}
   A straightforward calculation exploits the relation between $\varphi_A$, $\varphi_B$, and $1/s$, that is,
   \begin{align}
   	\begin{cases}
   	\varphi_A - \dfrac{1}{s} &= \dfrac{1 - s^2c^2}{s(c\sqrt{e_p}s - 1)},\\[1em]
   	\varphi_B - \dfrac{1}{s} &= \dfrac{s^2c^2 - 1}{s(c\sqrt{e_p}s + 1)},
   	\end{cases}
   \Rightarrow\quad
   \begin{cases} 
   &\varphi_A(s) = \varphi_B(s) = c,\;s = 1/c,\\[0.5em] &\varphi_A(s) < \dfrac{1}{s} < \varphi_B(s),\;s\in(1/c,\infty)\end{cases}.
   	\label{phi_s_relation}
   \end{align}
   We hence can conclude that $g > 0$ for $s \in (0,1/c]$ and $g \leq 0$ may only occur for $s \in (1/c,\infty)$.
   The discussion of $g$ guides the way for the proof of the main theorem and we will split the proof into the cases for $s \in (0,1/c]$ and $s \in (1/c,\infty)$ due to \eqref{phi_s_relation}.
   \subsection{The case $s \in (0,1/c]$}\label{small_s_prf}
   \begin{lemma}\label{lem:no_zeros}
Let $(u_0, p_0, \eta_0)$ be an initial datum with $p_0 > 0$ and $0 < |u_0| < c$.  
Then \eqref{systemODE} admits a unique solution for $s \in (0,1/c]$, which furthermore satisfies
\[
p(s) > 0, \qquad |u(s)| < c \quad \text{for all } s \in (0,1/c].
\]
\end{lemma}

   \begin{proof}
   	Due to the differentiability of the RHS and because of \eqref{ineq:g_pos_v1}, for the initial datum satisfying $0 < |u_0| < c$ and $s < 1/c$, we have a unique local solution of the ODE system \eqref{systemODE}.
   	We further want to show that the solution satisfies $p > 0$ and $0 < |u| < c$ for $s \in (0,1/c]$.
   	This will be done using a contradiction argument.
   	
    Assume that there exists a point $ s^\ast\in (0,1/c]$ such that $|u(s)| < c$ for all $s\in(0, s^\ast)$ and $|u(s^\ast)| = c$.
   	Then,  from $\eqref{systemODE}_1$, we  integrate on $(s^\ast - \varepsilon,s^\ast)$ for $\varepsilon > 0$ to obtain
   	\begin{align*}
   		\frac{\dd u}{\dd s} &= \frac{(d-1)u(c^2 - u^2)(u - sc^2)}{g}\\
   		\Leftrightarrow\quad\frac{1}{c^2 - u^2}\frac{\dd u}{\dd s} &= \frac{(d-1)u(u - sc^2)}{g}\\
   		\Leftrightarrow\quad\int_{u(s^\ast-\varepsilon)}^{u(s^\ast)}\frac{\dd u}{c^2 - u^2} &= \int_{s^\ast - \varepsilon}^{s^\ast}\frac{(d-1)u(u - sc^2)}{g}\,\dd s.
   	\end{align*}
   	The LHS becomes infinite, whereas the RHS stays bounded and thus we have a contradiction.\\
   	In an analogous way, we assume that there exists a $s^\ast\in (0, 1/c]$ such that the  local solution satisfies $|u(s)| > 0$ for all $s\in (0,s^\ast)$ and   $u(s^\ast) = 0$.
   	Then, by integration on $(s^\ast - \varepsilon,s^\ast)$ for $\varepsilon > 0$, we obtain  from $\eqref{systemODE}_1$ that
   	\begin{align*}
   		\frac{\dd u}{\dd s} &= \frac{(d-1)u(c^2 - u^2)(u - sc^2)}{g}\\
   		\Leftrightarrow\quad\frac{1}{u}\frac{\dd u}{\dd s} &= \frac{(d-1)(c^2 - u^2)(u - sc^2)}{g}\\
   		\Leftrightarrow\quad\int_{u(s^\ast-\varepsilon)}^{u(s^\ast)}\frac{\dd u}{u} &= \int_{s^\ast - \varepsilon}^{s^\ast}\frac{(d-1)(c^2 - u^2)(u - sc^2)}{g}\,\dd s.            %
   	\end{align*}
   	Again, the LHS becomes infinite, whereas the RHS stays bounded and thus we have a contradiction.\\
   	With similar arguments, we show the positivity of the pressure. Let us assume that we have a local solution on $(0, 1/c]$, and assume there exists a $s^\ast \in (0,1/c]$ such that $p(s^\ast) = 0$.
   	Then we obtain by integration on $(s^\ast - \varepsilon,s^\ast)$ for $\varepsilon > 0$ from $\eqref{systemODE}_2$
   	\begin{align*}
   		\frac{\dd p}{\dd s} &= \frac{(d-1)uc^2(us - 1)(e + p)}{g}\\
   		\Leftrightarrow\quad\frac{1}{e + p}\frac{\dd p}{\dd s} &= \frac{(d-1)uc^2(us - 1)}{g}\\
   		\Leftrightarrow\quad\int_{p(s^\ast-\varepsilon)}^{p(s^\ast)}\frac{g\dd p}{e+p} &= \int_{s^\ast - \varepsilon}^{s^\ast}(d-1)uc^2(us - 1)\,\dd s.
   	\end{align*}
   Obviously, the integration on the RHS is finite. Now we estimate the integration on the LHS.
   Note that  for $|u| < c$ and $s \in (0, s^\ast]$, 
   \[
   (cs^\ast - 1)^2 < (us-1)^2 < (cs^\ast + 1)^2,
   \] 
	which together with \eqref{ineq:g_pos_v1} implies that
	\begin{align*}
		\frac{2}{3}c^2(cs^\ast - 1)^2e_p < \frac{2}{3}c^2e_p(us - 1)^2 < g < c^2e_p(us - 1)^2 < c^2(cs^\ast + 1)^2e_p
	\end{align*}
   Thus $g$ is bounded from above and below by $e_p$ with appropriate constants and we hence study the integrand on the LHS with $e_p$ replacing $g$.
   According to \eqref{syngeg},  \eqref{dgamma1}, \eqref{ep_gamma_phi_rel}, and \eqref{17}, one has $\dd p=p_{\gamma}\dd \gamma$ and
   \begin{equation}\label{e_pddgammap}
		\frac{e_p}{e + p}p_{\gamma} = \frac{\gamma^2 \Phi'_i(1+\gamma\Phi_i)}{\gamma^2 \Phi'_i -1}	\frac{\gamma^2 \Phi'_i -1}{\gamma(1+\gamma\Phi_i)}
		= \gamma \Phi'_i = \gamma \left(h^2_i - 1\right) + (2i + 1) h_i - \frac{3}{\gamma}.
   \end{equation}
   	Since $\gamma \to \infty$ as $p \to 0$ the integrand on the left basically behaves as $-1/\gamma$ and we refer to \eqref{asymptotic_1} for a detailed asymptotic expansion.
   	This together with \eqref{e_pddgammap} implies that the integration on the LHS becomes minus infinity and thus leads to contradiction since the RHS stays bounded. This conclude the proof.
   \end{proof}
   Note that for $u_0 = 0$, the solution is constant in time and given by $(0,p_0,\eta_0)$. 
   \subsection{The case $s \in (1/c,\infty)$}\label{sec:inside_lightcone}
   The main outcome of the previous discussion for the case $s \in (0,1/c]$ is that $0 < |u| < c$ for $0 < |u_0| < c$.
   Thus we have
   \[
   u_0 \in \begin{cases} (0,c),\\ (-c,0)\end{cases}\quad\Leftrightarrow\quad u\left(\frac{1}{c}\right)\in\begin{cases} (0,c),\\ (-c,0)\end{cases}.
   \]
   The following lemma is important as it suggests to distinguish the cases of positive and negative initial velocity for the remainder of the main proof.
   \begin{lemma}\label{lem:A_neg}
	Assume $u > 0$ and $p > 0$ for $s \in (1/c,\infty)$.  
	Then $u < \varphi_A$, and hence $g > 0$ on $(1/c,\infty)$.
\end{lemma}
   For the sake of brevity, we provide the proof of Lem. \ref{lem:A_neg} in the appendix.
   Considering the previous result, we split the following part of the proof into two cases $u_0 \in (0,c)$ and $u_0 \in (-c,0)$.
   \subsection{The case $u_0 \in (0,c)$}\label{init_pos_u_prf}
   \begin{enumerate}[(a)]
   	\item Case I: Solution including vacuum for large $u_0$, i.e., close to $c$.
   	\item Case II: Solution becoming stationary for small $u_0$, i.e., close to zero.
   	\item Case III: Solution for intermediate region.
   \end{enumerate}
   This particular classification is due to the observation that for $u_0 = 0$, the solution is stationary. Moreover, it is shown in \cite{Ruggeri-Xiao-Zhao-ARMA-2021} that expansion into vacuum by rarefaction waves is possible if the difference of the initial velocities exceeds a certain threshold.\newline
   ~\\
   \underline{\textbf{Case I:}} Solution including vacuum for large $u_0$.
   \begin{lemma}\label{lem:case_1}
	Let $u_0 > \overline{u} > 0$, where $\overline{u}$ is given by \eqref{def_velo_bound}.  
	Then there exists $s^\ast > 1/c$ such that
	\[
	0 < u(s) < \varphi_A(s) \quad \text{for } s \in (1/c,s^\ast),
	\]
	and moreover
	\[
	u(s^\ast) = \varphi_A(s^\ast) = \frac{1}{s^\ast}, 
	\qquad p(s^\ast) = 0.
	\]
\end{lemma}
   \begin{proof}
    From the proof of Lemma \ref{lem:A_neg}, we have $0 < g$ and $0 < u(s) < \varphi_A(s)$ for $s\in(1/c,s^\ast)$. Due to \eqref{phi_s_relation}, we can conclude that $us - 1 < 0$ for $s\in(1/c,s^\ast)$. On the other hand, for $s > 1/c$, we have
    \[
      u - sc^2 < u - c < 0.
    \]
    Thus we can directly verify that $u(s)$ and $p(s)$ are monotonically decreasing for ${s\in(1/c,s^\ast)}$.
    Moreover, we can conclude from $u < \varphi_A$ that $A < 0$ and hence we yield
    \begin{align}
        0 > A = c\sqrt{e_p}(us - 1) - (u - sc^2)\quad\Leftrightarrow\quad c\sqrt{e_p} > \frac{u - sc^2}{us - 1}.\label{ineq:ep_u_case1}
    \end{align}
    From $\eqref{systemODE}_1$ and $\eqref{systemODE}_2$, we obtain
    \[
      \frac{1}{c^2(e + p)}p'(s) = \frac{us - 1}{(c^2 - u^2)(u - sc^2)}u'(s)
    \]
    and thus together with \eqref{ineq:ep_u_case1} we yield
    \begin{align}
    	\frac{\sqrt{e_p}}{c(e + p)}p'(s) = \frac{c\sqrt{e_p}(us - 1)}{(c^2 - u^2)(u - sc^2)}u'(s) < \frac{1}{c^2 - u^2}u'(s)\label{ineq:dp_du_vac}
    \end{align}
   due to $u'(s) < 0$. Let $\overline{s} \in (0,1/c)$ be defined by Lemma \ref{lem:u_eq_s} and hence
    \[
      u(s) - sc^2
      \begin{cases}
      &\geq0,\quad s \in (0,\overline{s}],\\
      &< 0,\quad s \in (\overline{s},\infty)
      \end{cases}.
    \]
    Further let $\sigma \in (\overline{s},s^\ast)$ and therefore, due to the monotonicity of the functions, $p'(s) < 0$ on $(0,s^\ast)$, $u'(s) > 0$ on  $(0,\overline{s})$ and $u'(s) < 0$ on $(\overline{s},s^\ast)$ we get
    \begin{align*}
        %
            %
            &\int_0^{s^\ast} \frac{\sqrt{e_p}}{c(e + p)}p'(s)\dd s <
            \int_0^\sigma \frac{\sqrt{e_p}}{c(e + p)}p'(s)\,\dd s < \int^\sigma_{\overline{s}} \frac{\sqrt{e_p}}{c(e + p)}p'(s)\,\dd s\\
            \stackrel{\eqref{ineq:dp_du_vac}}{<} &\int^\sigma_{\overline{s}} \frac{u'(s)}{c^2 - u^2}\,\dd s <  \int^\sigma_{\overline{s}} \frac{u'(s)}{c^2 - u^2}\,\dd s + \underbrace{\int_0^{\overline{s}} \frac{u'(s)}{c^2 - u^2}\,\dd s}_{>0} = \int^\sigma_0 \frac{u'(s)}{c^2 - u^2}\,\dd s.
            %
        %
    \end{align*}
    In particular we have
    \begin{align}
    	\int_{p_0}^{p(\sigma)} \frac{\sqrt{e_p}}{c(e + p)}\,\dd p <\int^{u(\sigma)}_{u_0} \frac{1}{c^2 - u^2}\,\dd u. \label{ineq:p_vac_int}
    \end{align}
    We want to comment on the integral
    \begin{align*}
        \mathcal{I}_p := \int_{0}^{p_0} \frac{\sqrt{e_p}}{c(e + p)}\dd p
    \end{align*}
    and will show that $\mathcal{I}_p$ is also bounded from above.
   Firstly, by the aid of \eqref{syngeg}, \eqref{dgamma1}, and \eqref{ep_gamma_phi_rel} we get
    \begin{align*}
        \frac{\sqrt{e_p}}{e + p} &= \frac{1}{p(\gamma\Phi_i(\gamma) + 1)}\left[\frac{\gamma^2\Phi'_i(\gamma)(\gamma\Phi_i(\gamma) + 1)}{\gamma^2\Phi'_i(\gamma) - 1}\right]^{\frac{1}{2}}\notag\\
        &= \frac{\gamma_p(\gamma^2\Phi'_i(\gamma) - 1)}{\gamma \sqrt{\gamma\Phi_i(\gamma) + 1}}\left[\frac{\gamma^2\Phi'_i(\gamma)}{\gamma^2\Phi'_i(\gamma) - 1}\right]^{\frac{1}{2}} = -\gamma_p\left[\frac{(\gamma^2\Phi'_i(\gamma) - 1)\Phi_i'(\gamma)}{\gamma\Phi_i(\gamma) + 1}\right]^{\frac{1}{2}}
    \end{align*}
    since $\gamma^2\Phi'_i(\gamma) - 1<0$ by \eqref{pg} and \eqref{eppd-den}.
    Thus we obtain with $\gamma_p < 0$ and \eqref{asymptotic_3} that
    \begin{align*}
        0 &< \mathcal{I}_p = \int_{0}^{p_0} \frac{\sqrt{e_p}}{c(e + p)}\dd p
        = -\frac{1}{c}\int_{0}^{p_0} \left[\frac{(\gamma^2\Phi'_i(\gamma) - 1)\Phi_i'(\gamma)}{\gamma\Phi_i(\gamma) + 1}\right]^{\frac{1}{2}}\gamma_p\dd p\notag\\
        &= \frac{1}{c}\int_{\gamma_0}^{\infty} \left[\frac{(\gamma^2\Phi'_i(\gamma) - 1)\Phi_i'(\gamma)}{\gamma\Phi_i(\gamma) + 1}\right]^{\frac{1}{2}}\dd\gamma< +\infty.
    \end{align*}
    Obviously, the bounded integral $\mathcal{I}_p$ increases w.r.t. $p_0$. Let us define the value $\bar{u}$ as follows
    \begin{align}
        \int_{0}^{\bar{u}} \frac{1}{c^2 - u^2}\dd u = \mathcal{I}_p\quad\Leftrightarrow\quad \bar{u} = c\frac{\exp(2\mathcal{I}_p) - 1}{\exp(2\mathcal{I}_p) + 1} \in (0,c).\label{def_velo_bound}
    \end{align}
    Since $u_0 > \overline{u}$, we  have
    \begin{align*}
        \int_0^{u_0}\frac{1}{c^2 - u^2}\dd u > \int_0^{\overline{u}}\frac{1}{c^2 - u^2}\dd u \stackrel{\eqref{def_velo_bound}}{=} \mathcal{I}_p \geq \int_{p(\sigma)}^{p_0} \frac{\sqrt{e_p}}{c(e + p)}\dd p
        \stackrel{\eqref{ineq:p_vac_int}}{>} \int_{u(\sigma)}^{u_0} \frac{1}{c^2 - u^2}\dd u.
    \end{align*}
    Clearly, the pressure must be zero prior to the velocity (both are decreasing functions), since otherwise the inequality would be violated and thus have $p(s^\ast) = 0$ and $u(s^\ast) > 0$. From $\eqref{systemODE}_2$, we conclude that $p'(s) = 0$ for $s \geq s^\ast$, i.e., the pressure remains zero.
    Due to the fact $e_p \to \infty$ as $p \to 0$, we further obtain
    \[
      \lim_{s\to s^\ast}\varphi_A(s) = \lim_{s\to s^\ast}\frac{c\sqrt{e_p} - sc^2}{c\sqrt{e_p}s - 1} = \lim_{s\to s^\ast}\dfrac{1 - \dfrac{sc}{\sqrt{e_p}}}{s - \dfrac{1}{c\sqrt{e_p}}} = \frac{1}{s^\ast}.
    \]
     Therefore, it remains to show that $u(s^\ast) = 1/s^\ast$. Since $e_p \to \infty$ as $p \to 0$ for $s \to s^\ast$, there exists an $\varepsilon > 0$ such that for all $s \in (s^\ast - \varepsilon,s^\ast)$, we have
    \[
      \frac{g}{c^2e_p} = (us - 1)^2 - \frac{(u - sc^2)^2}{c^2e_p} > \frac{1}{2}(us - 1)^2.
    \]
    We integrate $\eqref{systemODE}_2$ for $s \in (s^\ast - \varepsilon,s^\ast)$ along the solution, assuming that $u(s^\ast) < 1/s^\ast$, and thus yield
    \begin{align}
        \int_{p(s^\ast - \varepsilon)}^{p(s^\ast)} \frac{c^2e_p}{e + p}\dd p = \int_{s^\ast - \varepsilon}^{s^\ast} \dfrac{c^2(d - 1)u(us - 1)}{(us - 1)^2 - \dfrac{(u - sc^2)^2}{c^2e_p}}\dd s > \int_{s^\ast - \varepsilon}^{s^\ast} \frac{2c^2(d - 1)u}{us - 1}\dd s.\label{ineq:u_vac_contr}
    \end{align}
    For the left integral, we have with $p(s^\ast) = 0$
    \begin{align*}
        \int_{p(s^\ast - \varepsilon)}^{p(s^\ast)} \frac{c^2e_p}{e + p}\dd p = c^2\int_{\gamma(s^\ast -\varepsilon)}^\infty\gamma\Phi_i'(\gamma)\dd\gamma = -\infty.
    \end{align*}
    However this contradicts with the lower bound in \eqref{ineq:u_vac_contr} which has a finite value for $u(s^\ast) < 1/s^\ast$. Thus we have $u(s^\ast) = 1/s^\ast$.
   \end{proof}
   ~\\
   \underline{\textbf{Case II:}} Solution becoming stationary for small $u_0$.
   \begin{lemma}\label{lem:case_2}
   	For $u_0 > 0$ sufficiently small, there exists an $s^\ast$ such that
   	\[
   	0 < u(s) < \varphi_A(s) < \frac{1}{s},\;\, s\in(1/c,s^\ast),\quad u(s^\ast) = \varphi_A(s^\ast) = 0.
   	\]
   \end{lemma}
   \begin{proof} - 
  Since the right-hand side of the ODE system \eqref{systemODE} is smooth, the solution depends continuously on the initial datum for $s \in (0,1/c]$.  
We compare the two solutions at $s = 1/c$ corresponding to the initial data $(u_0,p_0,\eta_0)$ and $(0,p_0,\eta_0)$, respectively.  
Note that the latter yields a stationary solution, and that the entropy remains constant for smooth solutions.  
For this purpose, the velocity is scaled by $1/c$, while the remaining quantities are scaled by their corresponding initial values.  
We thus obtain for the scaled quantities
   \begin{align*}
       \left\|\begin{pmatrix} u(1/c)/c\\ p(1/c)/p_0\\ \eta(1/c)/\eta_0\end{pmatrix} - \begin{pmatrix} 0\\ p_0/p_0\\ \eta_0/\eta_0\end{pmatrix}\right\|
        \leq \tilde{C}\left\|\begin{pmatrix} u_0/c\\ 1\\ 1\end{pmatrix} - \begin{pmatrix} 0\\ 1\\ 1\end{pmatrix}\right\| = \tilde{C}\frac{u_0}{c}.
    \end{align*}
    Here $\tilde{C}$ is a constant depending on the Lipschitz constant of the RHS of the ODE system \eqref{systemODE}.
    Hence, in order for the solutions to be closer than any given $\varepsilon_c > 0$, we choose $u_0 < c\varepsilon_c/\tilde{C} =: \varepsilon_0$. Thus we further have $0 < u(1/c) < c\varepsilon_c$ and $0 < p_0 - p(1/c) < p_0\varepsilon_c$.\\
    For $s > 1/c$, we get
    \[
      \frac{us - 1}{u - c^2s} = \frac{1 - us}{c^2s - u} < \frac{1}{c}.
    \]
    Combining this with $\eqref{systemODE}_1$ and $\eqref{systemODE}_2$, we obtain that for $s > 1/c$,
    \begin{align*}
        \frac{\dd p}{\dd u} = \frac{c^2(e + p)(us - 1)}{(c^2 - u^2)(u - sc^2)} < c\frac{e + p}{c^2 - u^2}.
    \end{align*}
    Hence by integration we have
    \begin{align}
        &\int_{p(s)}^{p_0(1 - \varepsilon_c)}\frac{1}{c(e + p)}\dd p < \int_{p(s)}^{p(1/c)}\frac{1}{c(e + p)}\dd p\nonumber\\
        < &\int_{u(s)}^{u(1/c)}\frac{1}{c^2 - u^2}\dd u < \int_0^{u(1/c)}\frac{1}{c^2 - u^2}\dd u < \int_0^{c\varepsilon_c}\frac{1}{c^2 - u^2}\dd u.\label{ineq:int_p_u_case2}
    \end{align}
    Thus there exists a constant $\underline{p} > 0$ such that $0 < \underline{p} < p(s) < p(1/c)$ for all $s > 1/c$. In fact, assume no lower bound would exist, then the integral on the left will have an explicit lower bound.
    This contradicts with the smallness of $\varepsilon_c $. Therefore, we can conclude that there exists an $s^\ast$ such that
    \[
      \varphi_A(s^\ast) = \frac{c\sqrt{e_p} - s^\ast c^2}{c\sqrt{e_p}s^\ast - 1} = 0\quad\text{with}\quad s^\ast = \frac{1}{c}\sqrt{e_p(p(s^\ast),\eta_0)}.
    \]
    Since $0 < u(s) < \varphi_A(s)$ for $s \in (1/c,s^\ast)$, we have $u(s^\ast) = 0$. Note that the ODE becomes stationary then.
   \end{proof}
   ~\\
   \noindent
   \underline{\textbf{Case III:}} Solution for intermediate region.
   The result for \textbf{Case I} made a statement when the initial velocity is large enough, i.e. $u_0 > \bar{u}$. Whereas when the velocity is close to zero we yield \textbf{Case II}. By continuity we can conclude that for an initial datum with $u_0 \leq \overline{u}$, but large enough such that \textbf{Case II} is excluded, we obtain a smooth solution with
   \[
     0 < u(s) < \varphi_A(s) < \frac{1}{s}\quad\text{and}\quad 0 < p(s)\quad\text{for}\quad s\in \left(\frac{1}{c},\infty\right).
   \]
   Detailed arguments are given in the Appendix \ref{app:aux_res_u_pos}.
   \subsection{The case $u_0 \in (-c,0)$.}\label{init_neg_u_prf}
   The next lemma provides some information as $g$ becomes zero and thus the solution would blow-up. Therefore we investigate the function $\varphi_A(s)$ in order to make statements about $\varphi_A(\hat{s}) = 0$ and $\varphi_A(\bar{s}) = u(\bar{s})$.
   \begin{lemma}\label{lem:blowup}
	Let $u_0 \in (-c,0)$, and let $u(s)$ be the solution of $\eqref{systemODE}_1$. Then:
	\begin{enumerate}[(i)]
		\item There exists $\hat{s} > 1/c$ such that 
		\[
		u(\hat{s}) < \varphi_A(\hat{s}) = 0, 
		\qquad 
		u(s) < 0 < \varphi_A(s) \quad \text{for all } s \in (0,\hat{s}).
		\]
		\item There exists $\bar{s} \in (\hat{s}, \infty)$ such that 
		\[
		u(\bar{s}) = \varphi_A(\bar{s}) < 0, 
		\qquad 
		u_0 < u(s) < \varphi_A(s) \quad \text{for all } s \in (0,\bar{s}).
		\]
	\end{enumerate}
\end{lemma}
   \begin{proof}
   	For $s \in (0,1/c]$, we have due to the previous results that
   	\[
   	u_0 < u(s) < 0 < \varphi_A(s).
   	\]
   	Now let us consider $\varphi_A(s)$ given by \eqref{def:phi_A} for $s > 1/c$.    
    Its first derivative is given by 
   	\begin{align*}
   		\frac{\dd \varphi_A}{\dd s} = \dfrac{(d-1)c^3(s^2c^2 - 1)u(us - 1)e_{pp}(e + p) + 2c^2g\sqrt{e_p}(1 - e_p)}{2\sqrt{e_p}g(c\sqrt{e_p}s - 1)^2}.
   	\end{align*}
   	As $\varphi_A > u$, we have $g > 0$. Due to $e_{pp} < 0$, we further get for $s > 1/c$ and $u < 0$
   	\begin{align*}
   		\frac{\dd \varphi_A}{\dd s} &= \dfrac{(d-1)c^3(s^2c^2 - 1)u(us - 1)e_{pp}(e + p) + 2c^2g\sqrt{e_p}(1 - e_p)}{2\sqrt{e_p}g(c\sqrt{e_p}s - 1)^2}\\
   		&< \dfrac{2c^2g\sqrt{e_p}(1 - e_p)}{2\sqrt{e_p}g(c\sqrt{e_p}s - 1)^2} < 0.
   		%
   	\end{align*}
   	Thus $\varphi_A$ is monotonically decreasing. Now assume that there is no $\hat{s}$ such that $\varphi_A(\hat{s}) = 0$, i.e., $\varphi_A(s) > 0$ for all $s > 1/c$.
    We have, as long as $c\sqrt{e_p} - sc^2 > 0$,
    \begin{align*}
        \varphi_A(s) = \frac{c\sqrt{e_p} - sc^2}{c\sqrt{e_p}s - 1} < \frac{c\sqrt{e_p} - sc^2}{\sqrt{3} - 1} =: h(s).
    \end{align*}
    As we have $u(s) < 0 < \varphi_A(s)$, we conclude that $p'(s) > 0$ and hence
    \begin{align*}
        h'(s) = \frac{1}{\sqrt{3} - 1}\left(c\frac{e_{pp}}{2\sqrt{e_p}}p'(s) - c^2\right) < -\frac{c^2}{\sqrt{3} - 1}
    \end{align*}
    Thus there exists an point $ \hat{s}>1/c $ such that
    \begin{align*}
        0 < \varphi_A(\hat{s}) < h(\hat{s}) = 0 
    \end{align*}
    which contradicts the assumption.
    Thus there exist an $\hat{s} > 1/c$ such that $\varphi_A(\hat{s})=0$ and we need to show that $u(\hat{s}) < 0$.
    To proof this, we follow \cite{Lai2019} and use a comparison principle.
    Consider the initial value problem
    \begin{align*}
        \begin{cases}
            \dfrac{\dd \hat{u}}{\dd s} &= \dfrac{(d - 1)(\hat{u} - sc^2)\hat{u}(c^2 - \hat{u}^2)}{c^4\hat{s}^2(\hat{u}s - 1)^2 - (\hat{u} - sc^2)^2},\\
            \hat{u} &= u_0.
        \end{cases}
    \end{align*}
    This problem admits a solution for $s\in [0, \hat{s}]$ and in particular $\hat{u}(\hat{s}) < 0$ holds.
    Further we have for $s < \hat{s}$ that $u(s) < 0 < \varphi_A(s)$ and hence $p'(s) > 0$. With $e_{pp} < 0$ it thus follows
   	\[
   	e_p(p(\hat{s}),\eta_0) < e_p(p(s),\eta_0),\quad s < \hat{s}.
   	\]
   	Together with $\varphi_A(\hat{s}) = 0$ we therefore yield
   	\[
   	c^2\hat{s}^2 = e_p(p(\hat{s}),\eta_0) < e_p(p(s),\eta_0),\;s < \hat{s}.
   	\]
    This gives for the RHS of the ODEs
    \[
    \frac{(d-1)u(c^2 - u^2)(u - sc^2)}{c^2e_p(us - 1)^2 - (u - sc^2)^2} < \frac{(d-1)u(c^2 - u^2)(u - sc^2)}{c^4\hat{s}^2(us - 1)^2 - (u - sc^2)^2}
    \]
   	and therefore we conclude that $u(\hat{s}) < \hat{u}(\hat{s}) < 0$. The first statement is proven.\\
   	As long as $u < \varphi_A < 0$ for $s > \hat{s}$, we have $g > 0$ and thus $u'(s) > 0$.
    Now assume $\varphi_A$ and $u$ do not intersect, then there must exist a value $u_0 < \bar{u} < 0$ such that
    \[
    \lim_{s\to\infty}u(s) = \bar{u}.
    \]
    Due to $u'(s) >0$, $p'(s) > 0$ and $e_{pp} < 0$, we have for $s > 0$
    \[
    \frac{\dd u}{\dd s} = \frac{(d-1)u(c^2 - u^2)(u - sc^2)}{c^2e_p(p,\eta_0)(us - 1)^2 - (u - sc^2)^2} >-\frac{(d-1)s\bar{u}(c^2 - u_0^2)}{e_p(p_0,\eta_0)(u_0s - 1)^2}.
    \]
    Hence we obtain
    \begin{align*}
        \bar{u} - u_0 &= \int_0^{+\infty}\frac{(d-1)u(c^2 - u^2)(u - sc^2)}{c^2e_p(p,\eta_0)(us - 1)^2 - (u - sc^2)^2}\dd s\\
        &> -\int_0^{+\infty}\frac{(d-1)s\bar{u}(c^2 - u_0^2)}{e_p(p_0,\eta_0)(u_0s - 1)^2}\dd s = +\infty.
    \end{align*}
    Therefore we have a contradiction and can conclude that there exists an $\bar{s}$ with $u(\bar{s}) = \varphi_A(\bar{s}) < 0$.
   \end{proof}
  
In view of the previous result, it is necessary to include the shock wave as part of the solution for $s \in (1/c,\bar{s})$.  
For results on the one-dimensional relativistic Euler equations, we refer to \cite{Ruggeri-Xiao-Zhao-ARMA-2021,Chen1997}.  

The eigenvalues of system \eqref{Euler31} are
\begin{align*}
	\lambda_1 &= \frac{c\sqrt{e_p}\,u - c^2}{c\sqrt{e_p} - u}, 
	&\quad \lambda_2 &= u, 
	&\quad \lambda_3 &= \frac{c\sqrt{e_p}\,u + c^2}{c\sqrt{e_p} + u}.
\end{align*}
Note that $-\lambda_1(-u) = \lambda_3(u)$, and therefore only 3-shocks (or forward shocks) associated with $\lambda := \lambda_3$ will be studied.  
Further, since $r/t = \xi = 1/s$, the unperturbed state in front of (i.e., to the right of) the shock wave is given by the solution $(u(s),p(s),\eta(s))$ of \eqref{systemODE} with initial data $(u_0,p_0,\eta_0)$, where $u_0 \in (-c,0)$.  
This solution will be denoted by $(u(s),p(s),\eta(s))$.  
Typically, the unperturbed state (here on the right of the shock) is assumed to be given, and the Hugoniot curve is then obtained.  
The perturbed (left) state can subsequently be determined by solving the jump conditions.  
Quantities corresponding to the perturbed state will be denoted with a subscript $\delta$, i.e., $(u_\delta(s),p_\delta(s),\eta_\delta(s))$ and $\gamma_\delta(s)$.\footnote{Following Lai \cite{Lai2019}, we omit the subscript $1$ for the unperturbed state and replace the subscript $2$ with $\delta$ for the perturbed state.}  
Since the energy, and in particular $e_p$, are given functions of the state variables, the arguments will be written explicitly to emphasize whether a perturbed or unperturbed state is considered, e.g., $e_p(\gamma)$ versus $e_p(\gamma_\delta)$.  

The jump conditions for system \eqref{Euler31} are given by
  %
   	%
   	\begin{align}\label{jc:radsym_relEuler_sys}
   	\begin{split}
   	    	\sigma\dbl \frac{c\rho}{\sqrt{c^2 - u^2}}\dbr &=  \dbl\frac{c\rho u}{\sqrt{c^2 - u^2}}\dbr, \\
   		\sigma\dbl \frac{e + p}{c^2 - u^2}u\dbr &= \dbl \frac{e + p}{c^2 - u^2}u^2 + p\dbr, \\
   		\sigma\dbl e + \frac{u^2}{c^2 - u^2}(e + p)\dbr &= \dbl \frac{c^2(e + p)u}{c^2 - u^2}\dbr, 
   	\end{split}
   	\end{align}
   	%
   %
   also compare the results in \cite{Steinhardt1982,Chen1997,Lai2019,Ruggeri-Xiao-Zhao-ARMA-2021}. In order to keep the notation simple we use $u$ inside the jump brackets for the velocity on both sides of the shock. Once we are concerned with the detailed calculations later on, we will carefully distinguish between the perturbed and unperturbed state.
   For later use it is in particular helpful to rewrite the jump conditions in the rest frame with respect to the shock as follows.
   In the shock rest frame the shock velocity is zero and the transformed velocity $\tilde{u}(s)$ is given by
\begin{align}\label{eq:velo_trafo_m}
    \begin{split}
        \tilde{u} &= \mathcal{T}_\sigma^-(u,s)= \frac{c^2(u - \sigma)}{c^2 - \sigma u} = \frac{c^2(us - 1)}{c^2s - u},\\
    u &= \mathcal{T}_\sigma^+(\tilde{u},s) = \frac{c^2(\tilde{u} + \sigma)}{c^2 + \sigma \tilde{u}} = \frac{c^2(\tilde{u}s + 1)}{c^2s + \tilde{u}}. 
    \end{split}
\end{align}
We introduce the following quantities, following the work by Steinhardt \cite{Steinhardt1982},
	\begin{align}
	    H = e + p,\quad v = \frac{c\tilde{u}}{\sqrt{c^2 - \tilde{u}^2}},\quad\text{and}\quad\tilde\Gamma = \frac{c}{\sqrt{c^2 - \tilde{u}^2}}.\label{def:new_quantities}
	\end{align}
	We obtain
	%
		%
		\begin{align}\label{jc:radsym_relEuler_sys_v2}
		       \dbl\rho v\dbr =0, \quad 
		    \dbl \frac{1}{c^2}Hv^2 + p\dbr =0, \quad 
		    \dbl \tilde\Gamma Hv\dbr =0. 
		\end{align}
		%
	%
	In the following we validate that the unperturbed state, given by the solution of \eqref{Euler31}, satisfies the Lax condition and thus we have the prerequisites to apply \cite[Prop.\ 3 \& 4]{Ruggeri-Xiao-Zhao-ARMA-2021}. This implies the existence of a unique perturbed state solving the jump conditions \eqref{jc:radsym_relEuler_sys}.
  \begin{lemma}\label{lem:shock_wave}
	Let $(u(s),p(s),\eta(s))$ be a solution of \eqref{systemODE} with $u_0 \in (-c,0)$ and $s \in (1/c,\bar{s})$.  
	Then there exists a unique shock solution $(u_\delta(s),p_\delta(s),\eta_\delta(s))$ to \eqref{jc:radsym_relEuler_sys} with admissible speed $\sigma = 1/s$.
\end{lemma}
   \begin{proof}
   	Due to Lemma \ref{lem:blowup}, we have $u(s) < \varphi_A(s)$ for $s \in (1/c,\bar{s})$.
   	With \eqref{AB_u_phi_relation} and \eqref{cond_g_1}, this in particular gives $A < 0$ and $0 < g$.
   	From $A < 0$, we obtain
   	\begin{align*}
   		0 > A &= c\sqrt{e_p(\gamma(s))}(u(s)s - 1) - (u(s) - sc^2)\\
   		 &= s[\underbrace{c\sqrt{e_p(\gamma(s))}u(s) + c^2}_{> 0}) - (\underbrace{c\sqrt{e_p(\gamma(s))} + u(s)}_{>0}]\\
   		\Leftrightarrow\quad \lambda &=  \frac{c\sqrt{e_p(\gamma(s))}u(s) + c^2}{c\sqrt{e_p(\gamma(s))} + u(s)} < \frac{1}{s} = \sigma.
   	\end{align*}
   	Thus by applying \cite[Prop.\ 3 \& 4]{Ruggeri-Xiao-Zhao-ARMA-2021} we can conclude that there exists a unique perturbed state $(u_\delta(s),p_\delta(s),\eta_\delta(s))$ solving \eqref{jc:radsym_relEuler_sys}.
   \end{proof}
   Moreover, we can conclude from \cite[Prop.\ 4]{Ruggeri-Xiao-Zhao-ARMA-2021}, that  the Lax -- condition holds along the shock curve, i.e.
   \begin{align}
   	\lambda < \sigma < \lambda_\delta\label{Lax_cond}
   \end{align}
   for a shock wave with speed $\sigma = 1/s$.
   Further we have entropy-growth across the shock according to \cite[Cor.\ 1]{Ruggeri-Xiao-Zhao-ARMA-2021}, i.e.,
   \begin{align*}
   	\eta < \eta_\delta.
   \end{align*}
   The following lemma will provide more precise information about the shock solution.
  \begin{lemma}\label{lem:unique_shock}
	Let $(u(s),p(s),\eta(s))$ be the solution of \eqref{systemODE} with $u_0 \in (-c,0)$.  
	Further, let $(u_\delta(s),p_\delta(s),\eta_\delta(s))$ denote the uniquely defined perturbed state behind the shock, determined by \eqref{jc:radsym_relEuler_sys} and \eqref{Lax_cond} for $\sigma = 1/s$ and $s \in (1/c,\bar{s})$.  
	Then there exists a unique $s^\ast \in (1/c,\bar{s})$ such that $u(s^\ast) = 0$.
\end{lemma}

   \begin{proof}
The proof proceeds as follows:
\begin{enumerate}[Step 1:]
	\item Show that $u_\delta(s) > 0$ for some $s \in (1/c,\sqrt{3}/c)$.
	\item Verify that $u_\delta(s) \to u(\bar{s})$ as $s \to \bar{s}$, and hence $u_\delta(\bar{s}) < 0$.  
	Consequently, there exists $s^\ast \in (1/c,\bar{s})$ such that $u_\delta(s^\ast) = 0$.
	\item The uniqueness of $s^\ast$ is established by showing that for any $s^\ast$ with $u_\delta(s^\ast) = 0$, one has $u_\delta'(s^\ast) < 0$.
\end{enumerate}
\noindent\textbf{Step 1:}\quad For the perturbed state behind the shock, the Lax condition \eqref{Lax_cond} implies that
   	\begin{align*}
   		&\sigma < \lambda_\delta = \frac{c\sqrt{e_p(\gamma_\delta(s))}u_\delta(s) + c^2}{\underbrace{c\sqrt{e_p(\gamma_\delta(s))} + u_\delta(s)}_{>0}}\\
   		\Leftrightarrow\quad&\sigma(c\sqrt{e_p(\gamma_\delta(s))} + u_\delta(s)) < c\sqrt{e_p(\gamma_\delta(s))}u_\delta(s) + c^2\\
   		\Leftrightarrow\quad&\sigma c\sqrt{e_p(\gamma_\delta(s))} - c^2 < u_\delta(s)(\underbrace{c\sqrt{e_p(\gamma_\delta(s))} - \sigma}_{>0})\\
   		\Leftrightarrow\quad &\frac{\sigma c\sqrt{e_p(\gamma_\delta(s))} - c^2}{c\sqrt{e_p(\gamma_\delta(s))} - \sigma} < u_\delta(s).
   	\end{align*}
   	This in particular implies $A_\delta > 0$.
   	For $\sigma > c/\sqrt{3}\; \Leftrightarrow\; s < \sqrt{3}/c$, we obtain
   	\[
   	\sigma c\sqrt{e_p(\gamma_\delta(s))} - c^2 > \sigma c\sqrt{3} - c^2 > 0
   	\]
   	and thus $u_\delta(s) > 0$ for some $s \in (1/c,\sqrt{3}/c)$.\\
   	\newline
\noindent\textbf{Step 2:}\quad For $\bar{s}$, we have $u(\bar{s}) = \varphi_A(\bar{s})$ and by a simple manipulation this implies $\bar{\sigma} = \lambda(\bar{s})$.
      Due to the Lax condition we then have $\bar{\sigma} = \lambda_\delta(\bar{s})$, since otherwise we would yield
   	\[
   	\bar{\sigma} < \lambda_\delta(\bar{s})\quad\Rightarrow\quad \lambda(\bar{s}) < \bar{\sigma}.
   	\]
   	Let us define
   	\[
   	\lim_{s\to \bar{s}}u_\delta(s) = \bar{u}_\delta.
   	\]
   	From $\lambda_\delta(\bar{s}) =  \lambda(\bar{s})$ we compute
   	\begin{align*}
   		&~& \frac{c\sqrt{e_p(\gamma_\delta)}\bar{u}_\delta + c^2}{c\sqrt{e_p(\gamma_\delta)} + \bar{u}_\delta} &= \frac{c\sqrt{e_p(\gamma)}u + c^2}{c\sqrt{e_p(\gamma)} + u}\\
   		&\Leftrightarrow&\quad\left(c\sqrt{e_p(\gamma_\delta)}\bar{u} + c^2\right)\left(c\sqrt{e_p(\gamma)} + u\right) &= \left(c\sqrt{e_p(\gamma)}u + c^2\right)\left(c\sqrt{e_p(\gamma_\delta)} + \bar{u}_\delta\right)\\
   		&\Leftrightarrow&\quad c\left(\sqrt{e_p(\gamma)} - \sqrt{e_p(\gamma_\delta)}\right)\left(c^2 - u\bar{u}_\delta\right) &= c^2\left(\sqrt{e_p(\gamma)}\sqrt{e_p(\gamma_\delta)}(u - \bar{u}_\delta) - 1\right)(u - \bar{u}_\delta)\\
   		&\Leftrightarrow&\quad\frac{\sqrt{e_p(\gamma)} - \sqrt{e_p(\gamma_\delta)}}{\sqrt{e_p(\gamma)}\sqrt{e_p(\gamma_\delta)} - 1} &= c\frac{u - \bar{u}_\delta}{c^2 - u\bar{u}_\delta}.
   	\end{align*}
   	We clearly have $c^2 - u\bar{u}_\delta > 0$ and $\sqrt{e_p(\gamma)}\sqrt{e_p(\gamma_\delta)} - 1 > 0$. Now assume we have $\bar{u}_\delta > u$ as it holds for shocks, see \cite{Ruggeri-Xiao-Zhao-ARMA-2021}. This implies $\sqrt{e_p(\gamma)} - \sqrt{e_p(\gamma_\delta)} < 0$. With $e_{pp} < 0$ and $\dd p/\dd\gamma < 0$ we have
   	\[
   	  \frac{\dd}{\dd\gamma}e_p(\gamma) = e_{pp}(\gamma)\frac{\dd p}{\dd\gamma} > 0.
   	\]
   	Thus we yield the implication
   	\[
   	  \sqrt{e_p(\gamma)} - \sqrt{e_p(\gamma_\delta)} < 0\quad\Rightarrow\quad \gamma < \gamma_\delta.
   	\]
   	Considering \eqref{jc:radsym_relEuler_sys_v2}$_1$ together with $\rho = p\gamma/c^2$ we obtain
   	\begin{align}
   		0 = \dbl\rho v\dbr\quad\Leftrightarrow\quad\rho v = \rho_\delta v_\delta\quad\Leftrightarrow\quad\frac{p}{p_\delta} = \frac{\gamma_\delta}{\gamma}\frac{v_\delta}{v}.\label{eq:p_qout_mass}
   	\end{align}
   	Since we have $p_\delta > p > 0$ and $0 > v_\delta > v$ we conclude $\gamma > \gamma_\delta$. Thus we have a contradiction and hence it follows that
   	\begin{align*}
   		\lim_{s\to \bar{s}}u_\delta(s) = \bar{u}_\delta = u(\bar{s})
   	\end{align*}
   	Since $u_\delta(\bar{s}) = u(\bar{s})< 0$ and due to the result of the first step, there exists an $s^\ast$ with $s^\ast\in (1/c, \bar{s})$ such that $u_\delta(s^\ast) = 0$.\\
   	\newline
\noindent\textbf{Step 3:}\quad In the remaining step of the proof we will show the uniqueness of such a point $s^\ast$ with $u(s^\ast) = 0$. Therefore we will prove that $u'(s^\ast) < 0$ whenever $u(s^\ast) = 0$. To this end it is beneficial to study the jump conditions in the rest frame with respect to the shock given by \eqref{jc:radsym_relEuler_sys_v2}. Using \eqref{syngeg} we have the relations $\rho = p\gamma/c^2$ and $e = p\gamma\Phi_i(\gamma)$. This gives in particular
   	\begin{align}
   		H = p\chi_i(\gamma)\quad\text{with}\quad \chi_i(\gamma) := \gamma\Phi_i(\gamma) + 1.\label{def:chi}
   	\end{align}
With this we can obtain analogues to \eqref{eq:p_qout_mass}
\begin{align*}
    0 &= \dbl \frac{1}{c^2}Hv^2 + p\dbr
    \quad\Leftrightarrow&\quad\frac{p}{p_\delta} &= \frac{\chi_i(\gamma_\delta)v_\delta^2 + c^2}{\chi_i(\gamma)v^2 + c^2},\\
    0 &= \dbl \tilde\Gamma Hv\dbr\quad    
    \Leftrightarrow&\quad \frac{p}{p_\delta} &= \frac{\chi_i(\gamma_\delta)\tilde\Gamma_\delta v_\delta}{\chi_i(\gamma)\tilde\Gamma v}. 
\end{align*}
From these relations we can obtain the following non-linear system for the unknowns $\tilde{u}_\delta$ and $\gamma_\delta$, i.e.
\begin{align*}
    0 &= \frac{\chi_i(\gamma_\delta)}{\gamma_\delta}\tilde\Gamma_\delta - \frac{\chi_i(\gamma)}{\gamma}\tilde\Gamma =: G^{(1)}(\tilde{u},\gamma,\tilde{u}_\delta,\gamma_\delta), \\
    0 &= \frac{1}{\gamma_\delta v_\delta}(\chi_i(\gamma_\delta)v_\delta^2 + c^2) - \frac{1}{\gamma v}(\chi_i(\gamma)v^2 + c^2) =: G^{(2)}(\tilde{u},\gamma,\tilde{u}_\delta,\gamma_\delta). 
\end{align*}
We have the following derivatives
%
	%
	\begin{align}
	    \partial_\gamma\chi_i(\gamma) &= \Phi_i(\gamma) + \gamma\Phi_i'(\gamma) = h_i(\gamma)(\gamma h_i(\gamma) + 2(i + 1)) - \gamma = q_i(\gamma),\label{deriv:gamma_chi}\\
	    \partial_{\tilde{u}} v(\tilde{u}) &= \tilde\Gamma(\tilde{u})^3,\label{deriv:u_v}\\
	    \partial_{\tilde{u}} \tilde\Gamma(\tilde{u}) &= \frac{1}{c^2}v(\tilde{u})\tilde\Gamma(\tilde{u})^2,\nonumber\\
	    \partial_{\tilde{u}_\delta}G^{(1)}(\tilde{u},\gamma,\tilde{u}_\delta,\gamma_\delta) &= \frac{1}{c^2}\frac{\chi_i(\gamma_\delta)}{\gamma_\delta}v_\delta\tilde\Gamma_\delta^2,\label{deriv:u_G1}\\
	    \partial_{\tilde{u}_\delta}G^{(2)}(\tilde{u},\gamma,\tilde{u}_\delta,\gamma_\delta) &= \frac{\tilde\Gamma_\delta^3}{\gamma_\delta v_\delta^2}\left(\chi_i(\gamma_\delta)v^2_\delta - c^2\right),\nonumber\\
	    \partial_{\gamma_\delta}G^{(1)}(\tilde{u},\gamma,\tilde{u}_\delta,\gamma_\delta) &= \frac{\gamma_\delta q_i(\gamma_\delta) - \chi_i(\gamma_\delta)}{\gamma_\delta^2}\tilde\Gamma_\delta,\nonumber\\
	    \partial_{\gamma_\delta}G^{(2)}(\tilde{u},\gamma,\tilde{u}_\delta,\gamma_\delta) &= \frac{\gamma_\delta q_i(\gamma_\delta) - \chi_i(\gamma_\delta)}{\gamma_\delta^2}v_\delta - \frac{c^2}{\gamma_\delta^2 v_\delta}.\label{deriv:gamma_G2}   
	\end{align}
	%
%
Note that we have
\begin{align*}
	u(s) &< 0\quad\Rightarrow\quad\tilde{u}(s) < 0\quad\Rightarrow\quad v < 0, \\
	u_\delta(s) &< \sigma\quad\Rightarrow\quad\tilde{u}_\delta(s) < 0\quad\Rightarrow\quad v_\delta < 0 
\end{align*}
and hence we conclude $\partial_{\tilde{u}} v(\tilde{u}) > 0, \partial_{\tilde{u}} \tilde\Gamma(\tilde{u}) < 0$ for both  unperturbed and perturbed states, respectively. Additionally we directly have $\partial_{\tilde{u}_\delta}G^{(1)}(\tilde{u},\gamma,\tilde{u}_\delta,\gamma_\delta) < 0$  for $\sigma = 1/s^\ast$. 

For $i=0$ we have according to the results in \ref{app:eppd}
\begin{align*}
    \gamma q_0(\gamma) - \chi_0(\gamma) &= \gamma h_0(\gamma h_0 + 2) - \gamma^2 - \gamma \Phi_0 - 1\\
    &= \gamma^2h_0^2 + \gamma h_0 - \gamma^2 - 4 \stackrel{\eqref{eppd-den}}{<} 0.
\end{align*}
We can directly verify for $i=1$ that
\begin{align*}
    \gamma q_1(\gamma) - \chi_1(\gamma) &= \gamma h_1(\gamma h_1 + 4) - \gamma^2 - \gamma \Phi_1 - 1\\
    &= \gamma^2h_1^2 + 3\gamma h_1 - \gamma^2 - 4\\
    &= \gamma^2h_1^2 + 3\gamma h_1 - \gamma^2 - 3 - 1\stackrel{\text{\cite[Prop. 10, (21)]{Ruggeri-Xiao-Zhao-ARMA-2021}}}{<} -1.
\end{align*}
This leads to $\partial_{\gamma_\delta}G^{(1)}(\tilde{u},\gamma,\tilde{u}_\delta,\gamma_\delta) < 0$ and $\partial_{\gamma_\delta}G^{(2)}(\tilde{u},\gamma,\tilde{u}_\delta,\gamma_\delta) > 0$.
We have for the Jacobian
\begin{align*}
    \mathbf{D}_\delta G =
    \begin{pmatrix}
        \dfrac{1}{c^2}\dfrac{\chi_i(\gamma_\delta)}{\gamma_\delta}v_\delta\tilde\Gamma_\delta^2 & \dfrac{\gamma_\delta q_i(\gamma_\delta) - \chi_i(\gamma_\delta)}{\gamma_\delta^2}\tilde\Gamma_\delta\\
        \dfrac{\tilde\Gamma_\delta^3}{\gamma_\delta v_\delta^2}\left(\chi_i(\gamma_\delta)v^2_\delta - c^2\right) & \dfrac{\gamma_\delta q_i(\gamma_\delta) - \chi_i(\gamma_\delta)}{\gamma_\delta^2}v_\delta - \dfrac{c^2}{\gamma_\delta^2 v_\delta}
    \end{pmatrix}.
\end{align*}
A calculation given in Appendix \ref{app:shock_ineq} shows that
\begin{align}
    \det(\mathbf{D}_\delta G)(\tilde{u}_\delta,\gamma_\delta) &= \partial_{\tilde{u}_\delta}G^{(1)}\partial_{\gamma_\delta}G^{(2)} - \partial_{\gamma_\delta}G^{(1)}\partial_{\tilde{u}_\delta}G^{(2)} < 0.\label{jac_det}
\end{align}

Due to the symmetric structure of the jump conditions \eqref{jc:radsym_relEuler_sys_v2} we yield for the Jacobian with respect to the unperturbed state
\begin{align*}
    \mathbf{D}G = -
    \begin{pmatrix}
        \dfrac{1}{c^2}\dfrac{\chi_i(\gamma)}{\gamma}v\tilde\Gamma^2 & \dfrac{\gamma q_i(\gamma) - \chi_i(\gamma)}{\gamma^2}\tilde\Gamma\\
        \dfrac{\tilde\Gamma^3}{\gamma v^2}\left(\chi_i(\gamma)v^2 - c^2\right) & \dfrac{\gamma q_i(\gamma) - \chi_i(\gamma)}{\gamma^2}v - \dfrac{c^2}{\gamma^2 v}
    \end{pmatrix}.
\end{align*}
Here we conclude analogous to the perturbed Jacobian that $\partial_{\tilde{u}}G^{(1)}(\tilde{u},\gamma,\tilde{u}_\delta,\gamma_\delta) > 0$, $\partial_{\gamma}G^{(1)}(\tilde{u},\gamma,\tilde{u}_\delta,\gamma_\delta) > 0$ and $\partial_{\gamma}G^{(2)}(\tilde{u},\gamma,\tilde{u}_\delta,\gamma_\delta) < 0$. Due to \eqref{duG2_ineq_help} we have 
\[
  \chi_i(\gamma)v^2 - c^2 > 0
\]
and hence $\partial_{\tilde{u}}G^{(2)}(\tilde{u},\gamma,\tilde{u}_\delta,\gamma_\delta) < 0$.
Since the jump conditions have a unique solution for a given unperturbed state and shock velocity, we find a function $f \equiv (f^{(1)},f^{(2)})^T:\R^2\to\R^2$ such that
\[
  \begin{pmatrix}
      \tilde{u}_\delta\\ \gamma_\delta
  \end{pmatrix}
  = f(\tilde{u},\gamma)\quad\text{with}\quad G\left(\tilde{u},\gamma,f^{(1)}(\tilde{u},\gamma),f^{(2)}(\tilde{u},\gamma)\right) = 0.
\]
For the Jacobian of $f$ we have
\begin{align*}
    \mathbf{D}f = -\mathbf{D}_\delta G^{-1}\mathbf{D}G.
\end{align*}
In order to show the desired result $u'_\delta(s^\ast) < 0$, we want to remark that the ODE $\eqref{systemODE}_2$ for the pressure can be replaced with an ODE for $\gamma$ by the aid of \eqref{dgamma1}, i.e.,
\begin{align*}
	\frac{\dd\gamma}{\dd s} &= \frac{\dd\gamma}{\dd p}\frac{\dd p}{\dd s} = \gamma_p\frac{(d-1)uc^2(us - 1)(e + p)}{g}\notag\\
	&= \frac{\gamma}{\gamma^2\Phi_i'(\gamma) - 1}\frac{(d-1)uc^2(us - 1)\chi_i(\gamma)}{c^2e_p(\gamma)(us - 1)^2 - (u - sc^2)^2}. 
\end{align*}
From \cite[Prop.\ 5]{Ruggeri-Xiao-Zhao-ARMA-2021} we have
\[
  \frac{\dd u_\delta}{\dd p_\delta} > 0
\]
and hence the following relation holds
\begin{align*}
	\frac{\dd u_\delta}{\dd s} = \underbrace{\frac{\dd u_\delta}{\dd p_\delta}}_{> 0}\underbrace{\frac{\dd p_\delta}{\dd \gamma_\delta}}_{< 0}\frac{\dd\gamma_\delta}{\dd s}. 
\end{align*}
Thus, in order to have $ u'_\delta ( s^{\ast}) < 0$, we need to show that

\begin{align}
	\frac{\dd\gamma_\delta}{\dd s} = \frac{\partial f^{(2)}}{\partial \tilde{u}}\frac{\dd\tilde{u}}{\dd s} + \frac{\partial f^{(2)}}{\partial \gamma}\frac{\dd\gamma}{\dd s}>0\label{deriv:gamma_delta_s}
\end{align}
 for $\sigma = 1/s^\ast$. Due to $u(s) < 0$ and $g > 0$ for $s < \bar{s}$ it is directly verified, that $p'(s) > 0$ and therefore $\gamma'(s) < 0$.
The unperturbed velocity $u$ and its transform $\tilde{u}$ are related by the Lorentz transformation with respect to the shock rest frame, i.e.,
\begin{align}
	\frac{\dd \tilde{u}}{\dd s} &= \frac{\dd \mathcal{T}_\sigma^-(u(s),s)}{\dd s} = \frac{\partial \mathcal{T}_\sigma^-}{\partial u}\frac{\dd u}{\dd s} + \frac{\partial \mathcal{T}_\sigma^-}{\partial s}.\label{u_transform}
\end{align}
We have
\begin{align}
    \frac{\partial \mathcal{T}_\sigma^-}{\partial u} &= \frac{c^2(c^2 - \sigma^2)}{(c^2 - u\sigma)^2} > 0,\quad
    \frac{\partial \mathcal{T}_\sigma^-}{\partial s} = \frac{c^2(c^2 - u^2)}{(c^2 - u\sigma)^2} > 0.\label{deriv:trafo_m}
\end{align}
Thus we have due to $u'(s) > 0$ that $\tilde{u}'(s) > 0$.
For the partial derivatives we get
\begin{align}\label{deriv:df_du}
   \begin{split}
        \frac{\partial f^{(2)}}{\partial \tilde{u}} &= -\frac{1}{\det(\mathbf{D}_\delta G)}\left[\partial_{\tilde{u}_\delta}G^{(1)}\partial_{\tilde{u}}G^{(2)} - \partial_{\tilde{u}_\delta}G^{(2)}\partial_{\tilde{u}}G^{(1)}\right] > 0,\\
    \frac{\partial f^{(2)}}{\partial \gamma} &= -\frac{1}{\det(\mathbf{D}_\delta G)}\left[\partial_{\tilde{u}_\delta}G^{(1)}\partial_{\gamma}G^{(2)} - \partial_{\tilde{u}_\delta}G^{(2)}\partial_{\gamma}G^{(1)}\right] > 0. 
   \end{split}
\end{align}
Since we need to show that $\gamma_\delta'(s) > 0$ we start with \eqref{deriv:gamma_delta_s}, insert \eqref{u_transform} and then use \eqref{deriv:df_du} together with \eqref{deriv:trafo_m} to obtain
\begin{align*}
	\frac{\dd\gamma_\delta}{\dd s} &= \frac{\partial f^{(2)}}{\partial \tilde{u}}\frac{\dd\tilde{u}}{\dd s} + \frac{\partial f^{(2)}}{\partial \gamma}\frac{\dd\gamma}{\dd s}\\
	&= \frac{\partial f^{(2)}}{\partial \tilde{u}}\left(\frac{\partial \mathcal{T}_\sigma^-}{\partial u}\frac{\dd u}{\dd s} + \frac{\partial \mathcal{T}_\sigma^-}{\partial s}\right) + \frac{\partial f^{(2)}}{\partial \gamma}\frac{\dd\gamma}{\dd s}\\
	&> \frac{\partial f^{(2)}}{\partial \tilde{u}}\frac{\partial \mathcal{T}_\sigma^-}{\partial u}\frac{\dd u}{\dd s} + \frac{\partial f^{(2)}}{\partial \gamma}\frac{\dd\gamma}{\dd s}\\
	&= -\frac{c^2}{\det(\mathbf{D}_\delta G)}\frac{(d-1)u}{g}\left\{\left[\partial_{\tilde{u}_\delta}G^{(1)}\partial_{\tilde{u}}G^{(2)} - \partial_{\tilde{u}_\delta}G^{(2)}\partial_{\tilde{u}}G^{(1)}\right]\frac{(c^2s^2 - 1)}{(c^2s - u)^2}\left(c^2 - u^2\right)\left(u - sc^2\right)\right. \\
	&+ \left.\left[\partial_{\tilde{u}_\delta}G^{(1)}\partial_{\gamma}G^{(2)} - \partial_{\tilde{u}_\delta}G^{(2)}\partial_{\gamma}G^{(1)}\right]\frac{\gamma}{\gamma^2\Phi_i'(\gamma) - 1}(us - 1)\chi_i(\gamma)\right\}\\
	&= -\frac{c^2}{\det(\mathbf{D}_\delta G)}\frac{(d-1)u}{g}\left\{-\left[\partial_{\tilde{u}_\delta}G^{(1)}\partial_{\tilde{u}}G^{(2)} - \partial_{\tilde{u}_\delta}G^{(2)}\partial_{\tilde{u}}G^{(1)}\right]\frac{(c^2s^2 - 1)\left(c^2 - u^2\right)}{c^2s - u}\right. \\
	&+ \left.\left[\partial_{\tilde{u}_\delta}G^{(1)}\partial_{\gamma}G^{(2)} - \partial_{\tilde{u}_\delta}G^{(2)}\partial_{\gamma}G^{(1)}\right]\frac{\gamma}{\gamma^2\Phi_i'(\gamma) - 1}(us - 1)\chi_i(\gamma)\right\}.
\end{align*}
Since the pre--factor is negative we will further show, that the expression inside the brackets is also negative. For the following we want to remark the following observation
\begin{align}
	\gamma^2\Phi_i'(\gamma) - 1 = \gamma q_i(\gamma) - \chi_i(\gamma).\label{eq:deriv_chi_phi_rel}
\end{align}
For the treatment of the above expression it is beneficial to set
%
%
\begin{align}
	\Psi_I &:= -\left[\partial_{\tilde{u}_\delta}G^{(1)}\partial_{\tilde{u}}G^{(2)} - \partial_{\tilde{u}_\delta}G^{(2)}\partial_{\tilde{u}}G^{(1)}\right]\frac{(c^2s^2 - 1)\left(c^2 - u^2\right)}{c^2s - u}\notag\\
	&\hphantom{:}= \psi_{I,1} + \psi_{I,2},\label{def:psi_I}\\
	\psi_{I,1} &:= \frac{(c^2s^2 - 1)\left(c^2 - u^2\right)}{c^2s - u}\dfrac{1}{c^2}\dfrac{\chi_i(\gamma_\delta)}{\gamma_\delta}v_\delta\tilde\Gamma_\delta^2\dfrac{\tilde\Gamma^3}{\gamma v^2}\left(\chi_i(\gamma)v^2 - c^2\right)\notag,\\
	\psi_{I,2} &:= -\frac{(c^2s^2 - 1)\left(c^2 - u^2\right)}{c^2s - u}\dfrac{\tilde\Gamma_\delta^3}{\gamma_\delta v_\delta^2}\left(\chi_i(\gamma_\delta)v^2_\delta - c^2\right)\dfrac{1}{c^2}\dfrac{\chi_i(\gamma)}{\gamma}v\tilde\Gamma^2,\notag\\
	\Psi_{II} &:= \left[\partial_{\tilde{u}_\delta}G^{(1)}\partial_{\gamma}G^{(2)} - \partial_{\tilde{u}_\delta}G^{(2)}\partial_{\gamma}G^{(1)}\right]\frac{\gamma}{\gamma^2\Phi_i'(\gamma) - 1}(us - 1)\chi_i(\gamma)\notag\\
	&\hphantom{:}= \psi_{II,1} + \psi_{II,2},\notag\\
	\psi_{II,1} &:= -\frac{\gamma(us - 1)\chi_i(\gamma)}{\gamma q_i(\gamma) - \chi_i(\gamma)}\dfrac{1}{c^2}\dfrac{\chi_i(\gamma_\delta)}{\gamma_\delta}v_\delta\tilde\Gamma_\delta^2\left(\dfrac{\gamma q_i(\gamma) - \chi_i(\gamma)}{\gamma^2}v - \dfrac{c^2}{\gamma^2 v}\right),\notag\\
	\psi_{II,2} &:= \frac{(us - 1)\chi_i(\gamma)\tilde\Gamma}{\gamma}\dfrac{\tilde\Gamma_\delta^3}{\gamma_\delta v_\delta^2}\left(\chi_i(\gamma_\delta)v^2_\delta - c^2\right).\label{def:psi_II_2}
\end{align}
%
%
The detailed calculations in Appendix \ref{app:shock_ineq} give \eqref{psi_rel_fin}, i.e.,
\[
  \Psi_I + \Psi_{II} = \underbrace{\psi_{I,1} + \psi_{II,1}}_{< 0} + \underbrace{\psi_{I,2} + \psi_{II,2}}_{= 0} < 0.
\]
Thus collecting the results finally gives
\begin{align*}
	\frac{\dd\gamma_\delta}{\dd s} = \underbrace{-\frac{c^2}{\det(\mathbf{D}_\delta G)}\frac{(d-1)u}{g}}_{< 0}\underbrace{(\Psi_I + \Psi_{II})}_{< 0} > 0.
\end{align*}
Therefore we have $u_\delta'(s^\ast) < 0$ for $u_\delta(s^\ast)$ and by continuity the root is unique.
  \end{proof}
With the previous result on uniqueness for $u_0 \in (-c,0)$, the proof of Theorem \ref{thm:main} is complete, following the structure outlined above.  
Before discussing the piston problem, some further remarks and results concerning the shock solution are presented.  

\begin{remark}
	The following remarks are due:
	\begin{enumerate}[(i)]
		\item All previous results impose no particular restriction on $\gamma$, and are therefore valid for $\gamma \in (0,\infty)$.  
		Consequently, the present results also cover the limiting cases of the ultra-relativistic Euler equations for $\gamma \to 0$ and the classical  polytropic gas for $\gamma \to \infty$.
		\item Uniqueness of the shock solution has been established.  
		In \cite[Rem.~3.1]{Lai2019}, it is noted that this question remains open for the relativistic case.  
		In view of the previous remark, uniqueness is therefore also established for the polytropic case.
	\end{enumerate}
\end{remark}
The previous results concerning initial data with negative velocity are illustrated in Fig.~\ref{fig:shock_sol} and can be summarized as follows.  
For negative initial velocity, the ODE solution blows up at a point $\bar{s}$, see Lem.~\ref{lem:blowup}.  
By Lem.~\ref{lem:unique_shock}, this point corresponds to a characteristic shock, which is excluded here by the Lax condition.  
Hence, the shock occurs at some $s^\ast \in (\sqrt{3}/c, \bar{s})$.  

Furthermore, by Lem.~\ref{lem:shock_wave} and \cite[Prop.~4]{Ruggeri-Xiao-Zhao-ARMA-2021}, there exists a unique state that satisfies the jump conditions for a given $\sigma$ and the corresponding solution of \eqref{systemODE}.  
Thus, for a given $\sigma$, there is a unique $s^\ast$, a unique unperturbed state from \eqref{systemODE}, and a corresponding unique perturbed state behind the shock.  
Finally, Lem.~\ref{lem:unique_shock} guarantees the existence of a unique $s^\ast$ such that the velocity of the perturbed state vanishes, i.e., $u_\delta(s^\ast) = 0$.  
Consequently, for $u_0 \in (-c,0)$, the solution follows the ODE up to $s^\ast$ and is then continued by the constant state $(0,p_\delta(s^\ast),\eta_\delta(s^\ast))$ and $\gamma_\delta(s^\ast)$.
\begin{figure}
	\centering
    \includegraphics[width=\textwidth]{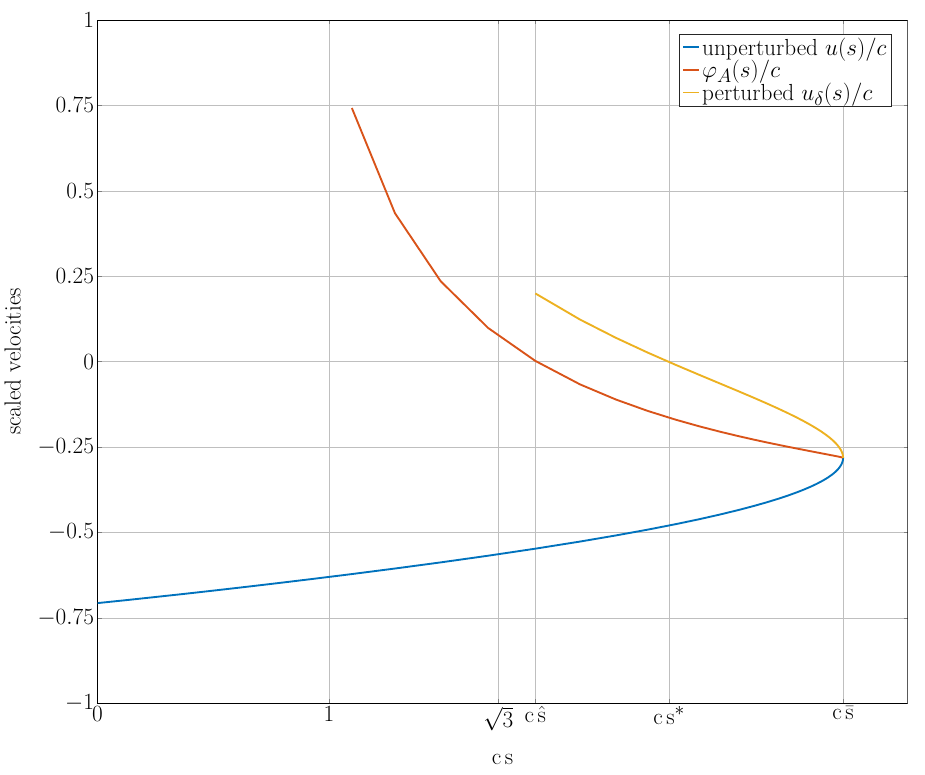}
    \caption{Visualization of the situation described in Lem.~\ref{lem:blowup} and Lem.~\ref{lem:unique_shock}.
    For an improved presentation we present the dimensionless quantities which are obtained with a proper scaling by the speed of light $c$. 
The blue curve represents the solution $u(s)/c$ of the ODE system up to $c\,\bar{s}$.
The red curve shows the function $\varphi_A(s)/c$ defined in \eqref{def:phi_A}.
Starting from $c\,\hat{s}$, we compute the perturbed state corresponding to each given unperturbed state defined by the ODE solution at given value of $c\,s$. The quantity $u_\delta(s)/c$ is depicted by the yellow curve.
One verifies that all curves coincide at $c\,\bar{s}$ and that $u_\delta(s)/c$ is monotonically decreasing with a unique zero at $c\,s^\ast$ where the shock occurs.}
\end{figure}
\FloatBarrier
We further present numerical results for the velocity of monoatomic and diatomic gases, as well as for different values of $\gamma$, prescribing an initial velocity $u_0/c = -1/\sqrt{2}$.  
The numerical solution is obtained as follows: First, the ODE is solved up to $cs = \sqrt{3}$, since no shock occurs before this point.  
From there, the ODE is continued, and for every triple $(s,\gamma(s),u(s))$, the perturbed state is computed using the jump conditions.  
The correct solution $(s^\ast,\gamma(s^\ast),u(s^\ast)), \gamma_\delta(s^\ast)$, with $u_\delta(s^\ast) = 0$, is identified when the perturbed velocity vanishes.  
By the preceding theoretical results, $s^\ast$ is unique.  
Numerically, this is implemented using a bisection method, checking that the absolute value of the perturbed velocity is below a prescribed tolerance; here $\varepsilon = 10^{-6}$.  
Once $s^\ast$ is found, the solution becomes stationary due to the zero velocity in the perturbed state.
\begin{figure}
    \centering
    \includegraphics[width=\textwidth]{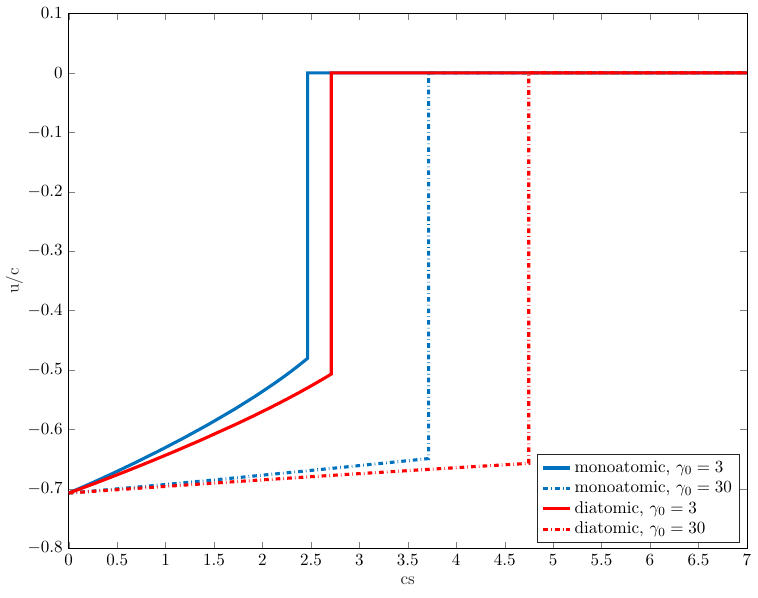}
    \caption{Shock solutions for the monoatomic (blue) and diatomic (red) gases with $u_0/c = -1/\sqrt{2}$, $\gamma_0 = 3$ (solid) and $\gamma_0 = 30$ (dashed). In the figure, the $x$–axis is expressed in terms of the dimensionless variable $cs$.}
  	\label{fig:shock_sol}
\end{figure}
%
%
\FloatBarrier
\section{The spherical piston problem}\label{sec:piston_res_prf}

In this section we consider the classical \emph{spherical piston problem} for the  
relativistic Euler system \eqref{Euler31}.  
We prescribe a uniform rest state
\[
(u,p,\eta)(0,r) = (0,p_0,\eta_0), \qquad p_0>0,\; \eta_0>0,
\]
and impose a boundary condition at a moving spherical surface of radius  
$r = \alpha t$, whose prescribed radial velocity is  
\[
u(t,\alpha t)=\alpha,\qquad \alpha\in(0,c).
\]

Physically, this setup describes a rigid spherical piston expanding outward with constant speed~$\alpha$, driving a compression wave into the surrounding relativistic fluid.  
The induced flow is radially symmetric and may develop shocks due to wave steepening, a phenomenon of particular relevance in high-energy astrophysical events, such as supernova explosions or relativistic jets.  
Relativistic effects, including Lorentz contraction and the coupling of energy and momentum, play a crucial role in shaping the dynamics, leading to non-linear behavior absent in classical fluids.  

Because both the initial data and the piston trajectory are scale-invariant, the induced motion cannot select any characteristic spatial or temporal scale.  
Consequently, the flow organizes itself into a self-similar pattern depending only on the similarity variable 
$
\xi =  {r}/{t}.
$

The main goal of this section is to establish existence and uniqueness of such a self-similar solution, which captures the entire piston-driven compression dynamics in a single reduced system of ordinary differential equations.  
Self-similar solutions also reveal universal features of relativistic flows, independent of specific initial conditions, making this problem a useful testbed for both theoretical analysis and numerical simulations.

   \begin{theorem}[Spherical Piston]\label{thm:piston}
   	For any given initial datum $(u_0,p_0,\eta_0)$ satisfying $p_0 > 0, \eta_0 > 0$, and
   	\[
   	(u,p,\eta)(0,r) = (0,p_0,\eta_0),\; u(t,\alpha t) = \alpha,\quad \alpha \in (0,c),
   	\]
   the spherical piston problem for system \eqref{Euler31} has a unique self-similar solution.
   \end{theorem}
   The proof of Theorem \ref{thm:piston} is presented subsequently. Note that we have a slightly different situation here. Due to the initial datum $u_0 = 0$, the solution up to the shock is constant. Then the state on the other side of the shock is governed by the jump conditions and from there the solution evolves according to \eqref{systemODE}, with the perturbed state being the initial datum. Thus we have an unperturbed state ahead of the shock given by $(0,p_0,\eta_0)$. Since a singularity occurs in \eqref{systemODE} when $\varphi_A(\overline{s}) = 0$ for
   \[
   \overline{s} = \frac{1}{c}\sqrt{e_p(p_0,\eta_0)},
   \]
   a shock wave will appear for some $s_P \in (1/c,\overline{s})$ with the speed $\sigma_P = 1/s_P$.
    Therefore,
   we need to look for a shock wave solution of the ODE system \eqref{systemODE}, which satisfies
   \[
   (u,p,\eta)(s_P) = (u_\delta,p_\delta,\eta_\delta)(s_P),\quad u(1/\alpha) = \alpha
   \]
    for $\alpha \in (0,c)$.\\
   Since we are considering a forward Lax shock, we have
   \[
   \sigma_P < \lambda_\delta\quad\text{and}\quad 0 = u_0 < u(s_P) = u_\delta < \sigma_P.
   \]
   A direct calculation shows (omitting the arguments)
   \begin{align*}
        &0 < \sigma_P < \lambda_\delta = \frac{c\sqrt{e_p(\gamma_\delta)}u_\delta + c^2}{c\sqrt{e_p(\gamma_\delta)} + u_\delta}
        \quad\stackrel{\sigma_P = 1/s_P}{\Leftrightarrow}\quad
        \frac{c\sqrt{e_p(\gamma_\delta)} - s_Pc^2}{cs_P\sqrt{e_p(\gamma_\delta)} - 1} < u_\delta\quad\Rightarrow\quad A > 0,\\
        &B = \left(c\sqrt{e_p(\gamma_\delta)}s_P + 1\right)u_\delta - \left(c\sqrt{e_p(\gamma_\delta)} + s_Pc^2\right) \stackrel{u_\delta < \sigma_P}{<} \frac{1}{s_P}(1 - s_P^2c^2) \stackrel{1/c < s}{<} 0
   \end{align*}
   and thus we have
   \[
   g = A\cdot B= c^2e_p(\gamma_\delta(s_P))\left(u_\delta(s_P)s_P - 1\right)^2 - \left(u_\delta(s_P) - s_Pc^2\right)^2 < 0.
   \]
   Therefore the solution of \eqref{systemODE} for the considered datum $(u,p,\eta)(s_P) = (u_\delta,p_\delta,\eta_\delta)(s_P)$ is locally well-defined.\\
   We can directly verify for the solution of the \eqref{systemODE} behind the shock that
   \[
   u'(s) > 0\quad\text{and}\quad p'(s)\,\begin{cases} &> 0,\;us < 1,\\ &< 0,\;us > 1.\end{cases}
   \]
   Collecting the previous results we have $\varphi_B(s) > 1/s$ for $s > 1/c$ and $\varphi_A(s) < u(s)$ as well as $\varphi_A'(s) < 0$ and $u'(s) > 0$ for $s \geq s_P > 1/c$. Thus we find a $\tilde{s}(s_P) > s_P$ such that
   \[
   u(\tilde{s}(s_P)) = \frac{1}{\tilde{s}(s_P)}.
   \]
   Note that $\tilde{s}(s_P)$ will play the role of $1/\alpha$.
   Since the solution of \eqref{systemODE} is continuous for $s > s_P$ and the given datum, $\tilde{s}$ is a continuous function of $s_P$.
   Due to  $u'(s) > 0$ in the interval $(s_P, \tilde{s}(s_P))$, we have
   \[
   \frac{1}{\tilde{s}(s_P)} = u(\tilde{s}(s_P)) > u(s_P) = u_\delta(s_P).
   \]
   Now, due to the Lax condition, we have
   \[
    u(s_P) = u_\delta > \frac{\sigma_Pc\sqrt{e_p(\gamma_\delta)} - c^2}{c\sqrt{e_p(\gamma_\delta)} - \sigma_P}.
   \]
   Together with $e_p > 3$, we hence conclude that for $s_P > 1/c$,
   \begin{align}\label{inf-ssp}
      \lim_{s_P \to 1/c}u(s_P) = c,\quad \lim_{s_P \to 1/c}\tilde{s}(s_P) = \frac{1}{c}.
   \end{align}
  Now we claim that
  \begin{align}\label{sup-ssp}
      \lim_{s_P \to \overline{s}}u(s_P) = 0,\quad \sup_{s_P \in (1/c,\overline{s})}\tilde{s}(s_P) = \infty.
   \end{align}
%
   Suppose the claim is false, then there exists a constant $R_0>1/c$ such that
\begin{align}\label{assu-ssR}
\sup_{s_P \in (1/c,\overline{s})}\tilde{s}(s_P) =R_0
\end{align}
Then, for $\overline{R}=R_0+\overline{s}$,  we consider the following equation
\begin{align}\label{vs-RR0}
    \frac{dv}{ds}=\frac{(d-1)v(c^2 - v^2)(v - sc^2)}{c^4\overline{s}^2(vs-1)^2-\left(v-sc^2\right)^2}, \quad \frac{1}{c}< s< \overline{R}
\end{align}
with $v(\overline{R})=1/\overline{R}$. By similar proof as in \cite[Lemma 2.6]{Lai2019}, we can obtain that the solution to \eqref{vs-RR0} satisfies $v(s)>0$ for  $s\in (1/c, \overline{R})$. This implies the solution $u=v(s)$ of \eqref{vs-RR0} should intersect the curve 
\[
u=\frac{c^2(s-\overline{s})}{1-c^2s\overline{s}}, \quad s>\frac1c
\]   
at some point $(\hat{s}, \hat{u})$ with $\hat{s} < \overline{s}$ and $\hat{u} > 0$.

Due to $\lim\limits_{s_P \to \overline{s}}u_\delta(s_P) = 0$ for the solution of \eqref{systemODE}, there exists some $\hat{s}_P$, which is sufficiently close to  $\overline{s}$, such that
\[
  u_\delta(\hat{s}_P) < \hat{u},\quad \hat{s} < \hat{s}_P < \overline{s}.
\]
 Let $\left(\widetilde{u},\widetilde{p},\widetilde{\eta}\right)(s)$ be the solution of \eqref{systemODE} with the initial condition
 \[
   (u,p,\eta)(\hat{s}_P) = (u_\delta,p_\delta,\eta_\delta)(\hat{s}_P).
 \]  
   Due to the assumption \eqref{assu-ssR}, there exists some $\check{s}\in (\hat{s}_P, \overline{R})$ such that
   \[
     v(\check{s})=\widetilde{u}(\check{s}), \qquad v(s)>\widetilde{u}(s),\,\, s\in\left(\hat{s}_P,\check{s}\right).
   \]
   Then we have
   \begin{align}\label{vus00}
        v'(\check{s})-\widetilde{u}'(\check{s})\leq 0.
   \end{align}
   Since $e_{pp}<0$ and $e_{p\eta}<0$ from \eqref{epeta}, we can obtain
   \[
   \overline{s}_0^2 \equiv e_p\left(\widetilde{p}(\check{s}), \eta_\delta(\hat{s}_P)\right) < e_p\left(p_0, \eta_\delta(\hat{s}_P)\right) < e_p\left(p_0, \eta_0\right) = \overline{s}_1^2
   \]
   since $\widetilde{p}(\check{s})>p_\delta(\hat{s}_P) > p_0$ and $\eta_\delta(\hat{s}_P) > \eta_0$. Therefore, we further get 
   \[
   v'(\check{s})=\frac{(d-1)v(c^2 - v^2)(v - \check{s}c^2)}{c^4\overline{s}_1^2(v\check{s}-1)^2-\left(v-\check{s}c^2\right)^2}>\frac{(d-1)\widetilde{u}(c^2 - \widetilde{u}^2)(\widetilde{u} - \check{s}c^2)}{c^4\overline{s}^2_0(\widetilde{u}\check{s}-1)^2-\left(\widetilde{u}-\check{s}c^2\right)^2}=\widetilde{u}'(\check{s}).
   \]
   This contradicts with \eqref{vus00} and thus the claim \eqref{sup-ssp} is valid. Hence Thm.\ \ref{thm:piston} follows from \eqref{inf-ssp} and \eqref{sup-ssp} and thus we completed the proof.
\section{Conclusion}
In this work we have studied the relativistic Euler equations in two and three dimensions with the physically realistic Synge equation of state for monoatomic gases, as well as its generalized version for diatomic gases.
We have established the existence and uniqueness of radially symmetric self-similar solutions to these systems, including the spherical piston problem.
A remarkable aspect of our analysis is that, thanks to the structure of the Synge energy, our results remain valid for the entire range of the relativistic temperature parameter~$\gamma$ \eqref{coldness}: from the ultra-relativistic regime where $\gamma \to 0$, through the relativistic regime where $\gamma \ll 1$, up to the classical limit corresponding to polytropic gases when $\gamma \to \infty$.
This stands in clear contrast with previous studies in the literature, which typically require restrictive assumptions on the constitutive equations.
In addition, we have rigorously proved the inequality $e_{pp}<0$ for the Synge energy both for monatomic than diatomic gases.
To the best of our knowledge, this result is new.
It plays a crucial role in the proof of our theorems and, at the same time, provides a rigorous justification of the monotone decrease of the characteristic velocity with respect to~$\gamma$.
The results obtained here contribute to the theoretical understanding of relativistic gases and may serve as benchmark solutions for the validation of numerical methods.
%



%
\appendix
\section{Modified Bessel functions and useful estimates}
In this part, we recall expressions and properties of the modified Bessel functions from \cite{Groot-Leeuwen-Weert-1980,Oliver-1974,Wa} and some useful estimates from \cite{Ruggeri-Xiao-Zhao-ARMA-2021}.
\begin{proposition}\cite{Groot-Leeuwen-Weert-1980,Oliver-1974,Wa}\label{def-pro}
	Let
	$K_j(\gamma)$ be the Bessel function defined by
	\begin{equation}\label{defini}
	 K_j(\gamma)=\frac{(2^j)j!}{(2j)!}\frac{1}{\gamma^j}\int_{\lambda=\gamma}^{\lambda=\infty}e^{-\lambda}(\lambda^2-\gamma^2)^{j-1/2}d\lambda,\quad(j\geq0).
	\end{equation}
	Then the following identities hold:
	\begin{align}	&K_j(\gamma)=\frac{2^{j-1}(j-1)!}{(2j-2)!}\frac{1}{\gamma^j}\int_{\lambda=\gamma}^{\lambda=\infty}e^{-\lambda}\lambda(\lambda^2-\gamma^2)^{j-3/2}d\lambda,\quad(j>0),\nonumber\\
	&K_{j+1}(\gamma)=2j\frac{K_j(\gamma)}{\gamma}+K_{j-1}(\gamma),\quad(j\geq1), \label{transform}\\
	&K_j(\gamma)<K_{j+1}(\gamma),\quad (j\geq0), \nonumber
	\end{align}
	and
	\begin{eqnarray}
	&\frac{d}{d\gamma}\left(\frac{K_j(\gamma)}{\gamma^j}\right)=-\frac{K_{j+1}(\gamma)}{\gamma^j},
\quad(j\geq0),\label{deriva}\\
	&\displaystyle K_{j}(\gamma)=\sqrt{\frac{\pi}{2\gamma}}e^{-\gamma}\Big(\gamma_{j,n}(\gamma)\gamma^{-n}
+\sum_{m=0}^{n-1}A_{j,m}\gamma^{-m}\Big),\quad(j\geq0,~n\geq1),\label{remainder}
	\end{eqnarray}
	where expressions of the coefficients in (\ref{remainder}) are
		\begin{align*}
	\begin{aligned}
	&A_{j,0}=1\\
	&A_{j,m}=\frac{(4j^2-1)(4j^2-3^2)\cdots(4j^2-(2m-1)^2)}{m!8^m},\quad(j\geq0,~m\geq1), \\
	&|\gamma_{j,n}(\gamma)|\leq2e^{[j^2-1/4]\gamma^{-1}}|A_{j,n}|,\quad(j\geq0,~n\geq1).
	\end{aligned}
	\end{align*}
		On the other hand, according to \cite[Page 80]{Wa}, the Bessel functions defined in (\ref{defini}) can also be written in the following form:
	\begin{align} 	K_0(\gamma)=&-\sum^{\infty}_{m=0}\frac{(\frac{1}{2}\gamma)^{2m}}{m!m!}
\Big[\ln\Big(\frac{\gamma}{2}\Big)-\psi(m+1)\Big],\label{k0}\\
	K_n(\gamma)=&\frac{1}{2}\sum^{n-1}_{m=0}(-1)^m\frac{(n-m-1)!}{m!}
\Big(\frac{1}{2}\gamma\Big)^{-n+2m}\nonumber\\
	&+(-1)^{n+1}\sum^{\infty}_{m=0}\frac{(\frac{1}{2}\gamma)^{n+2m}}{m!(m+n)!}\Big[\ln\Big(\frac{\gamma}{2}\Big)
	-\frac{1}{2}\psi(n+m)-\frac{1}{2}\psi(n+m+1)\Big],\nonumber\\
	\psi(1)=&-C_E,\quad \psi(m+1)=-C_E+\sum^{m}_{k=1}\frac{1}{k},\quad m\geq1,\nonumber\\
	 K_1(\gamma)=&\frac{1}{\gamma}+\sum^{\infty}_{m=0}\frac{(\frac{1}{2}\gamma)^{2m+1}}{m!(m+1)!}
\Big[\ln\Big(\frac{\gamma}{2}\Big)-\frac{1}{2}\psi(m+1)-\frac{1}{2}\psi(m+2)\Big],\label{k1}
		\end{align}
	where $C_E=0.5772157\ldots$ is the Euler's constant.
\end{proposition}
For later use, we also need two different estimates:
\begin{proposition}\label{K01p}\cite{Ruggeri-Xiao-Zhao-ARMA-2021} Let $\gamma\in(0, \infty)$. $\frac{K_0(\gamma)}{K_1(\gamma)}$ satisfies:
\item  for $\gamma\in(0, \gamma_0]$,
	\begin{align}	&\frac{\gamma}{\sqrt{\gamma^2+1}+1}\leq\frac{K_0(\gamma)}{K_1(\gamma)}\leq \gamma\left[\frac{11}{16}-\left(\ln(\frac{\gamma}{2})+C_E\right)\right],\label{uplow-b1}\\
	&\hspace{2.5cm}\left(\frac{K_0(\gamma)}{K_1(\gamma)}\right)^2+\frac{2}{\gamma}\frac{K_0(\gamma)}{K_1(\gamma)}-1>0;\label{low1}
\end{align}
\item for $\gamma\in(\gamma_0, \sqrt{2}]$,
	\begin{equation}\label{new1}
	\frac{K_0(\gamma)}{K_1(\gamma)}\leq 1-\frac{\gamma_0-1}{\gamma};
	\end{equation}
\item for $\gamma\in(\sqrt{2}, \infty)$,
	\begin{equation}\label{rough}
	1-\frac{1}{2\gamma}\leq\frac{K_0(\gamma)}{K_1(\gamma)}\leq 1-\frac{1}{2\gamma}+\frac{3}{8\gamma^2}+\frac{3}{16\gamma^3}.
	\end{equation}
Moreover, for $\gamma\in(2, \infty)$, it holds that
	\begin{align} \label{acurate}
	\begin{aligned}
	&\frac{K_0(\gamma)}{K_1(\gamma)}\geq 1-\frac{1}{2\gamma}+\frac{3}{8\gamma^2}-\frac{3}{8\gamma^3}+\frac{63}{128\gamma^4}-\frac{31}{20\gamma^5}, \\
	&\frac{K_0(\gamma)}{K_1(\gamma)}\leq 1-\frac{1}{2\gamma}+\frac{3}{8\gamma^2}-\frac{3}{8\gamma^3}+\frac{63}{128\gamma^4}+\frac{7}{8\gamma^5}.
	\end{aligned}
	\end{align}
Here $\gamma_0=1.1229189\ldots$ is a constant satisfying
$$\ln\Big(\frac{\gamma_0}{2}\Big)+C_\mathrm{E}=0.$$
		\end{proposition}
\begin{proposition} 
\cite{Ruggeri-Xiao-Zhao-ARMA-2021} Let $\gamma\in(0,\infty)$ and $K_j(\gamma) (j\geq0)$ be the functions defined in Proposition \ref{def-pro}. Then it holds that
	\begin{equation}\label{imp-ine1}
	\gamma^2\left(\frac{K_1(\gamma)}{K_2(\gamma)}\right)^2+3\gamma\frac{K_1(\gamma)}{K_2(\gamma)}- \gamma^2-3<0,
	\end{equation}
	\begin{equation*}
\gamma\left(\frac{K_1(\gamma)}{K_2(\gamma)}\right)^3+4\left(\frac{K_1(\gamma)}{K_2(\gamma)}\right)^2-\gamma\frac{K_1(\gamma)}{K_2(\gamma)}-1<0.
	\end{equation*}
\end{proposition}

\noindent\underline{\textit{ Lower bound of $\frac{K_0(\gamma)}{K_1(\gamma)}$ for $\gamma\in(\gamma_0, \sqrt{2}]$.}} To prove $e_{pp}<0$, besides inequality estimates in Proposition \ref{K01p}, it is essential to obtain a lower bound of $\frac{K_0(\gamma)}{K_1(\gamma)}$ for $\gamma\in(\gamma_0, \sqrt{2}]$.
\begin{proposition}
  For $\gamma\in(\gamma_0, \sqrt{2}]$, $\frac{K_0(\gamma)}{K_1(\gamma)}$ satisfies
  \begin{align}\label{k01-02}
    \frac{K_0(\gamma)}{K_1(\gamma)}\geq \frac{\gamma}{2}.
  \end{align}
\end{proposition}
\begin{proof} Note that from the definition of $\gamma_0$, for $\gamma\in(\gamma_0, \sqrt{2}]$,
$$\frac{1}{2}-\ln\Big(\frac{\gamma}{2}\Big)-C_\mathrm{E}=\frac{1}{2}+\ln\Big(\frac{\gamma_0}{\gamma}\Big)>0.$$
Then for $\gamma\in(\gamma_0, \sqrt{2}]$, we get from \eqref{k0} and \eqref{k1} that
\begin{equation*}
	\frac{K_0(\gamma)}{K_1(\gamma)}\geq \frac{\ln\big(\frac{\gamma_0}{\gamma}\big)+\frac{\gamma^2}{4}\big[1+\ln\big(\frac{\gamma_0}{\gamma}\big)\big]
+\frac{\gamma^4}{64}\big[\frac{3}{2}+\ln\big(\frac{\gamma_0}{\gamma}\big)\big]
	} {\frac{1}{\gamma}-\frac{\gamma}{2}\big[\frac{1}{2}+\ln\big(\frac{\gamma_0}{\gamma}\big)\big]-\frac{\gamma^3}{16}
\big[\frac{5}{4}+\ln\big(\frac{\gamma_0}{\gamma}\big)\big]-\frac{\gamma^5}{384}\big[\frac{5}{3}+\ln\big(\frac{\gamma_0}{\gamma}\big)\big]
	}.
	\end{equation*}
Via direct computations, we can obtain that to prove \eqref{k01-02}, it suffices to show
\begin{align*}
f(\gamma):=&-\frac{1}{2}+\frac{3}{8}\gamma^2+\frac{1}{16}\gamma^4
+\frac{41}{384\times36}\gamma^6\nonumber\\
&+\Big[1+\frac{1}{2}\gamma^2+\frac{3}{64}\gamma^4+\frac{19}{384\times36}\gamma^6\Big]\ln\Big(\frac{\gamma_0}{\gamma}\Big)\geq0
\end{align*}
for $\gamma\in(\gamma_0, \sqrt{2}]$. We have
\begin{align*}
f'(\gamma)=&-\frac{1}{\gamma}+\frac{\gamma}{4}+\frac{13}{64}\gamma^3+\frac{41}{2304}\gamma^5
+\Big[\gamma+\frac{3}{16}\gamma^3+\frac{19}{2304}\gamma^5\Big]
\ln\Big(\frac{\gamma_0}{\gamma}\Big).
\end{align*}
It is straightforward to verify that $f(\sqrt{2})>0$ and $f'(\gamma)<0$ for $\gamma\in(\gamma_0, \sqrt{2}]$. Then \eqref{k01-02} holds.
\end{proof}

\subsection{Asymptotic Results}
We further want to provide detailed calculations concerning the asymptotic behavior of expressions involving the Bessel functions.
In \eqref{e_pddgammap} we need to discuss the expression
\begin{align*}
	\frac{e_p}{(e + p)}\frac{\dd p}{\dd \gamma} = \gamma \left(h^2_i - 1\right) + (2i + 1) h_i - \frac{3}{\gamma}.
\end{align*}
Since the situation for $p = 0$ is studied at this point in the proof of Lem.\ \ref{lem:no_zeros} we study the expression for large $\gamma$ we yield the following asymptotic behavior at the aid of \eqref{acurate}
\begin{align}
	\begin{split}\label{asymptotic_1}
		\gamma \left(h^2_0 -1\right)+ h_0-\frac{3}{\gamma}
		&= \gamma \left(-\frac{1}{\gamma} + \frac{1}{\gamma^2}\right) + \left(1-\frac{1}{2\gamma}\right) - \frac{3}{\gamma} +\ldots
		= -\frac{5}{2\gamma}+\ldots,\\
   		\gamma \left(h^2_1 - 1\right) + 3 h_1 - \frac{3}{\gamma}
   		&= \gamma \left(-\frac{3}{\gamma} + \frac{6}{\gamma^2}\right) + 3 \left(1 - \frac{3}{2\gamma}\right) - \frac{3}{\gamma}+\ldots
   		= -\frac{3}{2\gamma}+\ldots\, .
   	\end{split}
\end{align}
In the proof of the subsequent Lem.\ \ref{lem:u_eq_s} we study \eqref{ddgammap}
\begin{align*}
	\frac{1}{c^2(e + p)}\dd p = \frac{\gamma^2 \left(h^2_i - 1\right) + (2i+1)\gamma h_i - 4}{c^2\gamma(\gamma	h_i + 4)}\dd \gamma.
\end{align*}
For $\gamma \to \infty$ we have again using \eqref{acurate}
\begin{align}
	\begin{split}\label{asymptotic_2}
		\frac{\gamma^2 \left(h^2_0 - 1\right) + \gamma h_0 - 4}{\gamma(\gamma h_0 + 4)}
		&= \frac{\gamma^2 \left(-\frac{1}{\gamma} + \frac{1}{\gamma^2}\right) + \gamma \left(1 - \frac{1}{2\gamma}\right) - 4 +\ldots}{\gamma\left[4 + \gamma \left(1 - \frac{1}{2\gamma} + \ldots\right)\right]}
		= \frac{-\frac{7}{2} + \ldots}{\gamma^2 + \frac{7\gamma}{2} +\ldots},\\
	   \frac{\gamma^2 \left(h^2_1 - 1\right) + 3\gamma h_1 - 4}{\gamma(\gamma h_1 + 4)} &= \frac{\gamma^2 \left(-\frac{3}{\gamma} + \frac{6}{\gamma^2}\right) + 3\gamma \left(1 - \frac{3}{2\gamma}\right) - 4 + \ldots}{\gamma\left[4 + \gamma \left(1 - \frac{3}{2\gamma} + \ldots\right)\right]}
	   = \frac{-\frac{5}{2} + \ldots}{\gamma^2 + \frac{5\gamma}{2} + \ldots}
	\end{split}
\end{align}
For the proof of Lem.\ \ref{lem:case_1} we need to study the expression
\begin{align*}
    \frac{\sqrt{e_p}}{e + p} &= -\gamma_p\left[\frac{(\gamma^2\Phi'_i(\gamma) - 1)\Phi_i'(\gamma)}{\gamma\Phi_i(\gamma) + 1}\right]^{\frac{1}{2}}.
\end{align*}
Similar to the previous estimation for \eqref{ddgammap},  when $\gamma$ is large, one has
\begin{align}
	\frac{(\gamma^2\Phi'_i(\gamma) - 1)\Phi_i'(\gamma)}{\gamma\Phi_i(\gamma) + 1}
	&=\frac{\left[\gamma	^2 \left(h^2_i -1\right)+(2i+1)\gamma h_i-4\right]\left[\gamma^2 \left(h^2_i -1\right)+(2i+1)\gamma h_i-3\right]}{\gamma^3(\gamma h_i+4)}\notag\\
    &=\frac{\frac{(5-2i)(5-2i+2)}{4}+\ldots}{\gamma^3\left(\frac{7-2i}{2}+\ldots\right)}.\label{asymptotic_3}
\end{align}
%
%
\section{The proof of $e_{pp}<0$} \label{Appendice1}
 In this section, we will prove the crucial inequality $e_{pp}<0$ for monatomic gases case and diatomic gases case in a successive way.
%
\subsection{Proof of $e_{pp|\text{mono}}<0$} 
With the preparations above, we are at position of the proof of $e_{pp|\text{mono}}<0$.
Note that
\begin{align}\label{pg}
-\gamma
		^2+\gamma  h_1 (\gamma
		h_1+3)-4<0
\end{align}
from \eqref{imp-ine1}.
From the proof of \cite[Proposition 5]{Ruggeri-Xiao-Zhao-ARMA-2021}, we know that
\begin{align}
I_1(\gamma)<0,\qquad I_2(\gamma)>0. \label{i12}
\end{align}
Collecting \eqref{pg} and \eqref{i12} in \eqref{epp-mon}, one can conclude that in order to show $e_{pp|\text{mono}}<0$, it suffices to prove $I_3(\gamma)<0$, which can be transformed into the following inequality w.r.t. $h_0$:
\begin{align}\label{ig}
I(\gamma):= (\gamma+1)h_0^2+\frac{4}{\gamma}h_0-\left(\gamma+\frac{4}{\gamma}
-\frac{4}{\gamma^2}\right)<0
\end{align}
by \eqref{transform}. Now we prove \eqref{ig} in the following four cases:
$$\gamma\in (0,  \gamma_0];\qquad \gamma\in(\gamma_0, \sqrt{2}];\qquad \gamma\in(\sqrt{2}, 3);\qquad \gamma\in[3, \infty).$$
\noindent $\bullet$ Case 1 ( $\gamma\in (0,  \gamma_0]$): From \eqref{low1}, we have
\begin{align*}
  I(\gamma)\geq& (\gamma+1)-\left(\gamma+\frac{4}{\gamma}
-\frac{4}{\gamma^2}\right)+\left[\frac{4}{\gamma}-(\gamma+1)\right]h_0\\
>& 1- \frac{4}{\gamma}
+\frac{4}{\gamma^2}=\left(1-\frac{2}{\gamma}\right)^2>0.
\end{align*}

\noindent $\bullet$ Case 2 ( $\gamma\in (\gamma_0,  \sqrt{2}]$): From \eqref{k01-02}, one has
\begin{align*}
  I(\gamma)\geq& (\gamma+1)\frac{\gamma^2}{4}+2-\left(\gamma+\frac{4}{\gamma}
-\frac{4}{\gamma^2}\right)\\
>& \left(\frac{\gamma^2}{4}-\gamma+1\right)+\left(1- \frac{4}{\gamma}
+\frac{4}{\gamma^2}\right)>0.
\end{align*}
\noindent $\bullet$ Case 3 ( $\gamma\in (\sqrt{2}, 3)$): From \eqref{rough}, we can obtain
\begin{align*}
  I(\gamma)\geq& (\gamma+1)\left(1-\frac{1}{2\gamma}\right)^2+\frac{4}{\gamma}\left(1-\frac{1}{2\gamma}\right)-\left(\gamma+\frac{4}{\gamma}
-\frac{4}{\gamma^2}\right)\\
=& - \frac{3}{4\gamma}
+\frac{9}{4\gamma^2}=\frac{3}{4\gamma}\left(\frac{3}{\gamma}-1\right)>0.
\end{align*}
\noindent $\bullet$ Case 4 ( $\gamma\in [3, \infty)$):
From \eqref{acurate}, we can obtain that
$$h_0\geq 1-\frac{1}{2\gamma}+\frac{3}{8\gamma^2}-\frac{1}{2\gamma^3},$$
and
\begin{align*}
  I(\gamma)\geq& (\gamma+1)\left(1-\frac{1}{2\gamma}+\frac{3}{8\gamma^2}-\frac{1}{2\gamma^3}\right)^2
  +\frac{4}{\gamma}\left(1-\frac{1}{2\gamma}+\frac{3}{8\gamma^2}-\frac{1}{2\gamma^3}\right)-\left(\gamma+\frac{4}{\gamma}
-\frac{4}{\gamma^2}\right)\\
=& (\gamma+1)\left(1-\frac{1}{\gamma}+\frac{1}{\gamma^2}-\frac{11}{8\gamma^3}+\frac{41}{64\gamma^4}-
\frac{3}{8\gamma^5}+\frac{1}{4\gamma^3}\right)-\gamma+\frac{2}{\gamma^2} +\frac{3}{2\gamma^3}-\frac{2}{\gamma^4}\\
>&\frac{13}{8\gamma^2} +\frac{1}{8\gamma^3}-\frac{2}{\gamma^4}>0.
\end{align*}
%
%
\subsection{Proof of $e_{pp|\text{dia}}<0$}\label{app:eppd}
Now we come to prove $e_{pp|\text{dia}}<0$. For this purpose, We first show that
\begin{align}\label{eppd-den}
  \gamma
			\left(\gamma
			\left(h_0^2-1\right)+
				h_0\right)-4<0 .
\end{align}
Noting that $h_0< 1$, it is straightforward to verify \eqref{eppd-den} for $\gamma\leq 4$. When $\gamma>4$, from \eqref{acurate}, we have $h_0\leq 1-\frac{1}{2\gamma}+\frac{3}{8\gamma^2}$. Then we can obtain that for $\gamma>4$,
\begin{align*}
  \gamma \left(\gamma
			\left(h_0^2-1\right)+
				h_0\right)-4 \leq& \gamma^2\left(1-\frac{1}{2\gamma}+\frac{3}{8\gamma^2}\right)^2-\gamma^2
+\gamma\left(1-\frac{1}{2\gamma}+\frac{3}{8\gamma^2}\right) -4 \\
  \leq & -\gamma+1-\frac{3}{8\gamma}+\frac{9}{64\gamma^2}-\gamma-\frac{9}{2}+\frac{3}{8\gamma}<-3.
\end{align*}
Therefore \eqref{eppd-den} holds. Combing \eqref{eppd-den} and the expression of $e_{pp|\text{dia}}$ in \eqref{epp-dia}, we infer that to show $e_{pp|\text{dia}}<0$  is equivalent to show that
\begin{align}\label{f-dia}
  f_{\text{dia}} >0\quad \mbox{or} \quad \gamma f_{\text{dia}}>0.
\end{align}
Now we prove \eqref{f-dia} in the following four cases:
$$\gamma\in (0,  \gamma_0];\qquad \gamma\in(\gamma_0, \sqrt{2}];\qquad \gamma\in(\sqrt{2}, 2);\qquad \gamma\in[2, \infty).$$
\noindent $\bullet$ Case 1 ( $\gamma\in (0,  \gamma_0]$): In this case, for convenience of our proof, we further divide it into three subcases:
$$\gamma\in \left(0,  \frac{1}{2}\right];\qquad \gamma\in\left(\frac{1}{2}, \frac{9}{10}\right];\qquad \gamma\in\left(\frac{9}{10}, \gamma_0\right).$$
Note that  \eqref{low1} holds in this case and it is valid for any $\gamma>0$. Now we verify \eqref{f-dia} for the above three subcases one by one.

\noindent\underline{\textit{Subcase 1.1 ($\gamma\in \left(0,  \frac{1}{2}\right]$):}}
Now we rewrite $f_{\text{dia}}$ as
\begin{align}
f_{\text{dia}}=:&\left(h_0^2+\frac{2}{\gamma}h_0-1\right)f_1+2\gamma h_0+8,\label{fdf1}
\end{align}
where
$$f_1:=\gamma^4h_0^4+2\gamma^3h_0^3
-\left(2\gamma^4+8\gamma^2\right)
h_0^2-\left(2\gamma^3+16\gamma^2\right)
h_0.$$
Noting that $f_1$ decrease w.r.t. $h_0$,  we use the upper bound estimate of $h_0$ in \eqref{uplow-b1} to have
\begin{align*}
  f_1 \geq g_1&= \gamma^8\left(\frac{11}{16}-\ln\left(\frac{\gamma}{2}\right)-C_E\right)^4
  +2\gamma^6\left(\frac{11}{16}-\ln(\frac{\gamma}{2})-C_E\right)^3\\
  &-\left(2\gamma^6+8\gamma^4\right)
  \left(\frac{11}{16}-\ln\left(\frac{\gamma}{2}\right)-C_E\right)^2-\left(2\gamma^4+16\gamma^2\right)
  \left(\frac{11}{16}-\ln\left(\frac{\gamma}{2}\right)-C_E\right)\\
  &+\gamma^4+7\gamma^2+8.
\end{align*}
Note that $\frac{11}{16}-\ln\left(\frac{\gamma}{2}\right)-C_E$ decrease w.r.t. $\gamma$ and
$$\frac{11}{16}-\ln\left(\frac{1}{4}\right)-C_E<\frac{3}{2}.$$
It is straightforward to verify that
$$f_1 \geq g_1(\gamma)>0,\qquad \gamma\in \left(0,  \frac{1}{2}\right].$$
This together with \eqref{fdf1} implies \eqref{f-dia}.

\noindent\underline{\textit{Subcase 1.2 ($\gamma\in \left(\frac{1}{2},  \frac{9}{10}\right]$):}}
By direct computation, we have
\begin{align*}
   g'_1(\gamma)=& 8\gamma^7\left(\frac{11}{16}-\ln\left(\frac{\gamma}{2}\right)-C_E\right)^4
  +\left(12\gamma^5 -4\gamma^7\right)\left(\frac{11}{16}-\ln\left(\frac{\gamma}{2}\right)-C_E\right)^3\\
  &-\left(18\gamma^5+32\gamma^3\right)
  \left(\frac{11}{16}-\ln\left(\frac{\gamma}{2}\right)-C_E\right)^2+6\gamma^3+30\gamma\\
  &+\left(4\gamma^5+8\gamma^3-32\gamma\right)
  \left(\frac{11}{16}-\ln\left(\frac{\gamma}{2}\right)-C_E\right).
\end{align*}
When $\gamma\in\left(\frac{1}{2}, \overline{\gamma}=0.821545\ldots\right]$ with $\overline{\gamma}$ satisfying
$$\frac{11}{16}-\ln(\frac{\overline{\gamma}}{2})-C_E=1,$$
it is easy to see that $g'_1(\gamma)<0$. While for $\gamma\in\left(\overline{\gamma},\frac{9}{10}\right]$ and $\widetilde{\gamma}=0.908792\ldots$ satisfying
$$\frac{11}{16}-\ln(\frac{9}{20})-C_E=\widetilde{\gamma},$$
 it holds that
\begin{align*}
   g'_1(\gamma)\leq& 2\gamma\left[-16\gamma^2\widetilde{\gamma}^4
  +\left(2\gamma^4+4\gamma^2-16\right)\widetilde{\gamma}+3\gamma^2+15\right]\\
  \leq&2\gamma\left[-13.21\gamma^2+1.8176\gamma^4+3.6352\gamma^2+3\gamma^2-0.46
  \right].
\end{align*}
Then we can conclude that $g'_1(\gamma)<0$ for $\gamma\in \left(\frac{1}{2},  \frac{9}{10}\right]$. Now we can further obtain
\begin{align*}
  f_1 \geq g_2=& \gamma^8\widetilde{\gamma}^4
  +2\gamma^6\widetilde{\gamma}^3-\left(2\gamma^6+8\gamma^4\right)
  \widetilde{\gamma}^2-\left(2\gamma^4+16\gamma^2\right)\widetilde{\gamma}+\gamma^4+7\gamma^2+8\\
  \geq&(0.9)^8\widetilde{\gamma}^4
  +2\times (0.9)^6\widetilde{\gamma}^3-\left(2\times(0.9)^6+8\times(0.9)^4\right)
  \widetilde{\gamma}^2\\
  &-\left(2\times(0.9)^4+16\times(0.9)^2\right)\widetilde{\gamma}+(0.9)^4+7\times(0.9)^2+8\\
  \geq&-2.77
\end{align*}
for $\gamma\in \left(\frac{1}{2},  \frac{9}{10}\right]$, where we used $g_2(\gamma)<0$ in the second inequality. Inserting the estimate of $f_1$ into \eqref{fdf1}, we use the upper bound estimate in \eqref{uplow-b1} to have
\begin{align}
f_{\text{dia}}\geq&-2.77\left(h_0^2+\frac{2}{\gamma}h_0-1\right)+2\gamma h_0+8\nonumber\\
\geq&10.77-2.77\gamma^2\left(\frac{11}{16}-\ln\left(\frac{\gamma}{2}\right)-C_E\right)^2
-\left(5.54-2\gamma^2\right)\left(\frac{11}{16}-\ln\left(\frac{\gamma}{2}\right)-C_E\right)\nonumber\\
=:&10.77-g_3(\gamma)\geq10.77-g_3(0.5)\nonumber\\
\geq&10.77-\frac{2.77}{4}\times\frac{9}{4}
-\left(5.54-0.5\right)\times\frac{3}{2}\geq 1.65.\label{fcase12}
\end{align}
Here we used the fact for $\gamma\in \left(\frac{1}{2},  \frac{9}{10}\right]$,
\begin{align*}
  g'_3(\gamma)= & -5.54\gamma\left(\frac{11}{16}-\ln\left(\frac{\gamma}{2}\right)-C_E\right)^2
  +5.54\gamma\left(\frac{11}{16}-\ln\left(\frac{\gamma}{2}\right)-C_E\right)\\
  &+4\gamma\left(\frac{11}{16}-\ln\left(\frac{\gamma}{2}\right)-C_E\right)+\frac{5.54-2\gamma^2}{\gamma}>0.
\end{align*}

\noindent\underline{\textit{Subcase 1.3 ($\gamma\in \left(\frac{9}{10},  \gamma_0\right]$):}} Note that for $\gamma\in \left(\frac{9}{10},  \gamma_0\right]$,
$$\left[\gamma\left(\frac{11}{16}-\ln\left(\frac{\gamma}{2}\right)-C_E\right)\right]_{\gamma}
=\frac{11}{16}-\ln\left(\frac{\gamma}{2}\right)-C_E-1<0.$$
This together with \eqref{uplow-b1} implies that for $\gamma\in \left(\frac{9}{10},  \gamma_0\right]$,
$$h_0\leq \gamma\left(\frac{11}{16}-\ln\left(\frac{\gamma}{2}\right)-C_E\right)\leq 0.9\widetilde{\gamma}.$$
Since $f_1$ decrease w.r.t. $h_0$, we can obtain
\begin{align*}
  f_1 \geq&\gamma^4(0.9)^4\widetilde{\gamma}^4
  +2\gamma^3\times (0.9)^3\widetilde{\gamma}^3-\left(2\gamma^4+8\gamma^2\right)\times 0.81
  \widetilde{\gamma}^2\\
  &-\left(2\gamma^3+16\gamma\right)\times0.9\widetilde{\gamma}+\gamma^4+7\gamma^2+8\\
  \geq&\gamma_0^4(0.9)^4\widetilde{\gamma}^4
  +2\gamma_0^3\times (0.9)^3\widetilde{\gamma}^3-\left(2\gamma_0^4+8\gamma_0^2\right)\times 0.81
  \widetilde{\gamma}^2\\
  &-\left(2\gamma_0^3+16\gamma_0\right)\times0.9\widetilde{\gamma}+\gamma_0^4+7\gamma_0^2+8\\
  \geq&-5.21
\end{align*}
As the derivation of \eqref{fcase12}, we can obtain that for $\gamma\in \left(\frac{9}{10},  \gamma_0\right]$,
\begin{align*}
f_{\text{dia}}\geq&-5.21\left(h_0^2+\frac{2}{\gamma}h_0-1\right)+2\gamma h_0+8\nonumber\\
\geq&13.21-5.21\times0.81\widetilde{\gamma}^2
-\left(10.42-1.62\right)\widetilde{\gamma}\geq 1.72.
\end{align*}

\noindent $\bullet$ Case 2 ( $\gamma\in (\gamma_0,  \sqrt{2}]$): We use the lower bound estimate \eqref{k01-02} to have
\begin{align*}
  f_1 \geq g_4=& \frac{\gamma^8}{16}+\frac{\gamma^6}{4}-\frac{\gamma^6+4\gamma^4}{2}-\gamma^4-8\gamma^2+\gamma^4+7\gamma^2+8\\
  =&\frac{\gamma^8}{16}-\frac{\gamma^6}{4}-2\gamma^4-\gamma^2+8
  \geq g_4(\sqrt{2})=-3,
\end{align*}
where we used the monotonicity of $g_4(\gamma)$ for $\gamma\in (\gamma_0,  \sqrt{2}]$. Now we combine the estimate of $f_1$ and \eqref{fdf1}, and further use the upper bound estimate in \eqref{new1} to have
\begin{align}
f_{\text{dia}}\geq&-3\left(h_0^2+\frac{2}{\gamma}h_0-1\right)+2\gamma h_0+8\nonumber\\
\geq&11-3\gamma^2\left(1-\frac{\gamma_0-1}{\gamma}\right)^2
-\left(\frac{6}{\gamma}-2\gamma\right)\left(1-\frac{\gamma_0-1}{\gamma}\right)\nonumber\\
=&2\gamma+10-2\gamma_0-\frac{6(2-\gamma_0)}{\gamma}+\frac{(9-3\gamma_0)(\gamma_0-1)}{\gamma^2}\nonumber\\
\geq&10-\frac{6(2-\gamma_0)}{\gamma_0}+\frac{(9-3\gamma_0)(\gamma_0-1)}{\gamma_0^2}> 0.\nonumber
\end{align}

\noindent $\bullet$ Case 3 ( $\gamma\in (\sqrt{2}, 2]$): Now we rewrite $f_{\text{dia}}$ as
\begin{align*}
  f_{\text{dia}}=&-\gamma ^5 \left(1-h_0^2\right)^3+4 \gamma ^4 h_0
			\left(h_0^2-1\right)^2+32 \gamma ^2 h_0 \left(1-h_0^2\right)\\
&+4\gamma^3\left(1+h_0^2\right)\left(1-h_0^2\right)+\left(11\gamma^3-24
			\gamma\right) h_0^2+16
			h_0-11\gamma ^3.
\end{align*}
Noting that $f_{\text{dia}}$ decrease w.r.t. $1-h_0^2$, we use the lower bound estimate of $h_0$ in \eqref{rough} to have
\begin{align*}
  f_{\text{dia}}\geq&-\gamma ^5 \left(\frac{1}{\gamma}-\frac{1}{4\gamma^2}\right)^3+4 \gamma ^4 h_0
			\left(\frac{1}{\gamma}-\frac{1}{4\gamma^2}\right)^2+32 \gamma ^2 h_0 \left(\frac{1}{\gamma}-\frac{1}{4\gamma^2}\right)\\
&+4\gamma^3\left(1+h_0^2\right)\left(\frac{1}{\gamma}-\frac{1}{4\gamma^2}\right)+\left(11\gamma^3-24
			\gamma\right) h_0^2+16
			h_0-11\gamma ^3\\
=&\left(11\gamma^3+4\gamma^2-25
			\gamma\right) h_0^2+\left(4\gamma^2+30\gamma+\frac{33}{4}\right)
			h_0-11\gamma ^3+3\gamma^2\\
&-\frac{\gamma}{4}-\frac{3}{16}+\frac{1}{64\gamma}=:g_5.
\end{align*}
Obviously, $g_5$ increase w.r.t. $h_0$. Then we further use the lower bound estimate of $h_0$ in \eqref{rough} to obtain
\begin{align*}
  f_{\text{dia}}\geq&\left(11\gamma^3+4\gamma^2-25
			\gamma\right) \left(1-\frac{1}{2\gamma}\right)^2+\left(4\gamma^2+30\gamma+\frac{33}{4}\right)
			\left(1-\frac{1}{2\gamma}\right)\\
&-11\gamma ^3+3\gamma^2-\frac{\gamma}{4}-\frac{3}{16}+\frac{1}{64\gamma}
\geq\frac{3\gamma}{2}+19\frac{1}{16}-\frac{83}{8\gamma}>0.
\end{align*}

\noindent $\bullet$ Case 4 ( $\gamma\in (2, \infty)$): Since $f_1$ decrease w.r.t. $h_0$,  we use the upper bound estimate of $h_0$ in \eqref{acurate} to have
\begin{align*}
  \gamma f_1 \geq & \gamma^5\left(1-\frac{1}{2\gamma}+\frac{1}{8\gamma^2}\right)^4
  +2\gamma^4\left(1-\frac{1}{2\gamma}+\frac{1}{8\gamma^2}\right)^3-\left(2\gamma^5+8\gamma^3\right)
  \left(1-\frac{1}{2\gamma}+\frac{1}{8\gamma^2}\right)^2\\
  &-\left(2\gamma^4+16\gamma^2\right)
  \left(1-\frac{1}{2\gamma}+\frac{1}{8\gamma^2}\right)+\gamma^4+7\gamma^2+8\\
  \geq&-2\gamma^3-7\gamma^2+\frac{23\gamma}{2}-1.
\end{align*}
With this estimate, we can bound $\gamma f_{\text{dia}}$ as
\begin{align*}
 \gamma f_{\text{dia}}\geq&\left(-2\gamma^3-7\gamma^2+\frac{23\gamma}{2}-1\right)\left(h_0^2+\frac{2}{\gamma}h_0-1\right)
 +2\gamma^2 h_0+8\gamma\\
 =&\left(-2\gamma^3-7\gamma^2+\frac{23\gamma}{2}-1\right)h_0^2+\left(-2\gamma^2-14\gamma+23-\frac{2}{\gamma}\right)h_0
 +2\gamma^3-7\gamma^2-\frac{7\gamma}{2}+1.
\end{align*}
Noting that $\gamma f_{\text{dia}}$ decrease w.r.t. $h_0$, we further use the upper bound of $h_0$ in \eqref{acurate} to bound it as
\begin{align*}
 \gamma f_{\text{dia}}\geq
 &\left(-2\gamma^3-7\gamma^2+\frac{23\gamma}{2}-1\right)\left(1-\frac{1}{2\gamma}+\frac{3}{8\gamma^2}
 -\frac{3}{8\gamma^3}+\frac{63}{128\gamma^4}+\frac{7}{8\gamma^5}
 \right)^2\\
 &+\left(-2\gamma^2-14\gamma+23-\frac{2}{\gamma}\right)\left(1-\frac{1}{2\gamma}+\frac{3}{8\gamma^2}
 -\frac{3}{8\gamma^3}+\frac{63}{128\gamma^4}+\frac{7}{8\gamma^5}
 \right)\\
 &+2\gamma^3-7\gamma^2-\frac{7\gamma}{2}+1=13-\frac{5}{8\gamma}+\ldots>0.
\end{align*}
%
%
%
\section{Auxiliary results for the case $u_0 > 0$}\label{app:aux_res_u_pos}
The following lemme is important as it provides information about the term $u - sc^2$ which is needed throughout the other proofs as it appears in the numerator of $\eqref{systemODE}_1$.
\begin{lemma}\label{lem:u_eq_s}
	Let $u$ be a solution defined on $(0,1/c]$ with initial condition $u_0 \in (0,c)$.
	Then there exists a unique $\overline{s} \in (0,1/c)$ such that
	\[
	u(\overline{s}) = \overline{s} c^2.
	\]
\end{lemma}
\begin{proof}
   	Since $s = 0 < u_0 = u(0)$, we have for suitably small $s$,
   	\[
   	u(s) - sc^2 > 0\quad\Leftrightarrow\quad\frac{\dd u}{\dd s} > 0.
   	\]
   	Let $w(s) = s c^2$ and  assume such a point $\overline{s}$ does not exist. Clearly we have $0 < w(s) < u(s) < c$ and further know that
   	\[
   	\lim_{w(s)\to c}u(s) = \lim_{s\to  {1}/{c}}u(s) =c.
   	\]
   	From $\eqref{systemODE}_1$ and $\eqref{systemODE}_2$, we get
   	\begin{align*}
   		\frac{1}{c^2(us - 1)(e + p)}\frac{\dd p}{\dd s} &= \frac{(d-1)u}{g} = \frac{1}{(c^2 - u^2)(u - sc^2)}\frac{\dd u}{\dd s}\\
   		\Leftrightarrow\quad \frac{1}{c^2(e + p)}\dd p &= \frac{us - 1}{(c^2 - u^2)(u - sc^2)}\dd u.
   	\end{align*}
   	 	Due to $u/c < 1$, we have
   	\[
   	us - 1 < cs - 1 < cs - \frac{u}{c} = -\frac{1}{c}(u - sc^2).
   	\]
   	This implies $\dd p/\dd s < 0$ and thus we can further obtain
   	\[
   	\frac{1}{c^2(e + p)}\dd p = \frac{us - 1}{(c^2 - u^2)(u - sc^2)}\dd u < -\frac{1}{c(c^2 - u^2)}\dd u.
   	\]
   	Similar to \eqref{e_pddgammap}, it follows from \eqref{syngeg},  \eqref{dgamma1}, and \eqref{17} that
    \begin{equation}\label{ddgammap}
   \frac{1}{c^2(e + p)}\dd p=	\frac{\gamma
   		^2 \Phi'_i -1}{c^2\gamma(1+\gamma
   		\Phi_i)}\dd \gamma=	\frac{\gamma
   		^2 \left(h^2_i -1\right)+(2i+1)\gamma h_i-4}{c^2\gamma(\gamma
   		h_i+4)}\dd \gamma.
   \end{equation}
	The detailed asymptotic expansion \eqref{asymptotic_2} shows that this expression becomes a bounded integrand for large $\gamma$.   
   	For a small $\varepsilon>0$, we integrate from $p_0 = p(0)$ up to $\bar{p} = p(1/c - \varepsilon)$ and denote $\gamma_0 = \gamma(0)$. Using \eqref{ddgammap} gives
   	\begin{align*}
   		\int_{p_0}^{\bar{p}}\frac{1}{c^2(e + p)}\dd p > -\int_{0}^{p_0}\frac{1}{c^2(e + p)}\dd p = \int_{\gamma_0}^{\infty} \frac{\gamma
   		^2 \left(h^2_i -1\right)+(2i+1)\gamma h_i-4}{c^2\gamma(\gamma
   		h_i+4)}\dd \gamma
   	\end{align*}
   	which is a finite value. Further we have with $u_0 = u(0)$ and $\bar{u} = u(1/c - \varepsilon)$
   	\begin{align*}
   		\int_{u_0}^{\bar{u}}\frac{us - 1}{(c^2 - u^2)(u - sc^2)}\dd u < -\int_{u_0}^{\bar{u}}\frac{1}{c(c^2 - u^2)}\dd u = -\frac{1}{2c^2}\left.\ln\left(\frac{c + u}{c - u}\right)\right|^{\bar{u}}_{u_0}.
   	\end{align*}
   	This gives in total the following inequality
   	\begin{align*}
   		\int_{\gamma_0}^{\infty} \frac{\gamma
   		^2 \left(h^2_i -1\right)+(2i+1)\gamma h_i-4}{c^2\gamma(\gamma
   		h_i+4)}\dd \gamma < -\frac{1}{2c^2}\left.\ln\left(\frac{c + u}{c - u}\right)\right|^{\bar{u}}_{u_0},
   	\end{align*}
   	which is violated as $\varepsilon \to 0$ and thus $u \to c$ since then the RHS goes to minus infinity in contrast to the bounded LHS.
\end{proof}
In the following we want to give the proof of Lem. \ref{lem:A_neg}
\begin{proof}
   	Since $u(1/c) < \varphi_A(1/c) = c$,  we assume that there exists an $s^\ast > 1/c$ such that
   	\[
   	0 < u(s) < \varphi_A(s),\;\;s\in(1/c,s^\ast),\quad\text{and}\quad u(s^\ast) = \varphi_A(s^\ast).
   	\]
   	Clearly this implies $g(s) > 0$ for $s \in (1/c,s^\ast)$ and $g(s^\ast) = 0$. Hence we also have
   	\[
   	\frac{\dd u}{\dd s} = \frac{(d-1)\overbrace{u}^{>0}\overbrace{(c^2 - u^2)}^{>0}\overbrace{(u - sc^2)}^{< 0}}{\underbrace{g}_{>0}}<0.
   	\]
   	Thus we must have for $s^\ast$
   	\begin{align*}
   		\frac{\dd u}{\dd s}(s^\ast) - \frac{\dd \varphi_A}{\dd s}(s^\ast) > 0.
   	\end{align*}
   	First we determine the derivative of $\varphi_A$ which is given by
   	\begin{align*}
   		\frac{\dd \varphi_A}{\dd s} &= \dfrac{\left(\dfrac{c}{2}\dfrac{e_{pp}}{\sqrt{e_p}}p'(s) - c^2\right)(c\sqrt{e_p}s - 1)
   			- (c\sqrt{e_p} - sc^2)\left(\dfrac{c}{2}\dfrac{e_{pp}}{\sqrt{e_p}}p'(s)s + c\sqrt{e_p}\right)}{(c\sqrt{e_p}s - 1)^2}\notag\\
   		&= \dfrac{(d-1)c^3(s^2c^2 - 1)u(us - 1)e_{pp}(e + p) + 2c^2g\sqrt{e_p}(1 - e_p)}{2\sqrt{e_p}g(c\sqrt{e_p}s - 1)^2}.
   	\end{align*}
    A direct calculation gives
    \begin{align*}
        \frac{\dd u}{\dd s} - \frac{\dd \varphi_A}{\dd s} =& \frac{(d-1)u(c^2 - u^2)(u - sc^2)}{g}\\ 
        & -\dfrac{(d-1)c^3(s^2c^2 - 1)u(us - 1)e_{pp}(e + p) + 2c^2g\sqrt{e_p}(1 - e_p)}{2\sqrt{e_p}g(c\sqrt{e_p}s - 1)^2}\\
        =& \left(2\sqrt{e_p}g(c\sqrt{e_p}s - 1)^2\right)^{-1}\left[2\sqrt{e_p}(c\sqrt{e_p}s - 1)^2(d-1)u(c^2 - u^2)(u - sc^2)\right.\\
        &-\left. (d-1)c^3(s^2c^2 - 1)u(us - 1)e_{pp}(e + p) + 2c^2g\sqrt{e_p}(1 - e_p)\right]
    \end{align*}
    Clearly we have for the denominator
    \[
    2\sqrt{e_p}g(c\sqrt{e_p}s - 1)^2 > 0\quad\text{and}\quad \lim_{s\to s^\ast}2\sqrt{e_p}g(c\sqrt{e_p}s - 1)^2 = 0.
    \]
    For $s = s^\ast$, we in particular have due to \eqref{def:g_A} and \eqref{AB_u_phi_relation} that
    \[
    u(s^\ast) = \frac{c\sqrt{e_p} - s^\ast c^2}{c\sqrt{e_p}s^\ast - 1}\quad\text{and}\quad A = 0\quad\stackrel{\eqref{def:g_A}}{\Leftrightarrow}\quad u - sc^2 = c\sqrt{e_p}(us - 1)<0.
    \]
    Applying these relations to the nominator at $s = s^\ast$ gives
    \begin{align*}
        &\left[2\sqrt{e_p}(c\sqrt{e_p}s - 1)^2(d-1)u(c^2 - u^2)(u - sc^2)\right.\\
         &- \left.\left.(d-1)c^3(s^2c^2 - 1)u(us - 1)e_{pp}(e + p) + \underbrace{2c^2g\sqrt{e_p}(1 - e_p)}_{=0}\right]\right|_{s=s^\ast}\\
        =\; &c(d-1)u(us - 1)\left.\left[2e_p(c\sqrt{e_p}s - 1)^2\left(c^2 - \left(\frac{c\sqrt{e_p} - sc^2}{c\sqrt{e_p}s - 1}\right)^2\right) - c^2(s^2c^2 - 1)e_{pp}(e + p)\right]\right|_{s=s^\ast}\\
        =\; &c(d-1)u(us - 1)\left.\left[2c^2e_p(c\sqrt{e_p}s - 1)^2 - 2e_p\left(c\sqrt{e_p} - sc^2\right)^2 - c^2(s^2c^2 - 1)e_{pp}(e + p)\right]\right|_{s=s^\ast}\\
        =\; &c(d-1)u(us - 1)\left.\left[2c^2e_p(s^2c^2 - 1)(e_p - 1) - c^2(s^2c^2 - 1)e_{pp}(e + p)\right]\right|_{s=s^\ast}\\
        =\; &c^3(d-1)(s^2c^2 - 1)u(us - 1)\left.\left[\underbrace{2e_p(e_p - 1) - e_{pp}(e + p)}_{>0}\right]\right|_{s=s^\ast}\\
        <\; &0.
    \end{align*}
   	In summary, we hence yield
   	\[
   	0 < \frac{\dd u}{\dd s}(s^\ast) - \frac{\dd \varphi_A}{\dd s}(s^\ast) = \lim_{s\to s^\ast}\frac{\dd u}{\dd s}(s ) - \frac{\dd \varphi_A}{\dd s}(s) = - \infty,
   	\]
   	which is a contradiction and thus we have proven the claimed statement.
\end{proof}
In order to obtain the statment for \textbf{Case III} we argued using a continuity argument. This will be made more precise with the following to lemmas.
\begin{lemma}
	Let $u_0 > \overline{u}$, and suppose that for the initial datum $(u_0,p_0,\eta_0)$, \textbf{\textup{Case I}} applies.  
	Then there exists $\varepsilon > 0$ such that \textbf{\textup{Case I}} also applies for any initial datum
	\[
	(u,p,\eta)(0) \in (u_0-\varepsilon,u_0+\varepsilon) \times \{p_0\} \times \{\eta_0\}.
	\]
\end{lemma}  %
   \begin{proof}
       Let $(\tilde{u},\tilde{p},\tilde{\eta})(s)$ denotes the solution corresponding to the initial data $(u_0,p_0,\eta_0)$. By Lemma \ref{lem:case_1}, we have that there exists an $\tilde{s}^\ast$ such that
       \[
         \tilde{u}(\tilde{s}^\ast) = \varphi_A(\tilde{s}^\ast) = \frac{1}{\tilde{s}^\ast}\quad\text{and}\quad \tilde{p}(\tilde{s}^\ast) = 0.
       \]
       We denote the integral
       \[
         \mathcal{I} := \int_{0}^{1/\tilde{s}^\ast} \frac{1}{c^2 - u^2}\dd u > 0.
       \]
       There exists a suitably small $\delta > 0$ such that
       \begin{align*}
         0 < \int_{0}^{\delta} \frac{\sqrt{e_p}}{c(e + p)}\dd p < \frac{1}{4}\mathcal{I} , \qquad \text{and} \qquad   \frac{3}{4}\mathcal{I} < \int_{0}^{1/\tilde{s}^\ast-\delta} \frac{1}{c^2 - u^2}\dd u. 
       \end{align*}
       Since $\tilde{p}(s)$ is continuous on $[0,\tilde{s}^\ast]$, there exists a sufficiently small $\mu > 0$ such that
       \begin{align*}
         \tilde{p}(\tilde{s}^\ast - \mu) < \frac{\delta}{2}.
       \end{align*}
       Given a sufficiently small $\varepsilon > 0$, the unique solution $(u,p,\eta)(s)$ for initial data\\
       ${(u,p,\eta)(0) \in (u_0-\varepsilon,u_0+\varepsilon)\times\{p_0\}\times\{\eta_0\}}$ satisfies
       \begin{align}
           \left| p(\tilde{s}^\ast - \mu) - \tilde{p}(\tilde{s}^\ast - \mu)\right| < \frac{\delta}{4}\quad\text{and}\quad |u(\tilde{s}^\ast - \mu) - \tilde{u}(\tilde{s}^\ast - \mu)| < \frac{\delta}{4}.\label{ineq_dist_sol_v1}
       \end{align}
       For $\sigma > \tilde{s}^\ast - \mu$, according to \eqref{ineq:p_vac_int}  and due to the monotonicity of $p$ and $u$, we have
       \begin{align}
           \int_{p(\sigma)}^{p(\tilde{s}^\ast - \mu)} \frac{\sqrt{e_p}}{c(e + p)}\dd p > \int_{u(\sigma)}^{u(\tilde{s}^\ast-\mu)} \frac{1}{c^2 - u^2}\dd u.\label{ineq:int_p_u_case3}
       \end{align}
       Now, due to \eqref{ineq_dist_sol_v1}, we have
       \[
       p(\tilde{s}^\ast - \mu) < \tilde{p}(\tilde{s}^\ast - \mu) + \frac{\delta}{4} < \delta\quad\text{and}\quad u(\tilde{s}^\ast - \mu) > \tilde{u}(\tilde{s}^\ast - \mu) - \frac{\delta}{4} > \frac{1}{\tilde{s}^\ast} - \delta.
       \]
       Combining these relations with inequality \eqref{ineq:int_p_u_case3} gives
       \begin{align*}
           \frac{1}{4}\mathcal{I} &> \int_{p(\sigma)}^{\delta} \frac{\sqrt{e_p}}{c(e + p)}\dd p > \int_{p(\sigma)}^{p(\tilde{s}^\ast - \mu)} \frac{\sqrt{e_p}}{c(e + p)}\dd p \stackrel{\eqref{ineq:int_p_u_case3}}{>} \int_{u(\sigma)}^{u(\tilde{s}^\ast-\mu)} \frac{1}{c^2 - u^2}\dd u\notag\\
           &> \int_{u(\sigma)}^{\frac{1}{\tilde{s}^\ast} - \delta} \frac{1}{c^2 - u^2}\dd u. 
       \end{align*}
       Since $\sigma > \tilde{s}^\ast - \mu$ and $u(\sigma)$ decreasing, there must exists an $s^\ast > \tilde{s}^\ast - \mu$ such that
       \[
         1/s > \varphi_A(s) > u(s)\quad\text{for}\quad \frac{1}{c} < s < s^\ast\quad\text{and}\quad u(s^\ast) = \varphi_A(s^\ast) = \frac{1}{s^\ast} > 0.
       \]
       Otherwise $u(\sigma)$ would decrease further due to $1/\sigma > \varphi_A(\sigma) > u(\sigma)$ and then there would exist a $\sigma$ such that
       \[
         \int_{u(\sigma)}^{\frac{1}{\tilde{s}^\ast} - \delta} \frac{1}{c^2 - u^2}\dd u > \frac{1}{2}\mathcal{I}.
       \]
   \end{proof}
   \begin{lemma}
	Let $u_0$ be sufficiently small, and suppose that for the initial datum $(u_0,p_0,\eta_0)$, \textup{\textbf{Case II}} applies.  
	Then there exists $\varepsilon > 0$ such that \textup{\textbf{Case II}} also applies for any initial datum
	\[
	(u,p,\eta)(0) \in (u_0-\varepsilon,u_0+\varepsilon) \times \{p_0\} \times \{\eta_0\}.
	\]
\end{lemma}
   \begin{proof}
       Let $(\tilde{u},\tilde{p},\tilde{\eta})(s)$ denotes the solution corresponding to the initial datum $(u_0,p_0,\eta_0)$. By Lemma \ref{lem:case_2}, we have that there exists a point $\tilde{s}^\ast$ with $1/c < \tilde{s}^\ast$ such that
       \[
         \tilde{u}(\tilde{s}^\ast) = 0\quad\text{and}\quad \tilde{p}(\tilde{s}^\ast) = p^\ast > 0.
       \]
       We denote
       \[
         \mathcal{I} := \int_{0}^{p^\ast/2} \frac{1}{c(e + p)}\dd p > 0.
       \]
       There exists a suitably small $\delta > 0$ such that
       \begin{align*}
         0 < \int_{0}^{\delta} \frac{1}{c^2 - u^2}\dd u < \mathcal{I}, 
       \end{align*}
       and also a sufficiently small $\mu > 0$, due to the  continuity of $\tilde{u}$, such that
       \begin{align}
         \tilde{u}(\tilde{s}^\ast - \mu) < \frac{\delta}{4}.\label{ineq:u_small_case3}
       \end{align}
       Given a sufficiently small $\varepsilon > 0$, the unique solution $(u,p,\eta)(s)$ for the initial datum $(u,p,\eta)(0) \in (u_0-\varepsilon,u_0+\varepsilon)\times\{p_0\}\times\{\eta_0\}$ satisfies
       \begin{align}
           |p(\tilde{s}^\ast - \mu) - \tilde{p}(\tilde{s}^\ast - \mu)| < \frac{p^\ast}{4}\quad\text{and}\quad |u(\tilde{s}^\ast - \mu) - \tilde{u}(\tilde{s}^\ast - \mu)| < \frac{\delta}{4}.\label{ineq_dist_sol_v2}
       \end{align}
       As in \eqref{ineq:int_p_u_case2}, we have for $\sigma > \tilde{s}^\ast - \mu$
       \begin{align}
           \int_{p(\sigma)}^{p(\tilde{s}^\ast - \mu)} \frac{1}{c(e + p)}\dd p < \int_{u(\sigma)}^{u(\tilde{s}^\ast-\mu)} \frac{1}{c^2 - u^2}\dd u.\label{ineq:int_p_u_case3_v2}
       \end{align}
       Now, due to \eqref{ineq_dist_sol_v2} we have
       \[
         p(\tilde{s}^\ast - \mu) \geq \tilde{p}(\tilde{s}^\ast - \mu) - \frac{p^\ast}{4} > \tilde{p}(\tilde{s}^\ast) - \frac{p^\ast}{4} > \frac{p^\ast}{2}
       \]
       and additionally, by \eqref{ineq:u_small_case3},
       \[
       u(\tilde{s}^\ast - \mu) < \tilde{u}(\tilde{s}^\ast - \mu) + \frac{\delta}{4} < \frac{\delta}{2}.
       \]
       Combining these relations with inequality \eqref{ineq:int_p_u_case3_v2} gives
       \begin{align*}
           \int_{p(\sigma)}^{p^\ast/2} \frac{1}{c(e + p)}\dd p < \int_{p(\sigma)}^{p(\tilde{s}^\ast - \mu)} \frac{1}{c(e + p)}\dd p < \int_{u(\sigma)}^{u(\tilde{s}^\ast-\mu)} \frac{1}{c^2 - u^2}\dd u < \int_0^{\delta/2} \frac{1}{c^2 - u^2}\dd u < \mathcal{I}. 
       \end{align*}
       Since $p(\sigma)$ decreases, there exists a costant $p_m > 0$ such that $p(\sigma)>p_m>0$ as $\sigma > \tilde{s}^\ast - \mu$ since otherwise the left hand side integral would tend to $\mathcal{I}$ contradicting the inequality. Therefore we can conclude that there exists an $s^\ast > \tilde{s}^\ast - \mu$ such that
       \[
         \varphi_A(s^\ast) = \frac{c\sqrt{e_p} - s^\ast c^2}{c\sqrt{e_p}s^\ast - 1} = 0\quad\text{with}\quad s^\ast = \frac{1}{c}\sqrt{e_p(p(s^\ast),\eta_0)}.
       \]
       Since $0 < u(s) < \varphi_A(s)$ for $s \in (1/c,s^\ast)$, we have $u(s^\ast) = 0$.
   \end{proof}

\section{Auxiliary estimates for the uniqueness of the shock solution}\label{app:shock_ineq}
In the following we collect some of the technical details of the proof of Lem.\ \ref{lem:unique_shock}. First we study the function $\varphi_A(s) > u(s)$ and we have
\begin{align*}
	\varphi_A(s) = \frac{c\sqrt{e_p} - sc^2}{c\sqrt{e_p}s - 1}\quad\Leftrightarrow\quad\tilde{\varphi}_A(s) = \mathcal{T}_\sigma^-(\varphi_A,s) = -\frac{c}{\sqrt{e_p}}\quad\Leftrightarrow\quad v(\tilde{\varphi}_A(s)) = -\frac{c}{\sqrt{e_p - 1}}.
\end{align*}
Due to Lem. \ref{lem:blowup}, \eqref{deriv:trafo_m} and \eqref{deriv:u_v} we have
\begin{align*}
	\varphi_A(s) > u(s)\quad\Rightarrow\quad 0 > \tilde{\varphi}_A(s) > \tilde{u}(s)\quad\Rightarrow\quad 0 > v(\tilde{\varphi}_A(s)) > v(\tilde{u}(s)).
\end{align*}
Furthermore we hence have
\begin{align}
	\tilde{\varphi}_A(s)^2 = \frac{c^2}{e_p} < \tilde{u}(s)^2\quad\text{and}\quad v(\tilde{\varphi}_A(s))^2 = \frac{c^2}{e_p - 1} < v(\tilde{u}(s))^2.\label{phi_u_square_mono}
\end{align}
From \eqref{ep_gamma_phi_rel} we yield
\begin{align}
	e_p = \underbrace{\frac{\gamma^2\Phi_i'(\gamma)}{\gamma^2\Phi_i'(\gamma) - 1}}_{0<\dots<1}\chi_i(\gamma) < \chi_i(\gamma).\label{ep_chi_ineq}
\end{align}
Combining these results we get
\begin{align}
	\chi_i(\gamma)v(\tilde{u})^2 - c^2 &\stackrel{\eqref{phi_u_square_mono}}{>} \chi_i(\gamma)v(\tilde{\varphi}_A)^2 - c^2 \stackrel{\eqref{ep_chi_ineq}}{>} \frac{c^2 e_p}{e_p - 1} - c^2 = \frac{c^2}{e_p - 1} > 0.\label{duG2_ineq_help}
\end{align}
Let us now prove that \eqref{jac_det} holds when $s=s^\ast$ and hence we explicitly work with $u_\delta(s^\ast) = 0$ which gives
\begin{align*}
	\tilde{u}_\delta(s^\ast) = \mathcal{T}_\sigma^-(u_\delta(s^\ast),s^\ast) = -\frac{1}{s^\ast}\quad\text{and}\quad v_\delta := v(\tilde{u}_\delta(s^\ast)) = -\frac{1}{s^\ast}\tilde\Gamma_\delta(\tilde{u}_\delta(s^\ast)) = -\frac{c}{\sqrt{c^2(s^\ast)^2 - 1}}.
\end{align*}
Using \eqref{deriv:u_G1} - \eqref{deriv:gamma_G2} the determinant is given by
\begin{align}
    \det(\mathbf{D}_\delta G)(\tilde{u}_\delta,\gamma_\delta)
    &= \dfrac{1}{c^2}\dfrac{\chi_i(\gamma_\delta)}{\gamma_\delta}v_\delta\tilde\Gamma_\delta^2\left(\dfrac{\gamma_\delta q_i(\gamma_\delta) - \chi_i(\gamma_\delta)}{\gamma_\delta^2}v_\delta - \dfrac{c^2}{\gamma_\delta^2 v_\delta}\right)\nonumber\\
    &-\dfrac{\gamma_\delta q_i(\gamma_\delta) - \chi_i(\gamma_\delta)}{\gamma_\delta^2}\tilde\Gamma_\delta
    \dfrac{\tilde\Gamma_\delta^3}{\gamma_\delta v_\delta^2}\left(\chi_i(\gamma_\delta)v^2_\delta - c^2\right)  \label{jac_det0}\\
    &= \frac{\tilde\Gamma_\delta^2}{\gamma_\delta^3}\left\{\left[\gamma_\delta q_i(\gamma_\delta) - \chi_i(\gamma_\delta)\right]\left[\frac{\chi_i(\gamma_\delta)v_\delta^2}{c^2}
        -\dfrac{\tilde\Gamma_\delta^2}{ v_\delta^2}\left(\chi_i(\gamma_\delta)v^2_\delta - c^2\right)\right]-\chi_i(\gamma_\delta)\right\}.\nonumber
\end{align}
We further have from \eqref{ep_gamma_phi_rel}, \eqref{def:chi}, and \eqref{deriv:gamma_chi} that
\begin{align*}
	q_{i}(\gamma_\delta)= \gamma_\delta\Phi'_i+\Phi_i,\quad \chi_i(\gamma_\delta)=\gamma_\delta\Phi_i+1,\quad e_{p}(\gamma_\delta)=\frac{\gamma_\delta^2\Phi'_i}{\gamma_\delta^2\Phi'_i-1}\chi_i(\gamma_\delta).
\end{align*}
Therefore, when $s=s^\ast$, we can further estimate \eqref{jac_det0} as
\begin{align*}
    &\det(\mathbf{D}_\delta G)(\tilde{u}_\delta,\gamma_\delta) \notag\\
    &= \frac{\tilde\Gamma_\delta^2}{\gamma_\delta^3}\left\{\left[\gamma_\delta \left(\gamma_\delta\Phi'_i+\Phi_i\right) -\gamma_\delta\Phi_i-1 \right]\left[\frac{\chi_i(\gamma_\delta)}{(cs^\ast)^2-1}
    -\frac{(cs^\ast)^2\left(\chi_i(\gamma_\delta) - (cs^\ast)^2+1\right)}{(cs^\ast)^2-1}\right]-\chi_i(\gamma_\delta)\right\}\nonumber\\
    &= \frac{\tilde\Gamma_\delta^2}{\gamma_\delta^3}\left\{\left[\gamma_\delta^2 \Phi'_i-1 \right]\left[-\chi_i(\gamma_\delta)+(cs^\ast)^2\right]-\chi_i(\gamma_\delta)\right\}.
\end{align*}
Note that for $\sigma = 1/s^\ast$ with $u_\delta(s^\ast) = 0$ \eqref{Lax_cond} implies
\begin{align*}
	(s^\ast)^2 > \frac{e_{p}(\gamma_\delta)}{c^2}\quad\Leftrightarrow\quad (cs^\ast)^2 > e_{p}(\gamma_\delta) = \frac{\gamma_\delta^2\Phi'_i}{\gamma_\delta^2\Phi'_i-1}\chi_i(\gamma_\delta).
\end{align*} 
Thus we can finally deduce
\begin{align*}
    \det(\mathbf{D}_\delta G)(\tilde{u}_\delta,\gamma_\delta)  < \frac{\tilde\Gamma_\delta^2}{\gamma_\delta^3}\left\{\left[\gamma_\delta^2 \Phi'_i-1 \right]\frac{\chi_i(\gamma_\delta)}{\gamma_\delta^2 \Phi'_i-1 }-\chi_i(\gamma_\delta)\right\}=0
\end{align*}
for $\sigma = 1/s^\ast$. We now want to prove $\Psi_I + \Psi_{II} < 0$ and therefore we recall \eqref{def:psi_I} -- \eqref{def:psi_II_2} 
%
%
\begin{align*}
	\Psi_I &:= -\left[\partial_{\tilde{u}_\delta}G^{(1)}\partial_{\tilde{u}}G^{(2)} - \partial_{\tilde{u}_\delta}G^{(2)}\partial_{\tilde{u}}G^{(1)}\right]\frac{(c^2s^2 - 1)\left(c^2 - u^2\right)}{c^2s - u}\notag\\
	&\hphantom{:}= \psi_{I,1} + \psi_{I,2},\\
	\psi_{I,1} &:= \frac{(c^2s^2 - 1)\left(c^2 - u^2\right)}{c^2s - u}\dfrac{1}{c^2}\dfrac{\chi_i(\gamma_\delta)}{\gamma_\delta}v_\delta\tilde\Gamma_\delta^2\dfrac{\tilde\Gamma^3}{\gamma v^2}\left(\chi_i(\gamma)v^2 - c^2\right),\\
	\psi_{I,2} &:= -\frac{(c^2s^2 - 1)\left(c^2 - u^2\right)}{c^2s - u}\dfrac{\tilde\Gamma_\delta^3}{\gamma_\delta v_\delta^2}\left(\chi_i(\gamma_\delta)v^2_\delta - c^2\right)\dfrac{1}{c^2}\dfrac{\chi_i(\gamma)}{\gamma}v\tilde\Gamma^2,\\
	\Psi_{II} &:= \left[\partial_{\tilde{u}_\delta}G^{(1)}\partial_{\gamma}G^{(2)} - \partial_{\tilde{u}_\delta}G^{(2)}\partial_{\gamma}G^{(1)}\right]\frac{\gamma}{\gamma^2\Phi_i'(\gamma) - 1}(us - 1)\chi_i(\gamma)\notag\\
	&\hphantom{:}= \psi_{II,1} + \psi_{II,2},\\
	\psi_{II,1} &:= -\frac{\gamma(us - 1)\chi_i(\gamma)}{\gamma q_i(\gamma) - \chi_i(\gamma)}\dfrac{1}{c^2}\dfrac{\chi_i(\gamma_\delta)}{\gamma_\delta}v_\delta\tilde\Gamma_\delta^2\left(\dfrac{\gamma q_i(\gamma) - \chi_i(\gamma)}{\gamma^2}v - \dfrac{c^2}{\gamma^2 v}\right),\\
	\psi_{II,2} &:= \frac{(us - 1)\chi_i(\gamma)\tilde\Gamma}{\gamma}\dfrac{\tilde\Gamma_\delta^3}{\gamma_\delta v_\delta^2}\left(\chi_i(\gamma_\delta)v^2_\delta - c^2\right).
\end{align*}
%
%
Thus we get
\begin{align*}
	\psi_{I,1} + \psi_{II,1}
	&= \frac{(c^2s^2 - 1)\left(c^2 - u^2\right)}{c^2s - u}\dfrac{1}{c^2}\dfrac{\chi_i(\gamma_\delta)}{\gamma_\delta}v_\delta\tilde\Gamma_\delta^2\dfrac{\tilde\Gamma^3}{\gamma v^2}\left(\chi_i(\gamma)v^2 - c^2\right)\\
	&- \frac{\gamma(us - 1)\chi_i(\gamma)}{\gamma q_i(\gamma) - \chi_i(\gamma)}\dfrac{1}{c^2}\dfrac{\chi_i(\gamma_\delta)}{\gamma_\delta}v_\delta\tilde\Gamma_\delta^2\left(\dfrac{\gamma q_i(\gamma) - \chi_i(\gamma)}{\gamma^2}v - \dfrac{c^2}{\gamma^2 v}\right)\\
	&= \frac{1}{\gamma q_i(\gamma) - \chi_i(\gamma)}\frac{\chi_i(\gamma_\delta)v_\delta\tilde\Gamma_\delta^2}{c^2\gamma_\delta}\left[\frac{(c^2s^2 - 1)\left(c^2 - u^2\right)}{c^2s - u}\dfrac{\tilde\Gamma^3}{\gamma v^2}\left(\chi_i(\gamma)v^2 - c^2\right)\left(\gamma q_i(\gamma) - \chi_i(\gamma)\right)\right.\\
	&- \left.\gamma(us - 1)\chi_i(\gamma)\left(\dfrac{\gamma q_i(\gamma) - \chi_i(\gamma)}{\gamma^2}v - \dfrac{c^2}{\gamma^2 v}\right)\right].
\end{align*}
\newpage
Observe that the following terms can be simplified by algebraic manipulations using \eqref{eq:velo_trafo_m} and \eqref{def:new_quantities}
\begin{align*}
	&\hphantom{=} \frac{(c^2s^2 - 1)\left(c^2 - u^2\right)}{c^2s - u}\dfrac{\tilde\Gamma^3}{\gamma v^2}\chi_i(\gamma)v^2\left(\gamma q_i(\gamma) - \chi_i(\gamma)\right)
	- \gamma(us - 1)\chi_i(\gamma)\dfrac{\gamma q_i(\gamma) - \chi_i(\gamma)}{\gamma^2}v\\
	&= (\gamma q_i(\gamma) - \chi_i(\gamma))\tilde\Gamma\frac{\chi_i(\gamma)}{\gamma}\left[\frac{(c^2s^2 - 1)\left(c^2 - u^2\right)}{c^2s - u}\tilde\Gamma^2
	- (us - 1)\tilde{u}\right]\\
	&= (\gamma q_i(\gamma) - \chi_i(\gamma))\tilde\Gamma\frac{\chi_i(\gamma)}{\gamma}\left[\frac{(c^2s^2 - 1)\left(c^2 - u^2\right)}{c^2s - u}\dfrac{c^2}{c^2 - c^4\dfrac{(us - 1)^2}{(c^2s - u)^2}}
	- (us - 1)c^2\dfrac{us - 1}{c^2s - u}\right]\\
	&= (\gamma q_i(\gamma) - \chi_i(\gamma))\tilde\Gamma\frac{\chi_i(\gamma)}{\gamma}\frac{(c^2 - u^2)(c^2s^2 - 1)}{c^2s - u}.
\end{align*}
Altogether we hence yield
\begin{align*}
	\psi_{I,1} + \psi_{II,1} &= \frac{1}{\gamma q_i(\gamma) - \chi_i(\gamma)}\frac{\chi_i(\gamma_\delta)v_\delta\tilde\Gamma_\delta^2}{c^2\gamma_\delta}\left[(\gamma q_i(\gamma) - \chi_i(\gamma))\tilde\Gamma\frac{\chi_i(\gamma)}{\gamma}\frac{(c^2 - u^2)(c^2s^2 - 1)}{c^2s - u}\right.\\
	&\left.- \frac{c^2(c^2s^2 - 1)\left(c^2 - u^2\right)}{c^2s - u}\dfrac{\tilde\Gamma^3}{\gamma v^2}\left(\gamma q_i(\gamma) - \chi_i(\gamma)\right)
	+ (us - 1)\chi_i(\gamma)\dfrac{c^2}{\gamma v}\right]\\
	&= \left[\frac{(c^2 - u^2)(c^2s^2 - 1)}{c^2s - u}\frac{\tilde\Gamma}{\gamma}(\gamma q_i(\gamma) - \chi_i(\gamma))\left(\chi_i(\gamma) - c^2\frac{\tilde\Gamma^2}{v^2}\right)\right.\\
	&\left.+ (us - 1)\chi_i(\gamma)\frac{c^2}{\gamma v}\right]\frac{1}{\gamma q_i(\gamma) - \chi_i(\gamma)}\frac{\chi_i(\gamma_\delta)v_\delta\tilde\Gamma_\delta^2}{c^2\gamma_\delta}\\
	&= \left[\frac{(c^2 - u^2)(c^2s^2 - 1)}{c^2s - u}\frac{\tilde\Gamma}{\gamma}(\gamma q_i(\gamma) - \chi_i(\gamma))\left(\chi_i(\gamma) - c^2\frac{(c^2s - u)^2}{c^4(us - 1)^2}\right)\right.\\
	&\left.+ (us - 1)\chi_i(\gamma)\frac{c^2}{\gamma \tilde\Gamma}\frac{c^2s - u}{c^2(us - 1)}\right]\frac{1}{\gamma q_i(\gamma) - \chi_i(\gamma)}\frac{\chi_i(\gamma_\delta)v_\delta\tilde\Gamma_\delta^2}{c^2\gamma_\delta}\\
	&= \frac{1}{\gamma q_i(\gamma) - \chi_i(\gamma)}\frac{\chi_i(\gamma_\delta)v_\delta\tilde\Gamma_\delta^2}{c^2\gamma_\delta}\frac{1}{c^2\gamma\tilde\Gamma(c^2s^2 - 1)(us - 1)^2}\\
	&\left[(c^2 - u^2)(c^2s^2 - 1)\tilde\Gamma^2(\gamma q_i(\gamma) - \chi_i(\gamma))\left(c^2(us - 1)^2\chi_i(\gamma) - (c^2s - u)^2\right)\right.\\
	&+ \left.c^2(us - 1)^2(c^2s - u)^2\chi_i(\gamma)\right]
\end{align*}
We continue to reformulate the expression within the brackets and yield
\begin{align*}
	&\hphantom{=}\frac{c^4\gamma\gamma_\delta\tilde\Gamma(c^2s - 1)(us - 1)^2(\gamma q_i(\gamma) - \chi_i(\gamma))}{\chi_i(\gamma_\delta)v_\delta\tilde\Gamma_\delta^2}(\psi_{I,1} + \psi_{II,1})\\
	&= (c^2 - u^2)(c^2s^2 - 1)\tilde\Gamma^2(\gamma q_i(\gamma) - \chi_i(\gamma))\left(c^2(us - 1)^2\chi_i(\gamma) - (c^2s - u)^2\right)\\
	&+ c^2(us - 1)^2(c^2s - u)^2\chi_i(\gamma)\\
	&= (c^2s - u)^2(\gamma q_i(\gamma) - \chi_i(\gamma))\left(c^2(us - 1)^2\chi_i(\gamma) - (c^2s - u)^2\right)\\
	&+ c^2(us - 1)^2(c^2s - u)^2\chi_i(\gamma)\\
	&= c^2(c^2s - u)^2(us - 1)^2\chi_i(\gamma)[\gamma q_i(\gamma) - \chi_i(\gamma) + 1] - (c^2s - u)^4(\gamma q_i(\gamma) - \chi_i(\gamma))\\
	&= \left(c^2s - u\right)^4\left[\left(\frac{\tilde{u}^2}{c^2}\chi_i(\gamma) - 1\right)(\gamma q_i(\gamma) - \chi_i(\gamma)) + \frac{\tilde{u}^2}{c^2}\chi_i(\gamma)\right].
\end{align*}
Now we apply \eqref{eq:deriv_chi_phi_rel} to get
\begin{align*}
	&\hphantom{=}c^2\gamma\tilde\Gamma(c^2s^2 - 1)(us - 1)^2(\psi_{I,1} + \psi_{II,1})\\
	&= \left(c^2s - u\right)^4\left[\left(\frac{\tilde{u}^2}{c^2}\chi_i(\gamma) - 1\right)(\gamma q_i(\gamma) - \chi_i(\gamma)) + \frac{\tilde{u}^2}{c^2}\chi_i(\gamma)\right]\\
	&= \left(c^2s - u\right)^4\left[\left(\frac{\tilde{u}^2}{c^2}\chi_i(\gamma) - 1\right)\left(\gamma^2\Phi_i'(\gamma) - 1\right) + \frac{\tilde{u}^2}{c^2}\chi_i(\gamma)\right]\\
	&= \left(c^2s - u\right)^4\left[\left(\frac{\tilde{u}^2}{c^2}\chi_i(\gamma) - 1\right)\gamma^2\Phi_i'(\gamma) + 1\right]\\
	&= \left(c^2s - u\right)^4\left[\frac{\tilde{u}^2}{c^2}\frac{\gamma^2\Phi_i'(\gamma) - 1}{\gamma^2\Phi_i'(\gamma)}e_p(\gamma)\gamma^2\Phi_i'(\gamma) - \gamma^2\Phi_i'(\gamma) + 1\right]\\
	&= \left(c^2s - u\right)^4\left[\frac{1}{e_p(\gamma)}\left(\gamma^2\Phi_i'(\gamma) - 1\right)e_p(\gamma) -  \left(\gamma^2\Phi_i'(\gamma) - 1\right)\right]\\
	&= 0.
\end{align*}
For the last inequality we have used $\gamma^2\Phi_i'(\gamma) - 1 < 0$ together with \eqref{phi_u_square_mono}.
For $\psi_{I,2} + \psi_{II,2}$, we use \eqref{eq:velo_trafo_m} and \eqref{def:new_quantities} to obtain
\begin{align*}
	\psi_{I,2} + \psi_{II,2} &= -\frac{(c^2s^2 - 1)\left(c^2 - u^2\right)}{c^2s - u}\dfrac{\tilde\Gamma_\delta^3}{\gamma_\delta v_\delta^2}\left(\chi_i(\gamma_\delta)v^2_\delta - c^2\right)\dfrac{1}{c^2}\dfrac{\chi_i(\gamma)}{\gamma}v\tilde\Gamma^2\\
	&+ \frac{(us - 1)\chi_i(\gamma)\tilde\Gamma}{\gamma}\dfrac{\tilde\Gamma_\delta^3}{\gamma_\delta v_\delta^2}\left(\chi_i(\gamma_\delta)v^2_\delta - c^2\right)\\
	&= \left(\chi_i(\gamma_\delta)v^2_\delta - c^2\right)\dfrac{\tilde\Gamma_\delta^3}{\gamma_\delta v_\delta^2}\frac{\chi_i(\gamma)\tilde\Gamma}{\gamma}\left[us - 1 -\frac{(c^2s^2 - 1)\left(c^2 - u^2\right)}{c^2s - u}\frac{v\tilde\Gamma}{c^2}\right]\\
	&= \left(\chi_i(\gamma_\delta)v^2_\delta - c^2\right)\dfrac{\tilde\Gamma_\delta^3}{\gamma_\delta v_\delta^2}\frac{\chi_i(\gamma)\tilde\Gamma}{\gamma}\left[us - 1 -\frac{(c^2s^2 - 1)\left(c^2 - u^2\right)}{c^2s - u}\frac{v\tilde\Gamma}{c^2} \frac{(c^2s - u)(us - 1)}{(c^2s^2 - 1)\left(c^2 - u^2\right)}\right]\\
	&= c^2(\chi_i(\gamma_\delta)v^2_\delta - c^2)\frac{us - 1}{(c^2 - \tilde{u}^2)(c^2s - u)^2}\dfrac{\tilde\Gamma_\delta^3}{\gamma_\delta v_\delta^2}\frac{\chi_i(\gamma)\tilde\Gamma}{\gamma}\cdot\dots\\
	&\dots\left[(c^2s - u)^2 - c^2(us - 1)^2 - (c^2s^2 - 1)\left(c^2 - u^2\right)\right]\\
	&= 0.
\end{align*}
In summary we hence have
\begin{align}
	\psi_{I,1} + \psi_{II,1} < 0\quad\text{and}\quad\psi_{I,2} + \psi_{II,2} = 0.\label{psi_rel_fin}
\end{align}
%
\section*{Data availability statement}
Data will be made available on reasonable request.
\section*{Acknowledgments}
The work of   T.~R. was   carried out in the framework of the activities of the Italian National Group for Mathematical Physics [Gruppo Nazionale per la Fisica Matematica (GNFM/INdAM)].

\noindent F.T. gratefully acknowledges the Deutsche Forschungsgemeinschaft (DFG, German
Research Foundation) for the financial support through 525939417/
SPP 2410 Hyperbolic Balance Laws in Fluid Mechanics: Complexity, Scales, Randomness (CoScaRa) and thanks M.~Kunik for introducing him to the interesting topic of relativistic Euler equations and fruitful discussions. 

\noindent The research of Q.  Xiao was supported by the National Natural Science Foundation of China (Project No.~12271506) and the National Key Research and Development Program of China (Project No. 2020YFA0714200). 

\phantomsection
\bibliographystyle{abbrv}

\end{document}